\documentclass[a4paper,11pt]{elsarticle}
\usepackage[utf8]{inputenc}

\usepackage{bm}
\usepackage{amsmath,amsthm,amssymb}
\usepackage{mathrsfs}
\usepackage{graphicx}
\usepackage{epsfig, setspace}
\usepackage{url}
\usepackage{graphicx, transparent, color}
\usepackage[ruled]{algorithm2e}
\usepackage{xcolor}
\usepackage{natbib}
\usepackage[titletoc]{appendix}

\DeclareSymbolFont{rsfs}{U}{rsfs}{m}{n}
\DeclareSymbolFontAlphabet{\mathcal}{rsfs}
\makeatletter
\def\ps@pprintTitle{%
	\let\@oddhead\@empty
	\let\@evenhead\@empty
	\let\@oddfoot\@empty
	\let\@evenfoot\@oddfoot
}
\makeatother

\usepackage[margin=2cm]{geometry}
\usepackage{tikz}
\usetikzlibrary{shapes,calc}
\usepackage{verbatim}
\usepackage{listings}
\usepackage{textcomp}
\usepackage[normalem]{ulem}
\usepackage{todonotes}

\usepackage{amsfonts}
\usepackage{lmodern}
\usepackage{euscript}
\usepackage{oldgerm}

\numberwithin{equation}{section}

\newtheorem{theo}{Theorem}[section]
\newtheorem{lemm}[theo]{Lemma}

\newtheorem{coro}[theo]{Corollary}
\newtheorem{prop}[theo]{Proposition}
\newtheorem{rema}{Remark}[section]
\numberwithin{equation}{section}

\allowdisplaybreaks[4]

\def\p{\partial}
\def\wx{\langle x_1\rangle}

\def\var{\varepsilon}
\def\gm{\gamma}
\def\e{\mathcal{E}(t)}

\def\dtau{\mathcal{D}(\tau)}
\def\f{\frac}
\def\k{\kappa}
\def\es{\mathcal{E}_{sm}}
\def\eb{\mathcal{E}_{bd}}

\def\beq{\begin{equation}}
	\def\bal{\begin{aligned}}
		\def\dal{\end{aligned}}
	\def\deq{\end{equation}}
\def\beqq{\begin{equation*}}
	\def\deqq{\end{equation*}}

\begin{document}
	\begin{frontmatter}
	\title{Well-posedness and exponential stability of the inhomogeneous
		anisotropic incompressible  Navier-Stokes equation with far-field vacuum
		in two-dimensional whole space}
	\tnotetext[t1]{The authors are listed alphabetically}
	\date{}
	\author{Jincheng Gao} \ead{gaojch5@mail.sysu.edu.cn}
	\author{Lianyun Peng  \corref{cor1}}\ead{pengly8@mail2.sysu.edu.cn}
	\author{Zheng-an Yao}\ead{mcsyao@mail.sysu.edu.cn}
	\cortext[cor1]{Corresponding author}
	\address{
		School of Mathematics,\\ Sun Yat-Sen University, 
		510275, Guangzhou, P. R. China}
	

	\begin{abstract}
		In this paper, we investigate the well-posedness theory
	and exponential stability for the inhomogeneous incompressible Navier-Stokes
	equation with only horizontal dissipative structure.
	Due to the lack of the vertical dissipative term and appearance of vacuum,
	it is a highly challenging tricky problem for us to study the well-posedness,
	stability and large-time behavior problems in two-dimensional whole space.
	The local-in-time well-posedness theory is successfully established at first 
	because we develop some good estimates for the density and vorticity
	to control the nonlinear term.
	Finally, these good estimates of density and vorticity help us to establish the global-in-time well-posedness and exponential stability
	if the initial velocity is suitable small.

		\vspace*{5pt}
		
	\end{abstract}
\end{frontmatter}
	
	\tableofcontents

	\section{Introduction}
	In this paper, we are concerned with the initial problem of the inhomogeneous incompressible
	Navier-Stokes equations in $\mathbb{R}^2$ with only horizontal dissipation:
	\beq\label{o-ns}
	\left\{\begin{aligned}
		&\p_t \rho+ u \cdot \nabla \rho=0,\\
		&\rho \p_t u+\rho u \cdot \nabla u- \mu_1 \p_1^2 u +\nabla p=0,\\
		&{\rm div}u=0,\\
		&(\rho, u)|_{t=0}=(\rho_0, u_0), \\
		&\lim_{|x| \to +\infty}(\rho, u)=(0, 0),
	\end{aligned}\right.
	\deq
	where $x=(x_1,x_2)$ is the spatial coordinate, and the unknown functions
	$\rho, u \overset{\text{def}}{=} (u_1, u_2)$ and $p$ denote the density, velocity
	and pressure of the fluid, respectively. Here  $\p_1$  is the
	abbreviation of the partial derivative $\p_{x_1}$.
	The positive constant viscosity coefficient $\mu_1$ denotes
	the horizontal viscosity coefficient.
	Without loss of generality, we set $\mu_1=1$ in this paper just for the sake of simplicity.
	The equation \eqref{o-ns} comes from the classical inhomogeneous incompressible
	Navier-Stokes equation as follows
	\beq\label{1-ns}
	\left\{\begin{aligned}
		&\p_t \rho+ u \cdot \nabla \rho=0,\\
		&\rho \p_t u+\rho u \cdot \nabla u- \mu_1 \p_1^2 u - \mu_2 \p_2^2 u+\nabla p=0,\\
		&{\rm div}u=0,
	\end{aligned}\right.
	\deq
	where $\mu_2$ denotes the vertical viscosity coefficient.
	In the anisotropic flow, the vertical viscosity coefficient is
	smaller than the  horizontal one.
	Then, $\mu_2$ is smaller than $\mu_1$ in anisotropic Navier-Stokes equation.
	Thus, taking $\mu_2 \rightarrow 0^+$ in \eqref{1-ns}, one
	can obtain the system \eqref{o-ns} investigated in this paper.
	
	In this paper, we will investigate the local-in-time well-posedness theory
	of the system \eqref{o-ns} for the case  $\inf_x \rho_0(x)=0$.
	Let us define the energy norms
	\begin{equation}\label{norm1}
		\begin{aligned}
			X_{sm}(t)
			\overset{\text{def}}{=}
			\sum_{k=0}^3 \|\sqrt{\rho} \p_2^k u\|_{L^2}^2
			+\sum_{k=0}^2 \|\p_2^k \p_1 u\|_{L^2}^2
			+\sum_{k=0}^1\|\sqrt{\rho} \p_t \p_2^k u\|_{L^2}^2
			+\|\p_t \p_1 u\|_{L^2}^2,
		\end{aligned}
	\end{equation}
	\begin{equation}\label{norm2}
		\begin{aligned}
			X_{bd}(t)
			\overset{\text{def}}{=}
			\sum_{0\le |\alpha|\le 2}\|\p_x^\alpha \sqrt{\rho} \wx^{\gm+1}\|_{L^2}^2
			+\sum_{|\alpha|=3}\|\p_x^\alpha \sqrt{\rho} \wx^{\gm}\|_{L^2}^2
			+\sum_{1\le |\alpha| \le 2}\|\p_x^\alpha \ln \sqrt{\rho}\|_{L^2}^2,
		\end{aligned}
	\end{equation}
	and the dissipation norm
	\beqq
	\begin{aligned}
		{Y}(t)\overset{\text{def}}{=}
		&~\sum_{k=0}^3\|\p_1 \p_2^k u\|_{L^2}^2
		+ \sum_{k=0}^2 \|\sqrt{\rho}  \p_2^k  \p_t u\|_{L^2}^2
		+ \sum_{k=0}^1  \|\p_1 \p_2^k \p_t u\|_{L^2}^2\\
		&+\|\sqrt{\rho}  \p_t^2 u\|_{L^2}^2
		+\|\nabla^3 p\|_{L^2}^2 + \|\p_t \nabla  p\|_{L^2}^2,
	\end{aligned}
	\deqq
	where the weight is defined by $\wx \overset{\text{def}}{=}1+|x_1|$.
	Finally, we define the initial data norm as follows:
	\beqq
	\begin{aligned}
		{X}(\rho_0, u_0)\overset{\text{def}}{=}
		&~
		\sum_{k=0}^3\|\sqrt{\rho_0}  \p_2^k u_0\|_{L^2}^2
		+\sum_{k=0}^2 \|\p_1 \p_2^k u_0\|_{L^2}^2
		+\sum_{0\le |\alpha|\le 2}
		\|\p_x^\alpha \sqrt{\rho_0} \wx^{\gm+1}\|_{L^2}^2\\
		&+\sum_{|\alpha|=3}\|\p_x^\alpha \sqrt{\rho_0} \wx^{\gm}\|_{L^2}^2
		+\sum_{1\le |\alpha| \le 2}\|\p_x^\alpha \ln \sqrt{\rho_0}\|_{L^2}^2.
	\end{aligned}
	\deqq
	
	In this paper, we will investigate the local-in-time well-posedness theory,
	global-in-time well-posedness theory with some small initial data condition
	and the exponential stability for the anisotropic Navier-Stokes equation
	\eqref{o-ns} for the case  $\inf_x \rho_0(x)=0$.
	
	Now we state our first result concerning on the local-in-time
	well-posedness of Eq.\eqref{o-ns} as follows.
	
	\begin{theo}[Local-in-time well-posedness]\label{local-well}
		Let $\gamma \ge 2$ be an integer. Assume the initial data
		$(\rho_0(x), u_0(x))$ satisfies
		\beqq
		{X}(\rho_0, u_0)<+\infty,\quad
		\rho_0(x)>0 \quad \text{\rm in} \quad \mathbb{R}^2,
		\deqq
		and
		\beqq
		\rho_0 u_0 \cdot \nabla u_0
		-\p_1^2 u_0+\nabla p_0=\rho_0 g \quad \text{\rm in} \quad \mathbb{R}^2,
		\deqq
		for some $(\nabla p_0, g)\in L^2 \times H^1$.
		Then, there exist a positive time $T_0>0$ and a unique solution $(\rho, u, p)$
		to the Cauchy problem \eqref{o-ns} satisfying
		\beqq
		\rho(x, t)>0, \quad (x, t)\in \mathbb{R}^2\times [0, T_0],
		\deqq
		and
		\beq\label{main-estimate}
		\underset{t\in [0, T_0]}{\sup}(X_{sm}(t)+X_{bd}(t)+\|\nabla p(t)\|_{H^1}^2)
		+\int_0^{T_0}{Y}(\rho, u, p)(t)d t
		\le C_{\phi_0, u_0, p_0, \gamma},
		\deq
		where $C_{\rho_0, u_0, p_0, \gamma}$ is a
		constant only depending on the parameters
		$\rho_0, u_0, p_0$ and $\gamma$.
	\end{theo}

	\begin{rema}
		In order to ensure the regular condition $\nabla \ln \sqrt{\rho} \in H^1$,
		we require the density satisfying $\rho>0$.
		Thus, our regular solution does not contain local vacuum, but vacuum occurs in the far field.
	\end{rema}
	
	
	With the local-in-time result at hand,
	we are already to investigate the global-in-time well-posedness
	and exponential stability of Eq.\eqref{o-ns}.
	
	\begin{theo}[Global well-posedness and exponential stability]\label{global-well}
		Suppose the conditions in Theorem \ref{local-well} hold.
		Furthermore, assume that there exists some small positive constant $\varepsilon_0$
		depending only on $X_{bd}(0)$ such that if
		\beq\label{small}
		X_{sm}(0)\le \varepsilon_0,
		\deq
		then the Cauchy problem \eqref{o-ns} admits a unique global solution
		$(\rho, u, p)$. Moreover, it holds that
		\beq\label{uniform-estimate}
		X_{bd}(t)\le 2 X_{bd}(0),
		\deq
		and there exists some positive constant $\gamma$
		depending only on $X_{bd}(0)$ such that
		\beq\label{decay}
		X_{sm}(t)+\|\nabla p(t)\|_{H^1}^2 \le C e^{-\gamma t},
		\deq
		where $C$ depends only on $X_{bd}(0)$.
	\end{theo}
	
	\begin{rema}
		For the inhomogeneous incompressible Navier-Stokes equation
		even for the homogeneous case, the corresponding decay-in-time
		rates for the strong solutions are algebraic
		(cf.\cite{{Abidi-Gui-Zhang-2011},{Craig-Huang-Wang-2013},{Lu-Xu-Zhong-2017},
			{Chen-1991},{He-Xin-2001},{Kato1984},{Schonbek1995}}).
		Recently, He, Li and L\"{u} \cite{He-Li-Lv2021} firstly established
		the exponential decay-in-time for the inhomogeneous incompressible Navier-Stokes equation
		with vacuum in three-dimensional unbounded domain.
		This result is interesting and somewhat surprising.
		To the best of authors' knowledge, whether the solution in $\mathbb{R}^2$
		will decay exponentially or not remains an open question.
		Thus, our decay rate \eqref{decay} in Theorem \ref{global-well} provides a positive answer
		to this open question.
	\end{rema}
	
	\begin{rema}
		For the 2D Navier-Stokes equation with only one-directional dissipation,
		the stability and large-time behavior problem is not well-understood.
		The paper \cite{Dong-Wu-2021} solves this problem when the domain
		is $\mathbb{T}\times \mathbb{R}$ with $\mathbb{T}$ being a 1D periodic box.
		However, when the spatial domain is the whole space $\mathbb{R}^2$,
		this problem is still widely open.
		The decay rate \eqref{decay} in Theorem \ref{global-well} implies that
		the global solution of inhomogeneous incompressible Navier-Stokes equation
		with far-field vacuum will decay exponentially in two-dimensional whole space
		although the equation only enjoys one-directional dissipation.
		This decay rate \eqref{decay} is new and somewhat surprising.
	\end{rema}
	
	\begin{rema}
		For the inhomogeneous incompressible Navier-Stokes equation
		(even considering magnetic field effects) in two-dimensional whole space,
		one can establish the global well-posedness theory with large initial data
		(cf.\cite{{Lu-Shi-Zhong},{Lu-Xu-Zhong-2017}}).
		However, if the viscosity coefficient depends on density rather than constant one,
		we should require some small condition on the initial data
		to make sure strong solution existing globally in time
		(cf.\cite{{Huang-Wang2014},{Abidi-Zhang2015}}).
		Thus, we require small initial condition \eqref{small} on the initial velocity
		to establish the global well-posedness for the Eq.\eqref{o-ns}.
	\end{rema}
	
	Mathematically the model in \eqref{1-ns} is intermediate between the 2D incompressible Navier-Stokes equations with full dissipations
	and the 2D compressible Navier-Stokes equations with full dissipations.
	Besides, the system \eqref{1-ns}  is originated in the theory of geophysical flows as in oceans or rivers and governs the evolution of incompressible viscous flows with nonconstant density. At the same time, it includes conservation laws for energy, total momentum, total mass and so on. These quantities hold for all time, provided the solutions are sufficiently smooth. They are of great importance for studying the mathematical properties of the solutions.
	
	When the density is constant, the system \eqref{1-ns} is reduced to the following classical incompressible Navier-Stokes equation \eqref{cla-ns}, which has the good property called scaling invariance.
	\beq\label{cla-ns}
	\left\{\begin{aligned}
		& \p_t u+ u \cdot \nabla u-  \triangle u +\nabla p=0,\\
		&{\rm div}u=0.
	\end{aligned}\right.
	\deq
	To be precise, if $u(t,x)$ is a solution of \eqref{cla-ns}, then $u_\lambda  \overset{\text{def}}{=} \lambda u(\lambda^2 t, \lambda x )$ is also a solution of \eqref{cla-ns}.
	We first give a short survey for the study of the classical Navier-Stokes
	system \eqref{cla-ns}.
	In 1934, Leray \cite{Leray1934} proved the existence of global weak solutions for the system \eqref{cla-ns} if the initial data $u_0$ belongs to $L^2$ space.
	Although this solution is unique in dimension two,
	its uniqueness of dimension three or more remains unknown.
	And Kato \cite{Kato1984} proved the local existence of the unique solutions for the Cauchy problem in scaling-invariant space $L^p$.  As for the global well-posedness, it has been proven when the initial data is small in the scaling-invariant space $\dot{H}^{\f{d}{2}-1}(\mathbb{R}^d)$ by Fujita and Kato \cite{Fujita-Kato}, in scaling-invariant space $BMO^{-1}$ by Koch and Tataru \cite{Koch2001} and so on.
	For the decay estimate of solution for the classical incompressible Navier-Stokes system,
	one can refer to \cite{Schonbek1985,Schonbek1991,Planchon1998,Karch-Pilarczyk-Schonbek2017}.

	If one investigates the anisotropic flow, the governing equation will only enjoy
	partial dissipative structure.
	Without the vertical viscosity,
	the system was firstly studied by Chemin et al.\cite{Chemin2000} in the case of three-dimensional fluids.
	The authors proved the results of the local existence for large data and global existence for small initial data in anisotropic Sobolev spaces instead of the classical ones. And then Iftimie \cite{iftimie-2002} proved the uniqueness of the global solution.
	Similar to the classical Navier-Stokes equations, the system in three-dimensional whole space has a scaling invariance.
	Paicu \cite{Paicu-2005}  proved the global existence of a unique solution by considering the initial data in the scaling invariant Besov space $\mathcal{B}^{0,\frac12}$.
	Then Chemin and Zhang \cite{Chemin-Zhang-2007} introduced the scaling invariant Besov space $\mathcal{B}_4^{-\f12,\f12}$
	and proved the global existence results with high oscillatory initial data.
	We refer to \cite{Gallagher, Paicu-zhang-2019, Liu-Paicu-Zhang-2020} for recent literatures on the well-posedness for the system in the 3D case.
	In two-dimensional whole space, the global existence and regularity
	relies on the Yudovich approach designed for the 2D Euler equations \cite{Yudovich-1963}.
	There are few results about the stability and large-time behavior problem for the 2D Navier-Stokes equations with only partial dissipation, for the one-directional dissipation fails to control the nonlinear term by using the energy method.
	Thus, the stability and large-time behavior problem for the 2D Navier-Stokes equations
	with only one-directional dissipation is not well-understood.
	Recently, Dong et al.\cite{Dong-Wu-2021} solve this problem in the Sobolev space $H^2$ when the domain is $\mathbb{T} \times \mathbb{R}$.
	It should be pointed out that the periodic nature of a direction in a region plays an extremely important role in their result.
	\textit{However, when the spatial domain is the whole space $\mathbb{R}^2$, the stability and large-time problems are widely open(see the abstract in \cite{Dong-Wu-2021})}.
	
	
	Now, let us turn back to the study of the inhomogeneous incompressible Navier-Stokes
	equation \eqref{1-ns}.
	{For the case $\inf_x \rho_0(x)>0$}, Kazhikov showed that
	there exists at least one global weak solution in the energy space \cite{Kazhikov1974}. It was widely discussed about the global well-posedness results with small initial data in critical spaces, see \cite{Danchin-2004,Abidi-Gui-Zhang-2011,Huang-Paicu-Zhang2013} . Danchin \cite{Danchin-2004} obtained the local well-posedness results with periodic boundary conditions or in the whole space. And the author further derived a blow-up criterion which entails global well-posedness in dimension two. Abidi, Gui and Zhang \cite{Abidi-Gui-Zhang-2011} investigated the large-time decay and stability to any given global smooth solutions of the 3D incompressible inhomogeneous Navier-Stokes equations. 
	{For the case $\inf_x \rho_0(x)=0$},
	the result in \cite{Kazhikov1974} was generalized to the case of initial data with vacuum
	by Simon \cite{Simon1990}.
	For the existence of strong solution, the local-in-time well-posedness
	with initial density, which may vanish in a open subset of the fluid domain
	in dimension three, was obtained by Choe and Kim \cite{Choe} if the initial
	data satisfies some compatibility condition.
	Both the local and global in time well-posedness results
	in dimension two were obtained in \cite{Liang-2015} and \cite{Lu-Shi-Zhong}.
	Recently, Li \cite{Lijinkai2017} established the local existence and uniqueness of strong solution in lower regular space without any compatibility condition.
	Without the condition that the density is bounded away from zero, the analysis will be more complicated since the system degenerates.
	And Danchin and Mucha \cite{Danchin-2019} showed the existence and uniqueness of the solution in 2D or 3D case with proper initial velocity and only bounded nonnegative density.
	Recently, He, Li and L\"{u} \cite{He-Li-Lv2021} firstly obtained both the global existence and exponential stability of strong solutions in the  three-dimensional whole space,
	provided that the initial velocity is suitably small in some homogeneous Sobolev space.
	The initial density may contain vacuum and even has compact support.
	It should be pointed out that this exponential decay-in-time rate is somewhat surprising since the corresponding decay-in-time rates for the strong solutions
	are algebraic(cf.\cite{{Abidi-Gui-Zhang-2011},{Craig-Huang-Wang-2013},{Lu-Xu-Zhong-2017}})
	even for the homogeneous case
	(cf.\cite{{Chen-1991},{He-Xin-2001},{Kato1984},{Schonbek1995}}).
	\textit{However, to the best of authors' knowledge, whether the solution in dimension two
		will decay exponentially remains an open question}.
	When the viscosity is coupled with density,  it is more difficult to investigate the well-posedness of the Navier-Stokes problems (cf.\cite{{Abidi-Zhang2015},{Lions1996},{Huang-Wang2014},{Huang-Wang2015},
		{Huang-Paicu-Zhang2013},{He-Li-Lv2021}}).

	\textit{
		Thus, our target in this paper is to investigate the global-in-time well-posedness
		and exponential stability for the inhomogeneous anisotropic incompressible Navier-Stokes system \eqref{o-ns} with far-field vacuum in $\mathbb{R}^2$.
		This result will  provide a positive answer to the two-dimensional
		decay open question addressed in \cite{He-Li-Lv2021} and \cite{Dong-Wu-2021}.}

	Before stating our result, we first give some notations.
	
	\textbf{Notations.}
	We denote by
	$\p_x^{\alpha}  \overset{\text{def}}{=} \p_1^{\alpha_0} \p_2^{\alpha_1}, \alpha=(\alpha_0,\alpha_1) \in \mathbb{N}^2$. Here  $\p_i$  is the
	abbreviation of the partial derivative $\p_{x_i}(i=1,2)$.
	For $1 \leq r \leq \infty$, the Lebesgue space and the Sobolev space are defined as follows:
	$$L^r \overset{\text{def}}{=} L^r(\mathbb{R}^2),\
	H^r \overset{\text{def}}{=} H^r(\mathbb{R}^2).$$
	We shall use the notations $A \lesssim B$ meaning $A \le C B$ with a generic constant $C>0$,
	$C_{D,E}$ meaning the constant $C$ depends on the parameters $D$ and $E$.
	For the sake of simplicity, let us set
	$$
	\int f(x) dx \overset{\text{def}}{=} \int_{\mathbb{R}^2} f(x) dx,
	$$
	
	The rest of paper is organized as follows. In Section \ref{sect-diff.} , we explain the main difficulty and our approach to establish the well-posedness theory
	and exponential stability for the inhomogeneous incompressible Navier-Stokes equations with only horizontal dissipation \eqref{o-ns} in two-dimensional whole space.
	In Section \ref{section-well-poseness}, we establish the local-in-time
	well-posedness with large initial data for the Eq.\eqref{o-ns}
	in two-dimensional whole space.
	Finally, we investigate the global-in-time well-posedness and exponential
	stability under the condition of some small initial data
	in Section \ref{section-global-well-poseness}.
	
	\section{Difficulties and Outline of Our Approach}\label{sect-diff.}
	
	In this section, we will explain the main difficulties of proving
	Theorems \ref{local-well} and \ref{global-well} as well as our strategies
	for overcoming them.
	
	Now let us explain the difficulties and our strategies for the
	local-in-time well-posedness theory.
	In order to solve Eq.\eqref{o-ns} in certain $H^s$-spaces,
	we have to overcome the some difficulties as follows.
	First of all, the far field vacuum of density and the disappearance of vertical
	dissipative term of velocity will prevent us from establishing the estimate
	for the vertical derivative of velocity. Thus, we apply the Hardy inequality
	and quantities $\p_1 \p_2^k u_1(k\ge 0)$ to control $\p_2^k u_1(k\ge 0)$
	by adding with weight $\wx$.
	Since the pressure is controlled by the density and velocity, the estimate
	of pressure with weight can be controlled by using the singular integral with
	weight that should be smaller than one \cite{Stein1957}. Thus, this weight is not compatible with the weight generated by the Hardy inequality.
	Then, our method here is to put the weight generated by the Hardy inequality
	along with the density.
	
	Secondly, in order to propagate the weight along with the density,
	we rewrite the density equation $\eqref{o-ns}_1$ as follows:
	\beqq
	\p_t \sqrt{\rho}+ u_1 \p_1 \sqrt{\rho}+u_2 \p_2 \sqrt{\rho}=0.
	\deqq
	Then, the term $u_1 \p_1 \sqrt{\rho}$ will cause the most difficult
	term $\int \p_2^3 u_1 \p_1 \sqrt{\rho} \cdot \p_2^3 \sqrt{\rho} \wx^{2 \gamma}dx$.
	Since the Hardy inequality will add one more weight, our method is
	to require the weight of $\p_1 \sqrt{\rho}$ is larger than $\p_2^3 \sqrt{\rho}$.
	Then, this difficult term can be estimated as follows:
	\beqq
	\begin{aligned}
		& \int \p_2^3 u_1 \p_1 \sqrt{\rho} \cdot \p_2^3 \sqrt{\rho} \wx^{2 \gamma}dx \\
		\lesssim &~ \|\p_2^3 u_1 \wx^{-1}\|_{L^2}^{\frac12}
		\|\p_1(\p_2^3 u_1 \wx^{-1})\|_{L^2}^{\frac12}
		\|\p_1 \sqrt{\rho} \wx^{\gm+1}\|_{L^2}^{\frac12}\\
		&\times \|\p_{12} \sqrt{\rho} \wx^{\gm+1}\|_{L^2}^{\frac12}
		\|\p_2^3 \sqrt{\rho} \wx^{\gm}\|_{L^2}.
	\end{aligned}
	\deqq
	The lower order derivative of density with weight will propagate
	since the quantity $\p_2^k u_1(0\le k \le 2)$ will be controlled
	by using the good estimate of vorticity(see \eqref{21001}).
	
	Thirdly, it should be pointed out when dealing with some high order derivative of velocity, it fails to use the Hardy inequality. For example, the most difficult term is $\int (\p_2 \sqrt{\rho})^2 \p_t u \cdot \p_t \p_2^2 u dx$
	(See term $I_{12, 3}$ in \eqref{2803}). Our method here is to introduce the quantity $\ln \sqrt{\rho}$ to close the estimate. Thus this term is established as follows:
	\beqq
	\begin{aligned}
		\int (\p_2 \sqrt{\rho})^2 \p_t u \cdot \p_t \p_2^2 u dx
		= & \int \p_2 \sqrt{\rho} (\p_2 \ln \sqrt{\rho}) \p_t u \cdot
		\sqrt{\rho} \p_t \p_2^2 u dx \\
		\lesssim&~ \|\sqrt{\rho} \p_t \p_2^2 u\|_{L^2}
		\|\p_t u \wx^{-1}\|_{L^2}^{\frac12}
		\|\p_1(\p_t u \wx^{-1})\|_{L^2}^{\frac12}\\
		&\times
		\|\p_2\ln \sqrt{\rho}\|_{L^2}^{\frac12}
		\|\p_2^2 \ln \sqrt{\rho}\|_{L^2}^{\frac12}
		\|\p_2 \sqrt{\rho} \wx\|_{L^\infty}.
	\end{aligned}
	\deqq
	Thus, we require the density satisfies the regularity
	$\nabla \ln \sqrt{\rho}\in L^\infty(0, T; H^1)$.
	
	Finally, the estimate of the third order derivative of density will employ the
	combination of good property of vorticity and divergence-free condition.
	On the one hand, we can apply the divergence-free condition to establish the estimate
	for the term $\p_1^{\alpha_1}\p_2^{\alpha_2}(u_2 \p_2 \sqrt{\rho})$.
	On the other hand, we can apply the good estimate of vorticity to
	deal with the term  $\p_1^{\alpha_1}\p_2^{\alpha_2}(u_1 \p_1 \sqrt{\rho})$.
	More precisely, let us define the vorticity
	$w\overset{\text{def}}{=} \p_1 u_2-\p_2 u_1$,
	then we can deal with the most difficult term
	$\int \p_2 u_1 \p_2^2 \p_1 \sqrt{\rho}  \cdot \p_2^3 \sqrt{\rho} \wx^{2\gm} dx$
	as follows:
	\beqq
	\begin{aligned}
		&\int \p_2 u_1 \p_2^2 \p_1 \sqrt{\rho}  \cdot \p_2^3 \sqrt{\rho} \wx^{2\gm} dx\\
		\lesssim &~\|\p_2 u_1\|_{L^\infty}
		\|\p_2^2 \p_1 \sqrt{\rho} \wx^\gm\|_{L^2}
		\|\p_2^3 \sqrt{\rho} \wx^\gm\|_{L^2}
		\\
		\lesssim &~(\|w \|_{L^\infty}
		+\|\p_1 u_2\|_{L^\infty})
		\|\p_2^2 \p_1 \sqrt{\rho} \wx^\gm\|_{L^2}
		\|\p_2^3 \sqrt{\rho} \wx^\gm\|_{L^2}\\
		\lesssim &~(\|w\|_{L^2}^{\frac12}\|\p_2 w\|_{L^2}^{\frac12}
		+\|\p_1 w\|_{L^2}^{\frac12}\|\p_{12}w\|_{L^2}^{\frac12}
		+\|\p_1 u_2\|_{L^\infty})
		\|\p_2^2 \p_1 \sqrt{\rho} \wx^\gm\|_{L^2}
		\|\p_2^3 \sqrt{\rho} \wx^\gm\|_{L^2}\\
		\lesssim &~(\|\p_1^2 w \wx^2 \|_{L^2}^{\frac12}
		\|\p_1^2 \p_{2}w \wx^2\|_{L^2}^{\frac12}+\|\p_1 u_2\|_{L^\infty})
		\|\p_2^2 \p_1 \sqrt{\rho} \wx^\gm\|_{L^2}
		\|\p_2^3 \sqrt{\rho} \wx^\gm\|_{L^2}.
	\end{aligned}
	\deqq
	In order to control the quantities $\|\p_1^2 w \wx^2 \|_{L^2}$
	and $\|\p_1^2 \p_{2} w \wx^2\|_{L^2}$,
	we use the velocity equality $\eqref{21008}$ to control them by the energy and the dissipation term $\|(\sqrt{\rho}\p_t \p_2^2 u_1,\p_t \p_{12} u_2)\|_{L^2}$.
	
	Now, let us focus on the global-in-time well-posedness and exponential stability.
	Indeed, the exponential stability depends on the following  new observation.
	Recall the basic energy equality(see \eqref{44102})
	\beqq
	\frac{d}{dt} \|\sqrt{\rho} u\|_{L^2}^2+2\|\p_1 u\|_{L^2}^2=0.
	\deqq
	The Hardy inequality yields directly
	\beqq
	\|\sqrt{\rho} u\|_{L^2}^2
	\lesssim \|\sqrt{\rho} \wx \|_{L^\infty}^2\|u \wx^{-1}\|_{L^2}^2
	\lesssim \|\sqrt{\rho} \wx \|_{L^\infty}^2\|\p_1 u \|_{L^2}^2.
	\deqq
	Then, we have the interesting energy inequality
	\beq\label{analysis}
	\frac{d}{dt} \|\sqrt{\rho} u\|_{L^2}^2
	+\|\sqrt{\rho} \wx \|_{L^\infty_t L^\infty_x}^{-2}
	\|\sqrt{\rho} u\|_{L^2}^2
	+\|\p_1 u\|_{L^2}^2\le 0.
	\deq
	The technique, mentioned before as we investigate the local-in-time well-posedness,
	can help us obtain uniform bound for the quantity
	$\|\sqrt{\rho} \wx \|_{L^\infty_t L^\infty_x}$.
	This and the inequality \eqref{analysis} can yield the exponential stability.
	Furthermore, this exponential integral of velocity will help us
	establish the uniform estimate of density again.
	Based on these analysis, we can obtain the global-in-time well-posedness
	and exponential stability under the condition of small initial data.
	
	In this paper, we will use the anisotropic Sobolev inequalities (see \eqref{holder3}-\eqref{inf-ineq2}) frequently to deal with the nonlinear term. Similar anisotropic Sobolev inequalities can be found in \cite{{Lai-Wu-Zhang2022},{Cao-Wu-JDE2013},{Cao-Wu-AD-2011}}.

	
	\section{Local-in-time well-posedness of equation}
	\label{section-well-poseness}
	
	In order to guarantee the density can absorb the weight generated by the Hardy inequality,
	we introduce the good unknown $\phi \overset{\text{def}}{=} \sqrt{\rho}$.
	Then, from the original system \eqref{o-ns}, the function
	$(\phi, u, p)$ will satisfy
	\beq\label{new-ns}
	\left\{\begin{aligned}
		&\p_t \phi+u \cdot \nabla \phi=0,\\
		&\phi^2 \p_t u+\phi^2 u\cdot \nabla u-\p_1^2 u+\nabla p=0,\\
		&{\rm div}u=0,\\
		&(\phi, u)|_{t=0}=(\phi_0, u_0)\overset{\text{def}}{=}(\sqrt{\rho_0}, u_0),\\
		&\lim_{|x|\to +\infty}(\phi, u)=(0, 0).
	\end{aligned}\right.
	\deq
	Let us denote the energy norm
	\beqq
	\begin{aligned}
		\mathcal{E}(t)
		\overset{\text{def}}{=}
		&~\sum_{k=0}^3 \|\phi \p_2^k u\|_{L^2}^2
		+\sum_{k=0}^2 \|\p_2^k \p_1 u\|_{L^2}^2
		+\sum_{k=0}^1\|\phi \p_t \p_2^k u\|_{L^2}^2
		+\|\p_t \p_1 u\|_{L^2}^2\\
		&+\sum_{0\le |\alpha|\le 2}\|\p_x^\alpha \phi \wx^{\gm+1}\|_{L^2}^2
		+\sum_{|\alpha|=3}\|\p_x^\alpha \phi \wx^{\gm}\|_{L^2}^2
		+\sum_{1\le |\alpha| \le 2}\|\p_x^\alpha \ln \phi\|_{L^2}^2,
	\end{aligned}
	\deqq
	and the dissipative norm
	\beqq
	\begin{aligned}
		\mathcal{D}(t)\overset{\text{def}}{=}
		&~\sum_{k=0}^2 \|\p_1 \p_2^k u\|_{L^2}^2
		+C_1 \|\p_1 \p_2^3 u\|_{L^2}^2
		+ \sum_{k=0}^2 \|\phi \p_2^k  \p_t u\|_{L^2}^2
		+\|\p_1 \p_t u\|_{L^2}^2\\
		&+C_2\|\p_{12} \p_t u\|_{L^2}^2 +\|\phi \p_t^2 u\|_{L^2}^2,
	\end{aligned}
	\deqq
	where the constants $C_1, C_2$ will be chosen later.
	Let us define the initial data norm as follows:
	\beqq
	\begin{aligned}
		\mathcal{E}(\phi_0, u_0)
		\overset{\text{def}}{=}
		&~\sum_{k=0}^3\|\phi_0 \p_2^k u_0\|_{L^2}^2
		+\sum_{k=0}^2 \|\p_2^k \p_1 u_0\|_{L^2}^2
		+\sum_{0\le |\alpha|\le 2}\|\p_x^\alpha \phi_0 \wx^{\gm+1}\|_{L^2}^2\\
		&+\sum_{|\alpha|=3}\|\p_x^\alpha \phi_0 \wx^{\gm}\|_{L^2}^2
		+\sum_{1\le |\alpha| \le 2}\|\p_x^\alpha \ln \phi_0\|_{L^2}^2.
	\end{aligned}
	\deqq
	Now, we state a priori estimate for the system \eqref{new-ns}.
	\begin{prop}\label{pro-priori}
		Let $\gamma \ge 2$ be an integer. Assume the initial data
		$(\phi_0, u_0)(x)$ satisfies
		$$
		\mathcal{E}(\phi_0, u_0)<+\infty,
		$$
		and
		\beq\label{c-d}
		\phi_0^2 u_0 \cdot \nabla u_0
		-\p_1^2 u_0+\nabla p_0=\phi_0^2 g \quad \text{\rm in} \quad \mathbb{R}^2,
		\deq
		for some $(\nabla p_0, g)\in L^2 \times H^1$.
		Suppose the good unknown $\phi(x, t)$  satisfies
		$$\phi(x, t)>0 \ \ \text{\rm in} \ \ \mathbb{R}^2 \times [0, T].$$
		Let $(\phi(t,x), u(t,x), p(t,x))$ be a smooth solution for the
		system \eqref{new-ns}, then it holds
		\beq\label{priori-estimate}
		{E}(T)+\int_0^T  \mathcal{D}(\tau) d\tau
		\le   C_{\phi_0, u_0, p_0, \gamma}
		+C (1+E(T)^9)T
		+C (1+E(T)^{14}) T^2 \int_0^T(\|\p_{12} \p_t u \|_{L^2}^2
		+\|\p_1 \p_2^3 u \|_{L^2}^2) d\tau,
		\deq
		where $E(T) \overset{\text{def}}{=}
		\underset{0\le \tau \le T}{\sup}\mathcal{E}(\tau)$.
	\end{prop}
	
	Before investigating the estimates stated in Proposition \ref{pro-priori},
	we state the following inequalities that will be used frequently in
	subsections \ref{estimate-v} and \ref{estimate-rho}.
	\begin{lemm}\label{tool}
		For any suitable functions $(f(x), g(x), h(x))$ defined on $\mathbb{R}^2$,
		the following estimates hold 
		\beq\label{holder3}
		\int |f g h |dx
		\lesssim \|f\|_{L^2}
		\|g\|_{L^2}^{\frac12}
		\|\p_1 g\|_{L^2}^{\frac12}
		\|h\|_{L^2}^{\frac12}
		\|\p_2 h\|_{L^2}^{\frac12},
		\deq
		and
		\beq\label{inf-ineq1}
		\|f\|_{L^\infty}
		\lesssim \|f\|_{L^2}^{\frac12}\|\p_1 f\|_{L^2}^{\frac12}
		+\|\p_2 f\|_{L^2}^{\frac12}\|\p_{12} f\|_{L^2}^{\frac12},
		\deq
		or
		\beq\label{inf-ineq2}
		\|f\|_{L^\infty}
		\lesssim \|f\|_{L^2}^{\frac12}\|\p_2 f\|_{L^2}^{\frac12}
		+\|\p_1 f\|_{L^2}^{\frac12}\|\p_{12} f\|_{L^2}^{\frac12}.
		\deq
		For any $\gamma \neq -\f12$, we have the following Hardy inequality
		\beq\label{Hardy}
		\|\wx^\gamma f\|_{L^2}
		\lesssim \|\wx^{\gamma+1} \p_1 f\|_{L^2}.
		\deq
	\end{lemm}
	These anisotropic Sobolev inequalities \eqref{holder3}-\eqref{inf-ineq2} can be found
	in \cite{{Lai-Wu-Zhang2022},{Cao-Wu-JDE2013},{Cao-Wu-AD-2011}} if the suitable functions
	are defined on $\mathbb{R}^2$.
	
	\subsection{Estimate for the velocity}\label{estimate-v}
	\quad
	First of all, let us establish the following energy equality
	for the equation \eqref{new-ns}.
	
	\begin{lemm}\label{lemma21}
		For any smooth solution of equation \eqref{new-ns}, it holds 
		\beq\label{2201}
		\frac{d}{dt}\frac12\int \phi^2 |u|^2 dx
		+\int |\p_1 u|^2 dx=0.
		\deq
	\end{lemm}
	\begin{proof}
		Multiplying the equation $\eqref{new-ns}_2$ by $u$
		and integrating over $\mathbb{R}^2$, it holds 
		\beq\label{2202}
		\int (\phi^2 \p_t u+\phi^2 u\cdot \nabla u-\p_1^2 u+\nabla p)\cdot u dx=0.
		\deq
		Integrating by part and using the divergence-free condition $\eqref{new-ns}_3$, we have
		\beq\label{2203}
		\int \nabla p \cdot u dx=-\int p {\rm div} u dx=0,
		\deq
		and
		\beq\label{2204}
		-\int \p_1^2 u \cdot u dx=\int |\p_1 u|^2 dx.
		\deq
		Using the density equation $\eqref{new-ns}_1$, divergence-free condition $\eqref{new-ns}_3$
		and integrating by part, we have
		\beq\label{2205}
		\begin{aligned}
			\int \phi^2 \p_t u \cdot u dx
			&=\frac{d}{dt}\frac12\int \phi^2 |u|^2 dx
			-\frac12 \int |u|^2 \p_t (\phi^2) dx\\
			&=\frac{d}{dt}\frac12\int \phi^2 |u|^2 dx
			+\frac12 \int |u|^2 u \cdot \nabla (\phi^2) dx\\
			&=\frac{d}{dt}\frac12\int \phi^2 |u|^2 dx
			-\frac12 \int \phi^2 {\rm div}(|u|^2 u)dx\\
			&=\frac{d}{dt}\frac12\int \phi^2 |u|^2 dx
			-\frac12 \int \phi^2 u\cdot \nabla (|u|^2)dx.
		\end{aligned}
		\deq
		Substituting the equalities \eqref{2203}, \eqref{2204}
		and \eqref{2205} into \eqref{2202}, we can obtain that
		\beqq
		\frac{d}{dt}\frac12\int \phi^2 |u|^2 dx
		+\int |\p_1 u|^2 dx=0,
		\deqq
		which yields the equality \eqref{2201} directly.
		Therefore, we complete the proof this lemma.
	\end{proof}
	
	\begin{lemm}\label{lemma22}
		For any smooth solution of equation \eqref{new-ns},
		it holds  with $\gamma \geq 2$
		\beq\label{22001}
		\frac{d}{dt}\| \p_1 u\|_{L^2}^2+\|\phi \p_t u\|_{L^2}^2 \lesssim \e^3.
		\deq
	\end{lemm}
	\begin{proof}
		Multiplying the equation $\eqref{new-ns}_2$ by $\p_t u$
		and integrating over $\mathbb{R}^2$, we have
		\beq\label{22002}
		\frac{d}{dt}\frac12 \int |\p_1 u|^2 dx
		+\int \phi^2 |\p_t u|^2 dx=-\int \phi^2 (u \cdot \nabla u)\cdot \p_t u dx.
		\deq
		Using the anisotropic Sobolev inequality \eqref{holder3}
		and Hardy inequality \eqref{Hardy}, we have
		\beq\label{22003}
		\begin{aligned}
			&\left|\int \phi^2 (u \cdot \nabla u)\cdot \p_t u dx \right|\\
			\lesssim &~\|\phi \wx^2\|_{L^\infty}\|\phi \p_t u\|_{L^2}
			\|u_1 \wx^{-1}\|_{L^2}^{\frac12}
			\|\p_1 (u_1 \wx^{-1})\|_{L^2}^{\frac12}\\
			&~\times\|\p_1 u \wx^{-1}\|_{L^2}^{\frac12}
			\|\p_{21} u \wx^{-1}\|_{L^2}^{\frac12}\\
			&~+\|\phi \wx^2\|_{L^\infty}\|\phi \p_t u\|_{L^2}
			\|u_2 \wx^{-1}\|_{L^2}^{\frac12}
			\|\p_2 (u_2 \wx^{-1})\|_{L^2}^{\frac12}\\
			&~\times\|\p_2 u \wx^{-1}\|_{L^2}^{\frac12}
			\|\p_1(\p_{2} u \wx^{-1})\|_{L^2}^{\frac12}\\
			\lesssim &~\|\phi \wx^2\|_{H^2}\|\phi \p_t u\|_{L^2}
			\|(\p_1 u, \p_{12} u)\|_{L^2}^2\\
			\le &~\var \|\phi \p_t u\|_{L^2}^2
			+C_\var \|\phi \wx^2\|_{H^2}^2\|(\p_1 u, \p_{12} u)\|_{L^2}^4.
		\end{aligned}
		\deq
		Substituting the estimate \eqref{22003}
		into \eqref{22002} and choosing $\var$ small enough, we have
		\beqq
		\frac{d}{dt}\frac12 \int |\p_1 u|^2 dx
		+\int \phi^2 |\p_t u|^2 dx
		\lesssim \|\phi \wx^2\|_{H^2}^2\|(\p_1 u, \p_{12} u)\|_{L^2}^4.
		\deqq
		Therefore, we complete the proof of this lemma.
	\end{proof}

	Next, we will establish the estimate for the vertical derivative of velocity.
	\begin{lemm}\label{lemma23}
		For any smooth solution of equation \eqref{new-ns}, it holds with $\gamma \geq 2$
		\beq\label{2301}
		\begin{aligned}
			&\frac{d}{dt} \| \phi \p_2 u\|_{L^2}^2 +\|\p_1 \p_2 u\|_{L^2}^2
			\lesssim 1+\e^5;\\
			&\frac{d}{dt} \| \phi \p_2^2 u\|_{L^2}^2 +\|\p_1 \p_2^2 u\|_{L^2}^2
			\lesssim 1+\e^5;
		\end{aligned}
		\deq
		and
		\beq\label{23001}
		\begin{aligned}
			\frac{d}{dt} \| \phi \p_2^3 u\|_{L^2}^2 +\|\p_1 \p_2^3 u\|_{L^2}^2
			\le  \sigma_1 \|\phi \p_2^2 \p_t u\|_{L^2}^2
			+\sigma_2\|\p_{12} \p_t u\|_{L^2}^2
			+C_{\sigma_1, \sigma_2}(1+\e^7),
		\end{aligned}
		\deq
		where $\sigma_1$ and $\sigma_2$ are the small constants that will be chosen later.
	\end{lemm}
	\begin{proof}
		Taking $\p_2^k(k=1,2,3)$ operator to the equation $\eqref{new-ns}_2$, multiplying the equation  by $\p_2^k u$ and integrating over $\mathbb{R}^2$, it holds 
		\beq\label{2302}
		\int \p_2^k (\phi^2 \p_t u+\phi^2 u\cdot \nabla u-\p_1^2 u+\nabla p)
		\cdot \p_2^k u  dx=0.
		\deq
		
		\textbf{Step 1: Deal with the case $k=1$ in \eqref{2302}}.
		Integrating by part and using the divergence-free condition
		$\eqref{new-ns}_2$, it is easy to check that
		\beq\label{2303}
		\begin{aligned}
			\underset{I_{1}}{\underbrace{\int \p_2(\phi^2 \p_t u )\cdot \p_2 u dx}}
			+\underset{I_{2}}{\underbrace{\int \p_2(\phi^2 u\cdot \nabla u) \cdot \p_2 u dx}}
			+\int |\p_{1}\p_{2} u|^2dx=0.
		\end{aligned}
		\deq
		Now let us deal with term $I_1$.
		Indeed, using the density equation $\eqref{new-ns}_1$, it holds
		\beq\label{2304}
		\begin{aligned}
			I_1
			&=\int \phi^2 \p_t \p_2 u \cdot \p_2 u dx
			+\int \p_2(\phi^2) \p_t u \cdot \p_2 u dx\\
			&=\frac{d}{dt}\frac12 \int \phi^2 |\p_2 u|^2 dx
			-\frac12 \int |\p_2 u|^2 \p_t (\phi^2)dx
			+2 \int \phi \p_2 \phi \p_t u \cdot \p_2 u dx\\
			&=\frac{d}{dt}\frac12 \int \phi^2 |\p_2 u|^2 dx
			+\underset{I_{1,1}}{\underbrace{\int |\p_2 u|^2 \phi u \cdot \nabla \phi dx}}
			+\underset{I_{1,2}}{\underbrace{2 \int \phi \p_2 \phi \p_t u \cdot \p_2 u dx}}.
		\end{aligned}
		\deq
		Using the anisotropic Sobolev inequality \eqref{holder3}
		and Hardy inequality \eqref{Hardy}, we have
		\beq\label{2305}
		\begin{aligned}
			|I_{1,1}|
			&\lesssim \|\phi \p_2 u\|_{L^2}
			\|\p_2 u \wx^{-1}\|_{L^2}^{\frac12}
			\|\p_1(\p_2 u \wx^{-1})\|_{L^2}^{\frac12}
			\|\nabla \phi\wx^2 \|_{L^2}^{\frac12}
			\|\p_2(\nabla \phi\wx^2)\|_{L^2}^{\frac12}
			\|u\wx^{-1}\|_{L^\infty}\\
			&\lesssim \|\phi \p_2 u\|_{L^2} \|\p_1 \p_2 u \|_{L^2}
			\|(\nabla \phi, \p_2 \nabla \phi)\wx^2 \|_{L^2}
			\|u\wx^{-1}\|_{L^\infty}.
		\end{aligned}
		\deq
		Due to the anisotropic Sobolev inequality \eqref{inf-ineq1}
		and Hardy inequality \eqref{Hardy}, we have
		\beq\label{2306}
		\begin{aligned}
			\|u\wx^{-1}\|_{L^\infty}
			&\lesssim \|u \wx^{-1}\|_{L^2}^{\frac12}
			\|\p_2(u \wx^{-1})\|_{L^2}^{\frac12}
			+\|\p_1 (u \wx^{-1})\|_{L^2}^{\frac12}
			\|\p_1\p_2 (u \wx^{-1})\|_{L^2}^{\frac12}\\
			&\lesssim \|\p_1 u\|_{L^2}^{\frac12}\|\p_1\p_2 u \|_{L^2}^{\frac12},
		\end{aligned}
		\deq
		which, together with estimate \eqref{2305}, yields directly
		\beq\label{2307}
		\begin{aligned}
			|I_{1,1}|
			&\lesssim \|\phi \p_2 u\|_{L^2}
			\|\p_1 u\|_{L^2}^{\frac12}
			\|\p_1 \p_2 u \|_{L^2}^{\frac32}
			\|(\nabla \phi, \p_2 \nabla \phi)\wx^2 \|_{L^2}\\
			&\le \var\|\p_1 \p_2 u \|_{L^2}^2
			+C_\var \|\phi \p_2 u\|_{L^2}^4
			\|\p_1 u\|_{L^2}^2
			\|(\nabla \phi, \p_2 \nabla \phi)\wx^2 \|_{L^2}^4.
		\end{aligned}
		\deq
		Similarly, it is easy to check that
		\beq\label{2308}
		\begin{aligned}
			|I_{1,2}|
			&\lesssim \|\phi \p_t u\|_{L^2}
			\|\p_2 u \wx^{-1}\|_{L^2}^{\frac12}
			\|\p_1(\p_2 u \wx^{-1})\|_{L^2}^{\frac12}
			\|\p_2 \phi \wx\|_{L^2}^{\frac12}
			\|\p_2^2 \phi \wx\|_{L^2}^{\frac12}\\
			&\lesssim \|\phi \p_t u\|_{L^2}
			\|\p_1 \p_2 u \|_{L^2}
			\|\p_2 \phi \wx\|_{L^2}^{\frac12}
			\|\p_2^2 \phi \wx\|_{L^2}^{\frac12}\\
			&\le  \var\|\p_1 \p_2 u \|_{L^2}^2
			+C_\var\|\phi \p_t u\|_{L^2}^2
			\|(\p_2 \phi, \p_2^2 \phi)\wx\|_{L^2}^2.
		\end{aligned}
		\deq
		Thus, substituting the estimates \eqref{2307} and
		\eqref {2308} into \eqref{2304}, we get that
		\beq\label{2309}
		I_1\ge \frac{d}{dt}\frac12 \int \phi^2 |\p_2 u|^2 dx
		-\var \|\p_1 \p_2 u \|_{L^2}^2
		-C_\var (1+\e^5).
		\deq
		Now let us deal with term $I_2$.
		Integrating by part and using the
		divergence-free condition $\eqref{new-ns}_3$, we have
		\beqq
		\begin{aligned}
			I_2
			&=\int (2\phi \p_2 \phi u \cdot \nabla u
			+\phi^2 \p_2 u\cdot \nabla u
			+\phi^2 u \cdot \nabla \p_2 u)\cdot \p_2 u dx\\
			&=\underset{I_{2,1}}{\underbrace{2\int \phi \p_2 \phi u \cdot \nabla u \cdot \p_2 u dx}}
			+\underset{I_{2,2}}{\underbrace{\int \phi^2 \p_2 u\cdot \nabla u \cdot \p_2 u dx}}
			-\underset{I_{2,3}}{\underbrace{\int |\p_2 u|^2 \phi u \cdot \nabla \phi dx}}.
		\end{aligned}
		\deqq
		Using the estimate \eqref{2306}, anisotropic Sobolev inequality and
		Hardy inequality, we have
		\beqq
		\begin{aligned}
			|I_{2,1}|
			\lesssim &~\|\phi \p_2 u\|_{L^2}
			\|\p_1 u\|_{L^2}^{\frac12}
			\|\p_2 \p_1 u \|_{L^2}^{\frac12}
			\|\p_2 \phi \wx\|_{L^2}^{\frac12}
			\|\p_1 (\p_2 \phi \wx)\|_{L^2}^{\frac12}
			\|u_1 \wx^{-1}\|_{L^\infty}\\
			&+\|\phi \p_2 u\|_{L^2}
			\|\p_2 u\wx^{-1}\|_{L^2}^{\frac12}
			\|\p_1 (\p_2 u \wx^{-1})\|_{L^2}^{\frac12}
			\|\p_2 \phi \wx^2\|_{L^2}^{\frac12}
			\|\p_2 (\p_2 \phi \wx^2)\|_{L^2}^{\frac12}
			\|u_2 \wx^{-1}\|_{L^\infty}\\
			\le &~\var\|\p_1 \p_2 u \|_{L^2}^2
			+C_\var(\|\phi \p_2 u\|_{L^2}^2\|\p_1 u\|_{L^2}^2
			\|\phi\wx\|_{H^2}^2
			+\|\phi \p_2 u\|_{L^2}^4\|\p_1 u\|_{L^2}^2
			\|\phi \wx^2\|_{H^2}^4).
		\end{aligned}
		\deqq
		Similarly, we have
		\beqq
		\begin{aligned}
			|I_{2,2}|
			\lesssim &~\|\phi \p_2 u\|_{L^2}
			\|\p_1 u\|_{L^2}^{\frac12}
			\|\p_{21} u\|_{L^2}^{\frac12}
			\|\p_2 u\wx^{-1}\|_{L^2}^{\frac12}
			\|\p_1 (\p_2 u\wx^{-1})\|_{L^2}^{\frac12}
			\|\phi \wx\|_{L^\infty}\\
			\lesssim&~\|\phi \p_2 u\|_{L^2}
			\|\p_1 u\|_{L^2}^{\frac{1}{2}}
			\|\p_{12} u\|_{L^2}^{\frac{3}{2}}
			\| \phi \wx^2 \|_{H^2}\\
			\le &~\var \|\p_1\p_2 u\|_{L^2}^2
			+C_\var \|\phi \p_2 u\|_{L^2}^4 \|\p_1 u\|_{L^2}^2
			\|\phi \wx^2\|_{H^2}^4,
		\end{aligned}
		\deqq
		and
		\beqq
		\begin{aligned}
			|I_{2,3}|
			&\lesssim \|\phi \p_2 u\|_{L^2}
			\|\p_2 u\wx^{-1}\|_{L^2}^{\frac12}
			\|\p_1 (\p_2 u\wx^{-1})\|_{L^2}^{\frac12}
			\|\nabla \phi \wx^2\|_{L^2}^{\frac12}
			\|\p_2 (\nabla \phi \wx^2)\|_{L^2}^{\frac12}
			\|u \wx^{-1}\|_{L^\infty}\\
			&\lesssim \|\phi \p_2 u\|_{L^2}
			\|\p_1 \p_2 u\|_{L^2}
			\|(\nabla \phi,\p_2 \nabla \phi)\wx^2\|_{L^2}
			\|\p_1 u\|_{L^2}^{\frac12}\|\p_1\p_2 u \|_{L^2}^{\frac12}\\
			&\le \var\|\p_1 \p_2 u \|_{L^2}^2
			+C_\var \|\phi \p_2 u\|_{L^2}^4
			\|\p_1 u\|_{L^2}^2 \| \phi\wx^2\|_{H^2}^4.
		\end{aligned}
		\deqq
		Thus, the combination of estimates from $I_{2,1}$ term
		to $I_{2,3}$ term yields directly
		\beq\label{2310}
		|I_2|\le \var\|\p_1 \p_2 u \|_{L^2}^2
		+C_\var (1+\e^5).
		\deq
		Substituting the estimates of  \eqref{2309} and \eqref{2310}
		into equality \eqref{2303} and choosing $\var$ small enough, we have
		\beqq
		\frac{d}{dt}\int \phi^2 |\p_2 u|^2 dx
		+\int |\p_{1}\p_{2} u|^2dx
		\lesssim 1+\e^5,
		\deqq
		which implies the estimate $\eqref{2301}_1$ directly.
		
		\textbf{Step 2: Deal with the case $k=2$ in \eqref{2302}}.
		Thus, we have
		\beq\label{2311}
		\underset{I_{3}}{\underbrace{\int \p_2^2(\phi^2 \p_t u )\cdot \p_2^2 u dx}}
		+\underset{I_{4}}{\underbrace{\int \p_2^2(\phi^2 u\cdot \nabla u) \cdot \p_2^2 u dx}}
		+ \int |\p_{1}\p_{2}^2 u|^2dx=0.
		\deq
		Similar to \eqref{2309} and \eqref{2310}, we can estimate $I_3$
		and $I_4$ terms as follows:
		\beq\label{2317}
		I_{3}
		\ge \frac{d}{dt}\frac12 \int \phi^2 |\p_2^2 u|^2 dx
		-\var \|\p_1 \p_2^2 u \|_{L^2}^2-C_\var (1+\e^3),
		\deq
		and
		\beq\label{2321}
		|I_{4}| \le  \var \|\p_1\p_2^2 u \|_{L^2}^2+C_\var (1+\e^5),
		\deq
		which can be found in detail in appendix \ref{claim-estimate}.
		Substituting the estimates \eqref{2317} and \eqref{2321}
		into \eqref{2311} and choosing $\var$ small enough, we can obtain
		\beqq
		\frac{d}{dt}\int \phi^2 |\p_2^2 u|^2 dx
		+\int |\p_{1}\p_2^2 u|^2dx
		\lesssim 1+\e^5,
		\deqq
		which implies the estimate $\eqref{2301}_2$ directly.
		
		\textbf{Step 3: Deal with the case $k=3$ in \eqref{2302}}.
		Thus, we have
		\beq\label{2322}
		\underset{I_{5}}{\underbrace{\int \p_2^3(\phi^2 \p_t u )\cdot \p_2^3 u dx}}
		+\underset{I_{6}}{\underbrace{\int \p_2^3(\phi^2 u\cdot \nabla u) \cdot \p_2^3 u dx}}
		+\int |\p_{1}\p_{2}^3 u|^2dx=0.
		\deq
		Let us deal with the $I_{5}$ term.
		Using the density equation $\eqref{new-ns}_1$, it is easy to check that
		\beqq
		\begin{aligned}
			I_5
			=&\frac{d}{dt}\frac12\int \phi^2 |\p_2^3 u|^2 dx
			+\int \phi u \cdot \nabla \phi |\p_2^3 u|^2 dx
			+3\int \p_2(\phi^2)\p_2^2 \p_t u \cdot \p_2^3 u dx\\
			&+3\int \p_2^2(\phi^2)\p_2 \p_t u \cdot \p_2^3 u dx
			+\int \p_2^3 (\phi^2)\p_t u \cdot \p_2^3 u dx\\
			\overset{\text{def}}{=} &\frac{d}{dt}\frac12\int \phi^2 |\p_2^3 u|^2 dx
			+I_{5,1}+I_{5,2}+I_{5,3}+I_{5,4}.
		\end{aligned}
		\deqq
		Using the Sobolev inequality and Hardy inequality, we have
		\beqq
		\begin{aligned}
			|I_{5,1}|
			&\lesssim \|\phi \p_2^3 u\|_{L^2}
			\|\p_2^3 u \wx^{-1}\|_{L^2}^{\frac12}
			\|\p_1(\p_2^3 u \wx^{-1})\|_{L^2}^{\frac12}
			\|\nabla \phi \wx^2\|_{L^2}^{\frac12}
			\|\p_2 \nabla \phi \wx^2\|_{L^2}^{\frac12}
			\|u \wx^{-1}\|_{L^\infty}\\
			&\lesssim \|\phi \p_2^3 u\|_{L^2}\|\p_1 \p_2^3 u\|_{L^2}
			\|(\p_1 u, \p_{12} u)\|_{L^2}
			\|(\nabla \phi, \p_2 \nabla \phi)\wx^2\|_{L^2}\\
			&\le \var\|\p_1 \p_2^3 u\|_{L^2}^2
			+C_\var \|\phi \p_2^3 u\|_{L^2}^2
			\|(\p_1 u, \p_{12} u)\|_{L^2}^2
			\|(\nabla \phi, \p_2 \nabla \phi)\wx^2\|_{L^2}^2.
		\end{aligned}
		\deqq
		It is easy to check that
		\beq\label{2324}
		\begin{aligned}
			|I_{5,2}|
			&=\left|6\int \phi^2 (\p_2 \ln \phi) \p_2^2 \p_t u \cdot \p_2^3 u dx \right|\\
			&\lesssim
			\|\phi \p_2^2 \p_t u\|_{L^2}
			\|\phi \p_2^3 u\|_{L^2}^{\frac12}
			\|\p_1(\phi \p_2^3 u)\|_{L^2}^{\frac12}
			\|\p_2 \ln \phi\|_{L^2}^{\frac12}
			\|\p_2^2 \ln \phi\|_{L^2}^{\frac12}.
		\end{aligned}
		\deq
		Using the Sobolev inequality, we can get
		\beqq
		\begin{aligned}
			\|\p_1(\phi \p_2^3 u)\|_{L^2}
			\lesssim
			&~\|\p_1 \phi \p_2^3 u\|_{L^2}
			+\|\phi \p_1 \p_2^3 u\|_{L^2}\\
			\lesssim
			&~\|\p_2^3 u \wx^{-1}\|_{L^2}^{\frac12}
			\|\p_1(\p_2^3 u \wx^{-1})\|_{L^2}^{\frac12}
			\|\p_1 \phi \wx\|_{L^2}^{\frac12}
			\|\p_{21} \phi \wx\|_{L^2}^{\frac12}\\
			&~+\|\p_1 \p_2^3 u\|_{L^2}\|\p_1 \phi \wx\|_{L^2}^{\frac12}
			\|\p_{12} \phi \wx\|_{L^2}^{\frac12}\\
			\lesssim
			&~\|\p_1 \p_2^3 u\|_{L^2}\|(\p_1 \phi, \p_{12} \phi)\wx\|_{L^2},
		\end{aligned}
		\deqq
		which, together with \eqref{2324}, yields directly
		\beqq
		\begin{aligned}
			|I_{5,2}|
			&\lesssim
			\|\phi \p_2^2 \p_t u\|_{L^2}
			\|\phi \p_2^3 u\|_{L^2}^{\frac12}
			\|\p_1 \p_2^3 u\|_{L^2}^{\frac12}
			\|(\p_1 \phi, \p_{12} \phi)\wx\|_{L^2}^{\frac12}
			\|(\p_2 \ln \phi, \p_2^2 \ln \phi)\|_{L^2}\\
			&\le
			\var\|\p_1 \p_2^3 u\|_{L^2}^2
			+\sigma_1 \|\phi \p_2^2 \p_t u\|_{L^2}^2
			+C_{\var, \sigma_1}\|\phi \p_2^3 u\|_{L^2}^2
			\|\phi \wx\|_{H^2}^2
			\|(\p_2 \ln \phi, \p_2^2 \ln \phi)\|_{L^2}^4.
		\end{aligned}
		\deqq
		Obviously, it holds that
		\beqq
		I_{5,3}
		=6\int (\p_2 \phi)^2 \p_2 \p_t u\cdot \p_2^3 u dx
		+6\int \phi \p_2^2 \phi \p_{2}\p_{t}u \cdot \p_2^3 u dx
		\overset{\text{def}}{=} I_{5,3,1}+I_{5,3,2}.
		\deqq
		Using the Sobolev and Hardy inequalities, we have
		\beq\label{2325}
		\begin{aligned}
			|I_{5,3,1}|
			=&~\left|6\int \p_2 \phi (\p_2 \ln \phi) \p_2 \p_t u \cdot \phi \p_2^3 u dx \right|\\
			\lesssim &~\|\p_2 \phi \wx\|_{L^\infty}
			\|\phi \p_2^3 u \|_{L^2}
			\|\p_2 \ln \phi\|_{L^2}^{\frac12}
			\|\p_2^2 \ln \phi \|_{L^2}^{\frac12}\\
			&~\times\|\p_2 \p_t u \wx^{-1}\|_{L^2}^{\frac12}
			\|\p_1(\p_2 \p_t u \wx^{-1})\|_{L^2}^{\frac12}\\
			\lesssim &~\|\p_2 \phi \wx\|_{H^2} \|\phi \p_2^3 u\|_{L^2}
			\|(\p_2 \ln \phi, \p_2^2 \ln \phi)\|_{L^2}
			\|\p_1 \p_2 \p_t u \|_{L^2}\\
			\le &~\sigma_2 \|\p_1 \p_2 \p_t u \|_{L^2}^2
			+C_{\sigma_2}\|\p_2 \phi \wx\|_{H^2}^2 \|\phi \p_2^3 u\|_{L^2}^2
			\|(\p_2 \ln \phi, \p_2^2 \ln \phi)\|_{L^2}^2,
		\end{aligned}
		\deq
		and
		\beq\label{2326}
		\begin{aligned}
			|I_{5,3,2}|
			&\lesssim \|\phi \p_2^3 u\|_{L^2}
			\|\p_2 \p_t u \wx^{-1}\|_{L^2}^{\frac12}
			\|\p_1(\p_2 \p_t u \wx^{-1})\|_{L^2}^{\frac12}
			\|\p_2^2 \phi \wx\|_{L^2}^{\frac12}
			\|\p_2^3 \phi \wx\|_{L^2}^{\frac12}\\
			&\lesssim \|\phi \p_2^3 u\|_{L^2}
			\|\p_1 \p_2 \p_t u \|_{L^2}
			\|(\p_2^2 \phi, \p_2^3 \phi)\wx\|_{L^2}\\
			&\le \sigma_2\|\p_1 \p_2 \p_t u \|_{L^2}^2
			+C_{\sigma_2}\|\phi \p_2^3 u\|_{L^2}^2
			\|(\p_2^2 \phi, \p_2^3 \phi)\wx\|_{L^2}^2.
		\end{aligned}
		\deq
		The combination of estimates \eqref{2325} and \eqref{2326}, we have
		\beqq
		|I_{5,3}|
		\le \var\|\p_1 \p_2^3 u\|_{L^2}^2
		+\sigma_2\|\p_1 \p_2 \p_t u \|_{L^2}^2
		+C_{\var,\sigma_2}(1+\e^3).
		\deqq
		Obviously, it holds that
		\beqq
		\begin{aligned}
			I_{5,4}=6\int \p_2 \phi \p_2^2 \phi \p_t u \cdot \p_2^3 u dx
			+2\int \phi \p_2^3 \phi \p_t u \cdot \p_2^3 u dx
			\overset{\text{def}}{=} I_{5,4,1}+I_{5,4,2}.
		\end{aligned}
		\deqq
		Using the Sobolev, Hardy and Cauchy inequalities, one arrives at
		\beq\label{2327}
		\begin{aligned}
			|I_{5,4,1}|
			\lesssim &~\|\p_2^3 u \wx^{-1}\|_{L^2}^{\frac12}
			\|\p_1(\p_2^3 u \wx^{-1})\|_{L^2}^{\frac12}
			\|\p_2^2 \phi \wx\|_{L^2}^{\frac12}
			\|\p_2^3 \phi \wx\|_{L^2}^{\frac12}\\
			&~\times
			\|\p_t u\wx^{-1}\|_{L^2}^{\frac12}
			\|\p_1(\p_t u\wx^{-1})\|_{L^2}^{\frac12}
			\|\p_2 \phi \wx\|_{L^2}^{\frac12}
			\|\p_2^2 \phi \wx\|_{L^2}^{\frac12}\\
			\lesssim&~\|\p_1 \p_2^3 u\|_{L^2}
			\|\p_1 \p_t u\|_{L^2}
			\|(\p_2 \phi, \p_2^2 \phi,\p_2^3 \phi)\wx\|_{L^2}^2\\
			\le &~\var\|\p_1 \p_2^3 u\|_{L^2}^2
			+C_\var\|\p_1 \p_t u\|_{L^2}^2\|\phi \wx\|_{H^3}^4,
		\end{aligned}
		\deq
		and
		\beq\label{2328}
		\begin{aligned}
			|I_{5,4,2}|
			\lesssim &~\|\p_t u\wx^{-1}\|_{L^\infty}
			\|\p_2^3 \phi \wx\|_{L^2}
			\|\phi \p_2^3 u\|_{L^2}\\
			\lesssim &~(\|\p_t u \wx^{-1}\|_{L^2}^{\frac12}
			\|\p_2 \p_t u \wx^{-1}\|_{L^2}^{\frac12}
			+\|\p_1(\p_t u \wx^{-1})\|_{L^2}^{\frac12}
			\|\p_{12}(\p_t u \wx^{-1})\|_{L^2}^{\frac12})\\
			&~\times\|\p_2^3 \phi \wx\|_{L^2}\|\phi \p_2^3 u\|_{L^2} \\
			\lesssim &~ \|\p_1 \p_t u\|_{L^2}^{\frac12}
			\|\p_{12} \p_t u\|_{L^2}^{\frac12}
			\|\p_2^3 \phi \wx\|_{L^2}\|\phi \p_2^3 u\|_{L^2}\\
			\le &~\sigma_2\|\p_{12} \p_t u\|_{L^2}^2
			+C_{\sigma_2}\|\p_1 \p_t u\|_{L^2}^{\frac23}
			\|\p_2^3 \phi \wx\|_{L^2}^{\frac43}
			\|\phi \p_2^3 u\|_{L^2}^{\frac43}.
		\end{aligned}
		\deq
		The combination of estimates \eqref{2327}
		and \eqref{2328} yields directly
		\beqq
		\begin{aligned}
			|I_{5,4}| \le \var\|\p_1 \p_2^3 u\|_{L^2}^2
			+\sigma_2\|\p_{12} \p_t u\|_{L^2}^2
			+C_{\var, \sigma_2}(1+\e^4).
		\end{aligned}
		\deqq
		The combination of estimates from term $I_{5,1}$ to $I_{5,4}$ gives directly
		\beq\label{2329}
		I_5
		\ge \frac{d}{dt}\frac12\int \phi^2 |\p_2^3 u|^2 dx
		-\var\|\p_1 \p_2^3 u\|_{L^2}^2
		-\sigma_1 \|\phi \p_2^2 \p_t u\|_{L^2}^2
		-\sigma_2\|\p_{12} \p_t u\|_{L^2}^2
		-C_{\var, \sigma_1, \sigma_2}(1+\e^4).
		\deq
		Finally, let us deal with the most difficult term $I_6$.
		Indeed, we can establish the following estimate
		\beq\label{2330}
		|I_{6}|
		\le \var \|\p_1 \p_2^3 u \|_{L^2}^2
		+C_\var (1+\e^7),
		\deq
		which can be found in detail in appendix \ref{claim-estimate}.
		Thus, substituting the estimates \eqref{2329}
		and \eqref{2330} into equality \eqref{2322} and choosing $\var$ small enough, we get
		\beqq
		\frac{d}{dt}\int \phi^2 |\p_2^3 u|^2 dx+\int |\p_{1}\p_{2}^3 u|^2dx
		\le \sigma_1 \|\phi \p_2^2 \p_t u\|_{L^2}^2
		+\sigma_2\|\p_{12} \p_t u\|_{L^2}^2
		+C_{\sigma_1, \sigma_2}(1+\e^7),
		\deqq
		which implies the estimate \eqref{23001}.
		Therefore, we complete the proof of this lemma.
	\end{proof}

	\begin{lemm}\label{lemma24}
		For any smooth solution of equation \eqref{new-ns},
		it holds with $\gamma \ge 2$
		\beq\label{2401}
		\frac{d}{dt}\|\phi \p_t u\|_{L^2}^2+\|\p_1 \p_t u\|_{L^2}^2 \lesssim 1+\e^5.
		\deq
	\end{lemm}
	\begin{proof}
		Taking $\p_t$ operator to the equation $\eqref{new-ns}_2$, multiplying the equation  by $\p_t u$,
		integrating over $\mathbb{R}^2$
		and  integrating by part, it holds
		\beq\label{2402}
		\int \p_t (\phi^2 \p_t u)\cdot \p_t u dx
		+\int \p_t (\phi^2 u\cdot \nabla u)\cdot \p_t u dx
		+\int |\p_1 \p_t u|^2 dx=0,
		\deq
		where we have used the divergence-free condition $\eqref{new-ns}_3$.
		Using the density equation $\eqref{new-ns}_1$, we have
		\beq\label{2403}
		\begin{aligned}
			\int \p_t(\phi^2 \p_t u)\cdot \p_t u dx
			&=\int \phi^2  \p_t^2 u\cdot \p_t u dx
			+\int \p_t (\phi^2) |\p_t u|^2 dx\\
			&=\frac{d}{dt}\frac12\int \phi^2  |\p_t u|^2 dx
			+\frac12\int \p_t (\phi^2) |\p_t u|^2 dx\\
			&=\frac{d}{dt}\frac12\int \phi^2  |\p_t u|^2 dx
			-\int \phi u\cdot \nabla \phi |\p_t u|^2 dx.
		\end{aligned}
		\deq
		Using the H\"{o}lder and Hardy inequalities, the last term can
		be estimated as follows:
		\beq\label{2404}
		\begin{aligned}
			&~\left|\int \phi u\cdot \nabla \phi |\p_t u|^2 dx \right|\\
			\lesssim &~\|\phi \p_t u\|_{L^2}\|\p_t u\wx^{-1}\|_{L^2}^{\frac12}
			\|\p_1(\p_t u\wx^{-1})\|_{L^2}^{\frac12}
			\|\nabla \phi \wx^2\|_{L^2}^{\frac12}
			\|\p_2(\nabla \phi \wx^2)\|_{L^2}^{\frac12}
			\|u\wx^{-1}\|_{L^\infty}\\
			\lesssim &~\|\phi \p_t u\|_{L^2}\|\p_1 \p_t u \|_{L^2}
			\|\p_1 u\|_{L^2}^{\frac12}\|\p_1\p_2 u \|_{L^2}^{\frac12}
			\|(\nabla \phi, \nabla \p_2 \phi)\wx^2\|_{L^2}\\
			\le &~\var \|\p_1 \p_t u \|_{L^2}^2
			+C_\var \|\phi \p_t u\|_{L^2}^2
			\|(\p_1 u, \p_1\p_2 u)\|_{L^2}^2
			\|(\nabla \phi, \nabla \p_2 \phi)\wx^2\|_{L^2}^2.
		\end{aligned}
		\deq
		Substituting the estimate \eqref{2404} into \eqref{2403}, we can get that
		\beq\label{2405}
		\begin{aligned}
			&\int \p_t(\phi^2 \p_t u)\cdot \p_t u dx\\
			\ge &\frac{d}{dt}\frac12\int \phi^2  |\p_t u|^2 dx
			-\var \|\p_1 \p_t u \|_{L^2}^2
			-C_\var \|\phi \p_t u\|_{L^2}^2
			\|(\p_1 u, \p_1\p_2 u)\|_{L^2}^2
			\|(\nabla \phi, \nabla \p_2 \phi)\wx^2\|_{L^2}^2\\
			\ge &\frac{d}{dt}\frac12\int \phi^2  |\p_t u|^2 dx
			-\var \|\p_1 \p_t u \|_{L^2}^2
			-C_\var \e^3.
		\end{aligned}
		\deq
		Finally, we can get that
		\beqq
		\begin{aligned}
			&\int \p_t (\phi^2 u\cdot \nabla u) \cdot \p_t u dx\\
			=&\int \p_t (\phi^2) u\cdot \nabla u \cdot \p_t u dx
			+\int \phi^2 \p_t u\cdot \nabla u  \cdot \p_t u dx
			+\int \phi^2 u\cdot \p_t \nabla u \cdot \p_t u dx\\
			\overset{\text{def}}{=} &~I_{7,1}+I_{7,2}+I_{7,3}.
		\end{aligned}
		\deqq
		Using the density equation $\eqref{new-ns}_1$
		and Hardy inequality, we have
		\beqq
		\begin{aligned}
			|I_{7,1}|
			=&\left|2 \int \phi(u\cdot \nabla \phi)
			(u\cdot \nabla u)\cdot \p_t u dx \right|\\
			\lesssim& \| \p_t u \wx^{-1}\|_{L^2}
			\|\nabla u \wx^{-1}\|_{L^2}^{\frac12}
			\|\p_1(\nabla u \wx^{-1})\|_{L^2}^{\frac12}
			\|\nabla \phi\wx^2\|_{L^2}^{\frac12}
			\|\p_2 (\nabla \phi \wx^2)\|_{L^2}^{\frac12}\\
			 &\times \|u\wx^{-1}\|_{L^\infty}^2 \|\phi \wx^2\|_{L^\infty} \\
			\lesssim& \|\p_1 \p_t u\|_{L^2}
			\|(\p_1 u, \p_1 \p_2 u, \p_1^2 u)\|_{L^2}^3
			\|(\nabla \phi, \p_2 \nabla \phi)\wx^2\|_{L^2} \|\phi \wx^2\|_{H^2}\\
			\le&  \var\|\p_1 \p_t u \|_{L^2}^2
			+C_\var \|(\p_1 u, \p_1 \p_2 u, \p_1^2 u)\|_{L^2}^6
			\|(\nabla \phi, \p_2 \nabla \phi)\wx^2\|_{L^2}^2 \|\phi \wx^2\|_{H^2}^2.
		\end{aligned}
		\deqq
		Using the Sobolev and Hardy inequalities, we have
		\beqq
		\begin{aligned}
			|I_{7,2}|
			&\lesssim \|\phi \p_t u\|_{L^2}
			\|\nabla u \wx^{-1}\|_{L^2}^{\frac{1}{2}}
			\|\p_2(\nabla u \wx^{-1})\|_{L^2}^{\frac{1}{2}}
			\|\p_t u \wx^{-1}\|_{L^2}^{\frac{1}{2}}
			\|\p_1(\p_t u \wx^{-1}) \|_{L^2}^{\frac{1}{2}}
			\|\phi \wx^2\|_{L^\infty}\\
			&\lesssim \|\phi \p_t u\|_{L^2}
			\|(\p_1 u, \p_1 \p_2 u, \p_1 \p_2^2 u)\|_{L^2}
			\|\p_1 \p_t u \|_{L^2}
			\|\phi \wx^2\|_{L^\infty}\\
			&\le \var\|\p_1 \p_t u \|_{L^2}^2
			+C_\var\|\phi \p_t u\|_{L^2}^2
			\|(\p_1 u, \p_1 \p_2 u, \p_1 \p_2^2 u)\|_{L^2}^2
			\|\phi \wx^2\|_{H^2}^2.
		\end{aligned}
		\deqq
		Integrating by part and using
		divergence-free condition $\eqref{new-ns}_1$, we have
		\beqq
		\begin{aligned}
			|I_{7,3}|
			&=\left|\frac12  \int \phi^2 u \cdot \nabla(|\p_t u|^2)dx \right|
			= \left|-\int \phi u \cdot \nabla \phi |\p_t u|^2 dx \right|\\
			&\le \var \|\p_1 \p_t u \|_{L^2}^2
			+C_\var \|\phi \p_t u\|_{L^2}^2
			\|(\p_1 u, \p_1\p_2 u)\|_{L^2}^2
			\|(\nabla \phi, \nabla \p_2 \phi)\wx^2\|_{L^2}^2.
		\end{aligned}
		\deqq
		Thus, the combination of estimates of terms from $I_{7,1}$
		to $I_{7,3}$ yields directly
		\beq\label{2406}
		\left|\int \p_t (\phi^2 u\cdot \nabla u) \cdot \p_t u dx \right|
		\le \var\|\p_1 \p_t u \|_{L^2}^2
		+C_\var (1+\e^5).
		\deq
		Substituting the estimates \eqref{2405}
		and \eqref{2406} into \eqref{2402} and choosing
		$\var$ small enough, we obtain the estimate \eqref{2401}.
		Therefore, we complete the proof of this lemma.
	\end{proof}
	
	\begin{lemm}\label{lemma25}
		For any smooth solution of equation \eqref{new-ns}, it holds with $\gamma \ge 2$
		\beq\label{2501}
		\frac{d}{dt}\|\p_1 \p_2 u\|_{L^2}^2
		+\|\phi \p_2 \p_t u\|_{L^2}^2
		\lesssim 1+\e^3.
		\deq
	\end{lemm}
	\begin{proof}
		Taking $\p_2$ operator to the equation $\eqref{new-ns}_2$, multiplying equation  by $\p_2 \p_t u$,
		integrating over $\mathbb{R}^2$
		and using divergence-free condition $\eqref{new-ns}_3$, it holds
		\beq\label{2502}
		\frac{d}{dt}\frac{1}{2}\int |\p_1 \p_2 u  |^2 dx
		+\int \p_2 (\phi^2 \p_t u) \cdot \p_2 \p_t u dx
		+\int \p_2 (\phi^2 u\cdot \nabla u) \cdot \p_2 \p_t u dx=0.
		\deq
		It is easy to check that
		\beqq
		\begin{aligned}
			\int \p_2 (\phi^2 \p_t u)\cdot \p_2 \p_t u dx
			=\int  \phi^2 |\p_2 \p_t u|^2 dx
			+2\int \phi \p_2 \phi \p_t u \cdot \p_2 \p_t u dx.
		\end{aligned}
		\deqq
		The last term can be estimated as follows:
		\beqq
		\begin{aligned}
			\left|\int \phi \p_2 \phi \p_t u \cdot \p_2 \p_t u dx\right|
			&\lesssim \|\phi \p_2 \p_t u\|_{L^2}
			\|\p_t u \wx^{-1}\|_{L^2}^{\frac12}
			\|\p_1(\p_t u \wx^{-1})\|_{L^2}^{\frac12}
			\|\p_2 \phi \wx\|_{L^2}^{\frac12}
			\|\p_2^2 \phi \wx\|_{L^2}^{\frac12}\\
			&\lesssim \|\phi \p_2 \p_t u\|_{L^2}
			\|\p_1 \p_t u \|_{L^2}\|(\p_2 \phi, \p_2^2 \phi)\wx\|_{L^2}\\
			&\le \var \|\phi \p_2 \p_t u\|_{L^2}^2
			+C_\var \|\p_1 \p_t u \|_{L^2}^2
			\|(\p_2 \phi, \p_2^2 \phi)\wx\|_{L^2}^2.
		\end{aligned}
		\deqq
		Thus, due the the smallness of $\var$, we obtain the estimate
		\beq\label{2503}
		\int \p_2 (\phi^2 \p_t u)\cdot \p_2 \p_t u dx
		\gtrsim \int  \phi^2 |\p_2 \p_t u|^2 dx
		-\|\p_1 \p_t u \|_{L^2}^2
		\|(\p_2 \phi, \p_2^2 \phi)\wx\|_{L^2}^2.
		\deq
		Finally, let us deal with the final term.
		\beqq
		\begin{aligned}
			&\int \p_2(\phi^2 u\cdot \nabla u) \cdot \p_2 \p_t u dx\\
			=&2\int (\p_2 \phi u\cdot \nabla u) \cdot \phi \p_2 \p_t u dx
			+\int (\phi \p_2 u\cdot \nabla u) \cdot \phi \p_2 \p_t u dx\\
			&+\int (\phi u\cdot \nabla \p_2 u) \cdot \phi \p_2 \p_t u dx\\
			\overset{\text{def}}{=} & I_{8,1}+I_{8,2}+I_{8,3}.
		\end{aligned}
		\deqq
		Using the Sobolev inequality and Hardy inequality, it is easy to check that
		\beqq
		\begin{aligned}
			|I_{8,1}|
			&\lesssim \|\phi \p_2 \p_t u\|_{L^2}
			\|\nabla u\wx^{-1}\|_{L^2}^{\frac12}
			\|\p_2 \nabla u\wx^{-1}\|_{L^2}^{\frac12}
			\|\p_2 \phi \wx^2\|_{L^2}^{\frac12}
			\|\p_1(\p_2\phi \wx^2)\|_{L^2}^{\frac12}
			\|u \wx^{-1}\|_{L^\infty}\\
			&\lesssim \|\phi \p_2 \p_t u\|_{L^2}
			\|(\p_1 u, \p_1 \p_2 u, \p_1 \p_2^2 u)\|_{L^2}^2
			\|\phi \wx^2\|_{H^2}\\
			&\le \var\|\phi \p_2 \p_t u\|_{L^2}^2
			+C_\var\|(\p_1 u, \p_1 \p_2 u, \p_1 \p_2^2 u)\|_{L^2}^4
			\|\phi \wx^2\|_{H^2}^2.
		\end{aligned}
		\deqq
		Similarly, it is easy to check that
		\beqq
		\begin{aligned}
			|I_{8,2}|
			&\lesssim \|\phi \p_2 \p_t u\|_{L^2}
			\|\p_2 u\wx^{-1}\|_{L^2}^{\frac12}
			\|\p_1 (\p_2 u\wx^{-1})\|_{L^2}^{\frac12}
			\|\nabla u \wx^{-1}\|_{L^2}^{\frac12}
			\|\p_2 \nabla u \wx^{-1}\|_{L^2}^{\frac12}
			\|\phi \wx^2\|_{L^\infty}\\
			&\lesssim \|\phi \p_2 \p_t u\|_{L^2}
			\|(\p_1 u, \p_1 \p_2 u, \p_1 \p_2^2 u)\|_{L^2}^2
			\|\phi \wx^2\|_{H^2}\\
			&\le \var\|\phi \p_2 \p_t u\|_{L^2}^2
			+C_\var\|(\p_1 u, \p_1 \p_2 u, \p_1 \p_2^2 u)\|_{L^2}^4
			\|\phi \wx^2\|_{H^2}^2,
		\end{aligned}
		\deqq
		and
		\beqq
		\begin{aligned}
			|I_{8,3}|
			&\lesssim \|\phi \p_2 \p_t u\|_{L^2}\|\nabla \p_2 u \wx^{-1}\|_{L^2}
			\|u \wx^{-1}\|_{L^\infty}\|\phi \wx^2\|_{L^\infty}\\
			&\lesssim \|\phi \p_2 \p_t u\|_{L^2}
			\|(\p_1 \p_2 u, \p_1 \p_2^2 u)\|_{L^2}
			\|(\p_1 u, \p_1 \p_2 u)\|_{L^2}
			\|\phi \wx^2\|_{H^2}\\
			&\le \var \|\phi \p_2 \p_t u\|_{L^2}^2
			+C_\var \|(\p_1 u, \p_1 \p_2 u, \p_1 \p_2^2 u)\|_{L^2}^4
			\|\phi \wx^2\|_{H^2}^2.
		\end{aligned}
		\deqq
		The combination of estimates of terms from $I_{8,1}$ to
		$I_{8,3}$ yields directly
		\beq\label{2504}
		\begin{aligned}
			\left|\int \p_2(\phi^2 u\cdot \nabla u) \cdot \p_2 \p_t u dx \right|
			\le  \var \|\phi \p_2 \p_t u\|_{L^2}^2
			+C_\var \e^3.
		\end{aligned}
		\deq
		Substituting the estimates of \eqref{2503} and \eqref{2504}
		into \eqref{2502}, we have
		\beqq
		\frac{d}{dt} \int |\p_2 \p_1 u|^2 dx +
		\int  \phi^2 |\p_2 \p_t u|^2 dx
		\lesssim 1+\e^3.
		\deqq
		Therefore, we complete the proof of this lemma.
	\end{proof}
	
	\begin{lemm}\label{lemma26}
		For any smooth solution of equation \eqref{new-ns},
		it holds with $\gamma \ge 2$
		\beq\label{2601}
		\frac{d}{dt}\|\p_1 \p_t u\|_{L^2}^2
		+\|\phi \p_t^2 u\|_{L^2}^2
		\lesssim   1+\e^4
		+\|\p_{12} \p_t u \|_{L^2}^2
		\|\p_1 u\|_{L^2}^2
		\|(\phi, \p_1 \phi, \p_2 \phi, \p_1 \p_2 \phi) \wx^2\|_{L^2}^2.
		\deq
	\end{lemm}
	\begin{proof}
		Taking $\p_t$ operator to the equation $\eqref{new-ns}_2$, multiplying the equation by $\p_t^2 u$,
		integrating over $\mathbb{R}^2$
		and using divergence-free condition $\eqref{new-ns}_3$, it holds
		\beq\label{2602}
		\frac{d}{dt}\frac{1}{2}\int |\p_1 \p_t u|^2 dx
		+\int \p_t (\phi^2 \p_t u)\cdot \p_t^2 u dx
		+\int \p_t (\phi^2 u\cdot \nabla u)\cdot \p_t^2 u dx=0.
		\deq
		Using the density equation $\eqref{new-ns}_1$, it holds 
		\beqq
		\begin{aligned}
			\int \p_t (\phi^2 \p_t u) \cdot \p_t^2 u dx
			&=\int \phi^2 |\p_t^2 u|^2 dx
			+\int \p_t(\phi^2)\p_t u \cdot \p_t^2 u dx\\
			&=\int \phi^2 |\p_t^2 u|^2 dx
			-2\int \phi(u\cdot \nabla \phi)\p_t u \cdot \p_t^2 u dx.
		\end{aligned}
		\deqq
		The last term can be estimated as follows
		\beqq
		\begin{aligned}
			&\left|\int \phi(u\cdot \nabla \phi)\p_t u \cdot \p_t^2 u dx\right|\\
			\lesssim &~\|\phi \p_t^2 u\|_{L^2}
			\|\p_t u \wx^{-1}\|_{L^2}^{\frac12}
			\|\p_1(\p_t u \wx^{-1})\|_{L^2}^{\frac12}
			\|\nabla \phi \wx^2\|_{L^2}^{\frac12}
			\|\p_2(\nabla \phi \wx^2)\|_{L^2}^{\frac12}
			\|u \wx^{-1}\|_{L^\infty}\\
			\lesssim &~\|\phi \p_t^2 u\|_{L^2}
			\|\p_1 \p_t u \|_{L^2}
			\|(\nabla \phi, \p_2 \nabla \phi)\wx^2\|_{L^2}
			\|(\p_1 u, \p_1 \p_2 u)\|_{L^2}\\
			\le &~\var  |\phi \p_t^2 u\|_{L^2}^2
			+C_\var \|\p_1 \p_t u \|_{L^2}^2
			\|(\nabla \phi, \p_2 \nabla \phi)\wx^2\|_{L^2}^2
			\|(\p_1 u, \p_1 \p_2 u)\|_{L^2}^2.
		\end{aligned}
		\deqq
		Thus, we can get that
		\beq\label{2603}
		\begin{aligned}
			\int \p_t (\phi^2 \p_t u) \cdot \p_t^2 u dx
			&\gtrsim \int \phi^2 |\p_t^2 u|^2 dx
			-\|\p_1 \p_t u \|_{L^2}^2
			\|(\nabla \phi, \p_2 \nabla \phi)\wx^2\|_{L^2}^2
			\|(\p_1 u, \p_1 \p_2 u)\|_{L^2}^2.
		\end{aligned}
		\deq
		Finally, it is easy to check that
		\beqq
		\begin{aligned}
			&\int \p_t (\phi^2 u\cdot \nabla u) \cdot \p_t^2 u dx\\
			=&\int \p_t (\phi^2) (u\cdot \nabla u) \cdot \p_t^2 u dx
			+\int \phi^2 (\p_t u\cdot \nabla u) \cdot \p_t^2 u dx
			+\int \phi^2 (u\cdot \nabla \p_t u) \cdot \p_t^2 u dx\\
			\overset{\text{def}}{=} &I_{9,1}+I_{9,2}+I_{9,3}.
		\end{aligned}
		\deqq
		Using the density equation $\eqref{new-ns}_1$ and Hardy inequality, it holds
		\beqq
		\begin{aligned}
			|I_{9,1}|
			&=\left| -2\int \phi(u\cdot \nabla \phi)(u\cdot \nabla u)\cdot \p_t^2 u dx \right|\\
			&\lesssim \|\phi \p_t^2 u\|_{L^2}
			\|\nabla u \wx^{-1}\|_{L^2}^{\frac12}
			\|\p_1(\nabla u \wx^{-1})\|_{L^2}^{\frac12}
			\|\nabla \phi \wx^3\|_{L^2}^{\frac12}
			\|\p_2 \nabla \phi \wx^3\|_{L^2}^{\frac12}
			\|u\wx^{-1}\|_{L^\infty}^2\\
			&\lesssim \|\phi \p_t^2 u\|_{L^2}\|(\p_1 u, \p_1 \p_2 u, \p_1^2 u )\|_{L^2}^3
			\|(\nabla \phi, \p_2 \nabla \phi)\wx^3\|_{L^2}\\
			&\le \varepsilon \|\phi \p_t^2 u\|_{L^2}^2
			+C_\varepsilon\|(\p_1 u, \p_1 \p_2 u, \p_1^2 u )\|_{L^2}^6
			\|(\nabla \phi, \p_2 \nabla \phi)\wx^3\|_{L^2}^2.
		\end{aligned}
		\deqq
		Using Sobolev and Hardy inequalities, we have
		\beqq
		\begin{aligned}
			|I_{9,2}|
			&\lesssim \|\phi \p_t^2 u\|_{L^2}
			\|\p_t u \wx^{-1}\|_{L^2}^{\frac{1}{2}}
			\|\p_1(\p_t u \wx^{-1})\|_{L^2}^{\frac{1}{2}}
			\|\nabla u \wx^{-1}\|_{L^2}^{\frac{1}{2}}
			\|\p_2(\nabla u \wx^{-1})\|_{L^2}^{\frac{1}{2}}
			\|\phi \wx^2\|_{L^\infty}\\
			&\lesssim \|\phi \p_t^2 u\|_{L^2}\|\p_1 \p_t u \|_{L^2}
			\|(\p_1 u, \p_1 \p_2 u, \p_1 \p_2^2 u)\|_{L^2}
			\|\phi \wx^2\|_{H^2}\\
			&\le  \var\|\phi \p_t^2 u\|_{L^2}^2
			+C_\var \|\p_1 \p_t u \|_{L^2}^2
			\|(\p_1 u, \p_1 \p_2 u, \p_1 \p_2^2 u)\|_{L^2}^2
			\|\phi \wx^2\|_{H^2}^2,
		\end{aligned}
		\deqq
		and
		\beqq
		\begin{aligned}
			|I_{9,3}|
			=&\left|\int \phi^2 (u_1 \p_1 \p_t u) \cdot \p_t^2 u dx
			+\int \phi^2 (u_2 \p_2 \p_t u) \cdot \p_t^2 u dx \right|\\
			\lesssim &~\|\phi \p_t^2 u\|_{L^2}\|\p_1 \p_t u\|_{L^2}
			\|u_1 \wx^{-1}\|_{L^\infty}\|\phi \wx\|_{L^\infty}\\
			&~+\|\phi \p_t^2 u\|_{L^2}
			\|\p_2 \p_t u \wx^{-1}\|_{L^2}^{\frac12}
			\|\p_1 (\p_2 \p_t u \wx^{-1})\|_{L^2}^{\frac12}
			\|u_2 \wx^{-1}\|_{L^2}^{\frac12}
			\|\p_2 (u_2 \wx^{-1})\|_{L^2}^{\frac12}
			\|\phi \wx^2\|_{L^\infty}\\
			\le &~\var\|\phi \p_t^2 u\|_{L^2}^2
			+C_\var \|\p_1 \p_t u\|_{L^2}^2
			\|(\p_1 u, \p_1 \p_2 u)\|_{L^2}^2
			\|\phi \wx^2\|_{H^2}^2\\
			&~+C_\var\|\p_{12} \p_t u \|_{L^2}^2
			\|\p_1 u\|_{L^2}^2
			\|(\phi, \p_1 \phi, \p_2 \phi, \p_1 \p_2 \phi) \wx^2\|_{L^2}^2.
		\end{aligned}
		\deqq
		The combination of estimates from term $I_{9,1}$
		to $I_{9,3}$ yields directly
		\beq\label{2604}
		\bal
		\left|\int \p_t (\phi^2 u\cdot \nabla u) \cdot \p_t^2 u dx \right|
		\le &~
		\varepsilon \|\phi \p_t^2 u\|_{L^2}^2
		+C_\varepsilon( \|\p_{12} \p_t u \|_{L^2}^2
		\|\p_1 u\|_{L^2}^2
		\|(\phi, \p_1 \phi, \p_2 \phi, \p_1 \p_2 \phi) \wx^2\|_{L^2}^2 \\
		&+1+\e^4).
		\dal
		\deq
		Substituting the estimates \eqref{2603} and \eqref{2604}
		into \eqref{2602} and choosing $\var$ small enough,
		we obtain the estimate \eqref{2601}.
		Therefore, we complete the proof of this lemma.
	\end{proof}
	
	\begin{lemm}\label{lemma27}
		For any smooth solution of equation \eqref{new-ns}, it holds with $\gamma \ge2$
		\beq\label{2701}
		\frac{d}{dt}\|\phi \p_2 \p_t u\|_{L^2}^2
		+\|\p_{12} \p_tu\|_{L^2}^2
		\le \sigma_1\|\phi \p_2^2 \p_t u\|_{L^2}^2
		+\sigma_3\|\phi \p_t^2 u\|_{L^2}^2
		+C_{\sigma_1, \sigma_3}(1+\e^4),
		\deq
		where the constants $\sigma_1$  and $\sigma_3$ are the small constants that will be chosen later.
	\end{lemm}
	\begin{proof}
		Taking $\p_2 \p_t$ operator to the equation $\eqref{new-ns}_2$, multiplying the equation by $\p_2 \p_t u$,
		integrating over $\mathbb{R}^2$
		and using divergence-free condition $\eqref{new-ns}_3$, it holds
		\beq\label{2702}
		\underset{I_{10}}{\underbrace{
				\int \p_2 \p_t (\phi^2 \p_t u)\cdot \p_2 \p_t u dx}}
		+\underset{I_{11}}{\underbrace{
				\int \p_2 \p_t (\phi^2 u\cdot \nabla u)\cdot \p_2 \p_t u dx}}
		+\int |\p_1  \p_2 \p_t  u|^2 dx=0.
		\deq
		Integrating by part, it is easy to check that
		\beqq
		\begin{aligned}
			I_{10}
			=&\int \phi^2 \p_2 \p_t^2 u \cdot \p_2 \p_t u dx
			+\int \p_2  (\phi^2) \p_t^2 u \cdot \p_2 \p_t u dx\\
			&+\int \p_t (\phi^2) \p_2 \p_t u \cdot \p_2 \p_t u dx
			+\int \p_2 \p_t (\phi^2) \p_t u \cdot \p_2 \p_t u dx\\
			=&\int \phi^2 \p_2 \p_t^2 u \cdot \p_2 \p_t u dx
			+\int \p_2  (\phi^2) \p_t^2 u \cdot \p_2 \p_t u dx\\
			&-\int  \p_t (\phi^2) \p_t u \cdot \p_2^2 \p_t u dx\\
			\overset{\text{def}}{=} &I_{10,1}+I_{10,2}+I_{10,3}.
		\end{aligned}
		\deqq
		Using the density equation $\eqref{new-ns}_1$, it is easy to check that
		\beqq
		I_{10,1}
		=\frac{d}{dt}\frac12\int \phi^2 |\p_2 \p_t u|^2dx
		-\frac{1}{2}\int |\p_2 \p_t u|^2 \p_t(\phi^2)dx
		=\frac{d}{dt}\frac12\int \phi^2 |\p_2 \p_t u|^2dx
		+\int |\p_2 \p_t u|^2 \phi u \cdot \nabla \phi dx,
		\deqq
		where the last term can be estimated as follows
		\beqq
		\begin{aligned}
			&\left|\int |\p_2 \p_t u|^2 \phi u \cdot \nabla \phi dx\right|\\
			\lesssim &~\|\phi \p_2 \p_t u\|_{L^2}
			\|\p_2 \p_t u \wx^{-1}\|_{L^2}^{\frac12}
			\|\p_1(\p_2 \p_t u \wx^{-1})\|_{L^2}^{\frac12}\\
			&~\times
			\|\nabla \phi \wx^2\|_{L^2}^{\frac12}
			\|\p_2(\nabla \phi \wx^2)\|_{L^2}^{\frac12}
			\|u \wx^{-1}\|_{L^\infty}\\
			\lesssim &~\|\phi \p_2 \p_t u\|_{L^2}\|\p_1 \p_2 \p_t u \|_{L^2}
			\|(\p_1 u, \p_1 \p_2 u)\|_{L^2}
			\|(\nabla \phi, \p_2 \nabla \phi)\wx^2\|_{L^2}\\
			\le &~\varepsilon \|\p_1 \p_2 \p_t u \|_{L^2}^2
			+C_\varepsilon\|\phi \p_2 \p_t u\|_{L^2}^2
			\|(\p_1 u, \p_1 \p_2 u)\|_{L^2}^2
			\|(\nabla \phi, \p_2 \nabla \phi)\wx^2\|_{L^2}^2.
		\end{aligned}
		\deqq
		Thus, we can get that
		\beq\label{2703}
		\begin{aligned}
			I_{10,1}
			\ge \frac{d}{dt}\frac12\int \phi^2 |\p_2 \p_t u|^2dx
			-\varepsilon \|\p_1 \p_2 \p_t u \|_{L^2}^2
			-C_\varepsilon\|\phi \p_2 \p_t u\|_{L^2}^2
			\|(\p_1 u, \p_1 \p_2 u)\|_{L^2}^2
			\|(\nabla \phi, \p_2 \nabla \phi)\wx^2\|_{L^2}^2.
		\end{aligned}
		\deq
		It is easy to get that
		\beq\label{2704}
		\begin{aligned}
			|I_{10,2}|
			=\left|2\int  \frac{\p_2 \phi}{\phi}  \phi \p_t^2 u \cdot \phi \p_2 \p_t u dx \right|
			\lesssim \|\phi \p_t^2 u\|_{L^2}
			\|\p_2 \ln \phi\|_{L^2}^{\frac12}
			\|\p_2^2 \ln \phi\|_{L^2}^{\frac12}
			\|\phi \p_2 \p_t u\|_{L^2}^{\frac12}
			\|\p_1(\phi \p_2 \p_t u)\|_{L^2}^{\frac12}.
		\end{aligned}
		\deq
		It is easy to check that
		\beq\label{2711}
		\begin{aligned}
			\|\p_1(\phi \p_2 \p_t u)\|_{L^2}
			\le &~\|\p_1 \phi \p_2 \p_t u\|_{L^2}
			+\|\phi \p_1 \p_2 \p_t u\|_{L^2}\\
			\lesssim&~\|\p_1 \phi \wx\|_{L^2}^{\frac12}
			\|\p_2 \p_1 \phi \wx\|_{L^2}^{\frac12}
			\|\p_2 \p_t u \wx^{-1}\|_{L^2}^{\frac12}
			\|\p_1(\p_2 \p_t u \wx^{-1})\|_{L^2}^{\frac12}\\
			&~+\|\phi\|_{L^\infty} \|\p_1 \p_2 \p_t u\|_{L^2}\\
			\lesssim &~\|\p_1 \p_2 \p_t u\|_{L^2}
			\|(\p_1 \phi, \p_1 \p_2 \phi)\wx\|_{L^2},
		\end{aligned}
		\deq
		which, together with the estimate \eqref{2704}, yields directly
		\beq\label{2705}
		\begin{aligned}
			|I_{10,2}|
			\lesssim &~\|\phi \p_t^2 u\|_{L^2}
			\|\p_2 \ln \phi\|_{L^2}^{\frac12}
			\|\p_2^2 \ln \phi\|_{L^2}^{\frac12}
			\|\p_1 \p_2 \p_t u\|_{L^2}^{\frac12}
			\|(\p_1 \phi, \p_1 \p_2 \phi)\wx\|_{L^2}^{\frac12} \|\phi \p_2 \p_t u\|_{L^2}^{\frac12}\\
			\le &~\var\|\p_1 \p_2 \p_t u\|_{L^2}^2
			+\sigma_3\|\phi \p_t^2 u\|_{L^2}^2
			+C_{\var,\sigma_3}\|(\p_2 \ln \phi, \p_2^2 \ln \phi)\|_{L^2}^4
			\|(\p_1 \phi, \p_1 \p_2 \phi)\wx\|_{L^2}^2 \|\phi \p_2 \p_t u\|_{L^2}^{2}.
		\end{aligned}
		\deq
		Using the density equation $\eqref{new-ns}_1$ and Hardy inequality, we have
		\beq\label{2706}
		\begin{aligned}
			|I_{10,3}|
			=&~\left|2\int \phi(u\cdot \nabla \phi) \p_t u \cdot \p_2^2 \p_t u dx\right|\\
			\lesssim &~\|u \wx^{-1}\|_{L^\infty}\|\phi \p_2^2 \p_t u \|_{L^2}
			\|\p_t u \wx^{-1}\|_{L^2}^{\frac12}
			\|\p_1(\p_t u \wx^{-1})\|_{L^2}^{\frac12}\\
			&~\times
			\|\nabla \phi \wx^2\|_{L^2}^{\frac12}
			\|\p_2\nabla \phi \wx^2\|_{L^2}^{\frac12}\\
			\lesssim &~\|\phi \p_2^2 \p_t u \|_{L^2}\|\p_1 \p_t u\|_{L^2}
			\|(\p_1 u, \p_{12} u)\|_{L^2}\|(\nabla \phi,  \p_2 \nabla \phi)\wx^2\|_{L^2}\\
			\le &~\sigma_1\|\phi \p_2^2 \p_t u \|_{L^2}^2
			+C_{\sigma_1}\|\p_1 \p_t u\|_{L^2}^2
			\|(\p_1 u, \p_{12} u)\|_{L^2}^2
			\|(\nabla \phi,  \p_2 \nabla \phi)\wx^2\|_{L^2}^2.
		\end{aligned}
		\deq
		The combination of estimates \eqref{2703}, \eqref{2705}
		and \eqref{2706} yields directly
		\beq\label{2707}
		\begin{aligned}
			I_{10}
			\ge \frac{d}{dt}\frac12\int \phi^2 |\p_2 \p_t u|^2dx
			-\varepsilon \|\p_1 \p_2 \p_t u \|_{L^2}^2
			-\sigma_1\|\phi \p_2^2 \p_t u \|_{L^2}^2
			-\sigma_3\|\phi \p_t^2 u\|_{L^2}^2
			-C_{\varepsilon,\sigma_1,\sigma_3}\e^4.
		\end{aligned}
		\deq
		Using the Sobolev and Hardy inequalities, we can deduce that
		\beq\label{2710}
		|I_{11}|\le \var\|\p_{12}\p_t u \|_{L^2}^2
		+\sigma_1\|\phi \p_2^2 \p_t u\|_{L^2}^2
		+C_{\var,  \sigma_1}(1+\e^4),
		\deq
		which can be found in detail in appendix \ref{claim-estimate}.
		Substituting the estimates \eqref{2707}  and
		\eqref{2710} into \eqref{2702} and choosing
		$\var$ small enough, we have
		\beqq
		\frac{d}{dt}\int \phi^2 |\p_2 \p_t u|^2dx
		+\int |\p_1  \p_2 \p_t  u|^2 dx
		\le
		\sigma_1\|\phi \p_2^2 \p_t u\|_{L^2}^2
		+\sigma_3\|\phi \p_t^2 u\|_{L^2}^2
		+C_{\sigma_1, \sigma_3}(1+\e^4).
		\deqq
		Therefore, we complete the proof of this lemma.
	\end{proof}

	\begin{lemm}\label{lemma28}
		For any smooth solution of equation \eqref{new-ns}, it holds with $\gamma \geq 2$
		\beq\label{2801}
		\begin{aligned}
			&~\frac{d}{dt}\|\p_2^2 \p_1  u\|_{L^2}^2
			+\|\phi \p_t \p_2^2 u\|_{L^2}^2\\
			\le &~\sigma_2\|\p_{12} \p_t u\|_{L^2}^2
			+C_{\sigma_2} (1+\e^4)
			+C \|\p_1 \p_2^3 u \|_{L^2}^2
			\|(\p_1 u, \p_{12} u)\|_{L^2}^2
			\|(\phi, \p_1 \phi, \p_2 \phi, \p_{12}\phi)\wx\|_{L^2}^2,
		\end{aligned}
		\deq
		where $\sigma_2$ is a small constant that will be chosen later.
	\end{lemm}
	
	\begin{proof}
		Taking $\p_2^2$ operator to the equation $\eqref{new-ns}_2$, multiplying the equation by $\p_t \p_2^2 u$ and
		integrating over $\mathbb{R}^2$, it holds that
		\beq\label{2802}
		\frac{d}{dt}\frac{1}{2}\int |\p_2^2 \p_1  u|^2 dx
		+\underset{I_{12}}{\underbrace{
				\int \p_2^2 (\phi^2 \p_t u) \cdot \p_t  \p_2^2 u dx}}
		+\underset{I_{13}}{\underbrace{
				\int \p_2^2 (\phi^2 u\cdot \nabla u)\cdot \p_t  \p_2^2 u dx}}=0.
		\deq
		It is easy to check that
		\beq\label{2803}
		\begin{aligned}
			I_{12}
			=&\int \phi^2 |\p_t \p_2^2 u|^2 dx
			+4\int \phi \p_2 \phi \p_t \p_2 u \cdot \p_t \p_2^2 udx\\
			&+2\int \phi \p_2^2 \phi \p_t u \cdot \p_t \p_2^2 udx
			+2\int (\p_2 \phi)^2 \p_t u \cdot \p_t \p_2^2 u dx\\
			\overset{\text{def}}{=} &\int \phi^2 |\p_t \p_2^2 u|^2 dx+I_{12,1}+I_{12,2}+I_{12,3}.
		\end{aligned}
		\deq
		Using the Sobolev inequality, Hardy inequality
		and estimate \eqref{2711}, it holds 
		\beqq
		\begin{aligned}
			|I_{12,1}|
			&\lesssim \|\phi\p_t \p_2^2 u\|_{L^2}
			\|\phi \p_t \p_2 u \|_{L^2}^{\frac12}
			\|\p_1(\phi \p_t \p_2 u)\|_{L^2}^{\frac12}
			\|\p_2 \ln \phi\|_{L^2}^{\frac12}
			\|\p_2^2 \ln \phi\|_{L^2}^{\frac12}\\
			&\lesssim \|\phi\p_t \p_2^2 u\|_{L^2}
			\|\phi \p_t \p_2 u \|_{L^2}^{\frac12}
			\|\p_1 \p_2 \p_t u\|_{L^2}^{\frac12}
			\|(\p_1 \phi, \p_1 \p_2 \phi)\wx\|_{L^2}^{\frac12}
			\|(\p_2 \ln \phi, \p_2^2 \ln \phi)\|_{L^2}\\
			&\le \var \|\phi\p_t \p_2^2 u\|_{L^2}^2
			+\sigma_2\|\p_1 \p_2 \p_t u\|_{L^2}^2
			+C_{\var,\sigma_2}\|\phi \p_t \p_2 u \|_{L^2}^2
			\|\phi \wx\|_{H^2}^2
			\|(\p_2 \ln \phi, \p_2^2 \ln \phi)\|_{L^2}^4.
		\end{aligned}
		\deqq
		Similarly, it is easy to check that
		\beqq
		\begin{aligned}
			|I_{12,2}|
			&\lesssim
			\|\phi\p_t \p_2^2 u\|_{L^2}
			\|\p_2^2  \phi \wx\|_{L^2}^{\frac12}
			\|\p_2^3 \phi \wx\|_{L^2}^{\frac12}
			\|\p_t u \wx^{-1}\|_{L^2}^{\frac12}
			\|\p_1(\p_t u \wx^{-1})\|_{L^2}^{\frac12}\\
			&\lesssim \|\phi\p_t \p_2^2 u\|_{L^2}
			\|\p_1 \p_t u \|_{L^2}
			\|(\p_2^2 \phi, \p_2^3 \phi)\wx\|_{L^2}\\
			&\le \var  \|\phi\p_t \p_2^2 u\|_{L^2}^2
			+C_{\var}\|\p_1 \p_t u \|_{L^2}^2
			\|(\p_2^2 \phi, \p_2^3 \phi)\wx\|_{L^2}^2,
		\end{aligned}
		\deqq
		and
		\beqq
		\begin{aligned}
			|I_{12,3}|
			&\lesssim \|\phi \p_t \p_2^2 u\|_{L^2}
			\|\p_t u \wx^{-1}\|_{L^2}^{\frac12}
			\|\p_1(\p_t u \wx^{-1})\|_{L^2}^{\frac12}
			\|\p_2\ln \phi\|_{L^2}^{\frac12}
			\|\p_2^2 \ln \phi\|_{L^2}^{\frac12}
			\|\p_2 \phi \wx\|_{L^\infty}\\
			&\lesssim \|\phi \p_t \p_2^2 u\|_{L^2}
			\|\p_1 \p_t u\|_{L^2}
			\|(\p_2\ln \phi, \p_2^2 \ln \phi)\|_{L^2}
			\|\phi \wx\|_{H^3}\\
			&\le \var\|\phi \p_t \p_2^2 u\|_{L^2}^2
			+C_\var\|\p_1 \p_t u\|_{L^2}^2
			\|(\p_2\ln \phi, \p_2^2 \ln \phi)\|_{L^2}^2
			\|\phi \wx\|_{H^3}^2.
		\end{aligned}
		\deqq
		Substituting the estimates from term $I_{12,1}$
		to $I_{12,3}$ into \eqref{2803} and choosing
		$\var$ small enough, we have
		\beq\label{2804}
		I_{12}\ge
		\|\phi \p_t \p_2^2 u\|_{L^2}^2-\sigma_2\|\p_{12} \p_t u\|_{L^2}^2-C_{\sigma_2}(1+\e^4).
		\deq
		Finally, we can estimate term $I_{13}$ as follows:
		\beq\label{2810}
		|I_{13}|
		\le \var\|\phi \p_t \p_2^2 u\|_{L^2}^2
		+C_\var (1+\e^4)
		+C_\var \|\p_1 \p_2^3 u \|_{L^2}^2
		\|(\p_1 u, \p_{12} u)\|_{L^2}^2
		\|(\phi, \p_1 \phi, \p_2 \phi, \p_{12}\phi)\wx\|_{L^2}^2,
		\deq
		which can be found in detail in appendix \ref{claim-estimate}.
		Substituting the estimates \eqref{2804}
		and \eqref{2810} into \eqref{2802} and
		choosing $\var$ small enough, we have
		\beqq
		\begin{aligned}
			&\frac{d}{dt}\|\p_2^2 \p_1  u\|_{L^2}^2
			+\|\phi \p_t \p_2^2 u\|_{L^2}^2\\
			\le &~
			\sigma_2\|\p_{12} \p_t u\|_{L^2}^2
			+C_{\sigma_2} (1+\e^4)
			+C \|\p_1 \p_2^3 u \|_{L^2}^2
			\|(\p_1 u, \p_{12} u)\|_{L^2}^2
			\|(\phi, \p_1 \phi, \p_2 \phi, \p_{12}\phi)\wx\|_{L^2}^2.
		\end{aligned}
		\deqq
		Therefore, we complete the proof of this lemma.
	\end{proof}
	
	\subsection{Estimate for the density}\label{estimate-rho}
	
	In this subsection, we will establish the estimate for the density.

	\begin{lemm}\label{lemma29}
		For any smooth solution of equation $\eqref{new-ns}$, it holds with $\gamma \geq 2$
		\beq\label{2901}
		\frac{d}{dt}\|\phi \wx^{\gm+1}\|_{L^2}^2
		\lesssim  1+\e^2.
		\deq
	\end{lemm}
	\begin{proof}
		The density equation $\eqref{new-ns}_1$ yields directly
		\beq\label{2902}
		\frac{d}{dt}\frac12\int \phi^2 \wx^{2\gm+2}dx
		+\int (u_1 \p_1 \phi+u_2 \p_2 \phi)\phi \wx^{2\gm+2}dx=0.
		\deq
		Integrating by part and using the divergence-free
		condition $\eqref{new-ns}_3$, we have
		\beq\label{2903}
		\begin{aligned}
			&\left|\int (u_1 \p_1 \phi+u_2 \p_2 \phi)\phi \wx^{2\gm+2}dx \right|\\
			=& \left|-\frac12\int (\p_1 u_1+\p_2 u_2)\phi^2 \wx^{2\gm+2}dx
			-(\gm+1)\int  u_1 \phi^2 \wx^{2\gm+1}dx \right|\\
			=& \left| -(\gm+1)\int  u_1 \phi^2 \wx^{2\gm+1}dx \right| \\
			\lesssim &~\|u_1 \wx^{-1}\|_{L^\infty}\|\phi \wx^{\gm+1}\|_{L^2}^2\\
			\lesssim &~\|(\p_1 u_1, \p_{12} u_1)\|_{L^2}\|\phi \wx^{\gm+1}\|_{L^2}^2.
		\end{aligned}
		\deq
		Substituting estimate \eqref{2903} into \eqref{2902}, we have
		\beqq
		\frac{d}{dt}\|\phi \wx^{\gm+1}\|_{L^2}^2
		\lesssim  \|(\p_1 u_1, \p_{12} u_1)\|_{L^2}\|\phi \wx^{\gm+1}\|_{L^2}^2.
		\deqq
		Therefore, we complete the proof of this lemma.
	\end{proof}
	
	Before we establish the estimate for the derivative of density,
	we need to investigate the estimate for the vorticity as follows.
	
	\begin{lemm}\label{lemma210}
		For any smooth solution of equation \eqref{new-ns}, it holds with $\gamma \ge 2$
		\beq\label{21001}
		\begin{aligned}
			&\|\p_1^2 u\|_{L^2}
			\lesssim 1+\e^4;\\
			&\|\p_1^2 \p_2 u\|_{L^2} \lesssim 1+\e^2;\\
			&\|\p_1^3 u\|_{L^2} \lesssim 1+\e^6;\\
			&\|\p_1^2 w \wx^3\|_{L^2}\lesssim  1+\e^6;\\
			&\|\p_1^2 \p_2 w \wx^2\|_{L^2}\le  \sigma_1 \|\phi \p_t \p_2^2 u_1\|_{L^2}^2 +\sigma_2 \|\p_t \p_{12} u_2\|_{L^2}^2
			+C_{\sigma_1,\sigma_2}(1+\e^9),\\
		\end{aligned}
		\deq
		where the vorticity $w\overset{\text{def}}{=} \p_1 u_2-\p_2 u_1$, $\sigma_1$ and $\sigma_2$ are the small constants that will be chosen later.
	\end{lemm}
	\begin{proof}
		\textbf{Step 1}: The equation $\eqref{new-ns}_2$
		and divergence-free condition $\eqref{new-ns}_3$ yields
		\beqq
		\int |\p_1^2 u|^2 dx
		=\int(\phi^2 \p_t u+\phi^2 u\cdot \nabla u)\cdot \p_1^2 u dx,
		\deqq
		which, together with Cauchy inequality, yields directly
		\beq\label{21002}
		\|\p_1^2 u\|_{L^2}
		\lesssim \|\phi^2 \p_t u\|_{L^2}+\|\phi^2 u\cdot \nabla u\|_{L^2}.
		\deq
		Using the Sobolev and Hardy inequalities, we have
		\beq\label{21003}
		\begin{aligned}
			\|\phi^2 u\cdot \nabla u\|_{L^2}
			&\lesssim \|\nabla u\wx^{-1}\|_{L^2}^{\frac12}
			\|\p_1(\nabla u\wx^{-1})\|_{L^2}^{\frac12}
			\|u\wx^{-1}\|_{L^2}^{\frac12}
			\|\p_2 u\wx^{-1}\|_{L^2}^{\frac12}
			\|\phi \wx\|_{L^\infty}^2\\
			&\lesssim \|(\p_1 u, \p_{12} u)\|_{L^2}^{\frac12}
			(\|\p_1^2 u\|_{L^2}^{\frac12}
			+\|(\p_1 u,\p_{12} u) \|_{L^2}^{\frac12})
			\|(\p_1 u, \p_{12} u)\|_{L^2}
			\|\phi \wx\|_{H^2}^2\\
			&\le \var\|\p_1^2 u\|_{L^2}
			+C_\var \|(\p_1 u, \p_{12} u)\|_{L^2}^3\|\phi \wx\|_{H^2}^4
			+\|(\p_1 u, \p_{12} u)\|_{L^2}^2\|\phi \wx\|_{H^2}^2,
		\end{aligned}
		\deq
		and
		\beq\label{21004}
		\|\phi^2 \p_t u\|_{L^2}
		\le \|\phi\|_{L^\infty}\|\phi \p_t u\|_{L^2}
		\lesssim  \|\phi\|_{H^2}\|\phi \p_t u\|_{L^2}.
		\deq
		Thus, substituting the estimates \eqref{21003} and \eqref{21004}
		into \eqref{21002} and choosing $\var$ small enough, one arrives at
		\beq\label{210021}
		\|\p_1^2 u\|_{L^2}
		\lesssim \|(\p_1 u, \p_{12} u)\|_{L^2}^3\|\phi \wx\|_{H^2}^4
		+\|(\p_1 u, \p_{12} u)\|_{L^2}^2\|\phi \wx\|_{H^2}^2
		+\|\phi\|_{H^2}\|\phi \p_t u\|_{L^2}
		\lesssim 1+\e^4,
		\deq
		which implies the estimate $\eqref{21001}_1$.
		
		\textbf{Step 2}: Similar to estimate \eqref{21002}, it is easy to check that
		\beq\label{21005}
		\|\p_1^2 \p_2 u\|_{L^2}
		\lesssim \|\p_2(\phi^2 \p_t u)\|_{L^2}
		+\|\p_2(\phi^2 u \cdot \nabla u)\|_{L^2}.
		\deq
		It is easy to obtain
		\beq\label{21006}
		\begin{aligned}
			\|\p_2(\phi^2 \p_t u)\|_{L^2}
			&\lesssim \|\phi \p_t u\|_{L^2}\|\p_2 \phi\|_{L^\infty}
			+\|\phi \p_t \p_2 u\|_{L^2}\|\phi\|_{L^\infty}\\
			&\lesssim\|\phi \p_t u\|_{L^2}
			\|\p_2 \phi\|_{H^2}
			+\|\phi \p_t \p_2 u\|_{L^2}
			\|\phi\|_{H^2},
		\end{aligned}
		\deq
		and
		\beq\label{21007}
		\begin{aligned}
			&\|\p_2(\phi^2 u\cdot \nabla u)\|_{L^2}\\
			\lesssim&~\|\phi \p_2 \phi u \cdot \nabla u\|_{L^2}
			+\|\phi^2 \p_2 u \cdot \nabla u\|_{L^2}
			+\|\phi^2 u \cdot \nabla \p_2 u\|_{L^2}\\
			\lesssim&~\|\nabla u \wx^{-1}\|_{L^2}^{\frac12}
			\|\p_2(\nabla u \wx^{-1})\|_{L^2}^{\frac12}
			\|\p_2 \phi \wx\|_{L^2}^{\frac12}
			\|\p_1 (\p_2 \phi \wx)\|_{L^2}^{\frac12}
			\|\phi \wx\|_{L^\infty}\|u \wx^{-1}\|_{L^\infty}\\
			&~+\|\nabla u \wx^{-1}\|_{L^2}^{\frac12}
			\|\p_2 \nabla u \wx^{-1}\|_{L^2}^{\frac12}
			\|\p_2 u \wx^{-1}\|_{L^2}^{\frac12}
			\|\p_1(\p_2 u \wx^{-1})\|_{L^2}^{\frac12}
			\|\phi \wx\|_{L^\infty}^2\\
			&~+\|\nabla \p_2 u \wx^{-1}\|_{L^2}
			\|u \wx^{-1}\|_{L^\infty}
			\|\phi \wx\|_{L^\infty}^2\\
			\lesssim&~\|(\p_1 u, \p_{12}u, \p_{122}u)\|_{L^2}^2
			\|\phi \wx\|_{H^2}^2.
		\end{aligned}
		\deq
		Substituting the estimates \eqref{21006}
		and \eqref{21007} into \eqref{21005}, we can obtain
		\beq\label{210022}
		\|\p_1^2 \p_2 u\|_{L^2}
		\lesssim \|\phi \p_t u\|_{L^2}\|\p_2 \phi\|_{H^2}
		+\|\phi \p_t \p_2 u\|_{L^2}\|\phi\|_{H^2}
		+\|(\p_1 u, \p_{12}u, \p_{122}u)\|_{L^2}^2
		\|\phi \wx\|_{H^2}^2
		\lesssim 1+\e^2,
		\deq
		which implies the estimate $\eqref{21001}_2$.
		Similarly, we can obtain
		\beq\label{21005-a}
		\|\p_1^3 u\|_{L^2}
		\lesssim \|\p_1(\phi^2 \p_t u)\|_{L^2}
		+\|\p_1(\phi^2 u \cdot \nabla u)\|_{L^2}.
		\deq
		It is easy to obtain
		\beq\label{21006-a}
		\begin{aligned}
			\|\p_1(\phi^2 \p_t u)\|_{L^2}
			&\lesssim\|\phi \p_t u\|_{L^2}\|\p_1 \phi\|_{L^\infty}
			+\|\p_t \p_1 u\|_{L^2}\|\phi\|_{L^\infty}^2\\
			&\lesssim \|\phi \p_t u\|_{L^2}
			\|\p_1 \phi\|_{H^2}
			+\|\p_t \p_1 u\|_{L^2}
			\|\phi\|_{H^2}^2,
		\end{aligned}
		\deq
		and
		\beq\label{21007-a}
		\begin{aligned}
			&\|\p_1(\phi^2 u\cdot \nabla u)\|_{L^2}\\
			\lesssim&~\|\phi \p_1 \phi u \cdot \nabla u\|_{L^2}
			+\|\phi^2 \p_1 u \cdot \nabla u\|_{L^2}
			+\|\phi^2 u \cdot \nabla \p_1 u\|_{L^2}\\
			\lesssim&~\|\nabla u \wx^{-1}\|_{L^2}^{\frac12}
			\|\p_2(\nabla u \wx^{-1})\|_{L^2}^{\frac12}
			\|\p_1 \phi \wx\|_{L^2}^{\frac12}
			\|\p_1 (\p_1 \phi \wx)\|_{L^2}^{\frac12}
			\|\phi \wx\|_{L^\infty}\|u \wx^{-1}\|_{L^\infty}\\
			&~+\|\nabla u \wx^{-1}\|_{L^2}^{\frac12}
			\|\p_2 \nabla u \wx^{-1}\|_{L^2}^{\frac12}
			\|\p_1 u \|_{L^2}^{\frac12}
			\|\p_{11} u \|_{L^2}^{\frac12}
			\|\phi \wx\|_{L^\infty}^2\\
			&~+\|\nabla \p_1 u \wx^{-1}\|_{L^2}
			\|u \wx^{-1}\|_{L^\infty}
			\|\phi \wx\|_{L^\infty}^2\\
			\lesssim&~\|(\p_1 u, \p_{12}u, \p_{122}u)\|_{L^2}^2
			\|\phi \wx\|_{H^2}^2
			+\|(\p_1 u, \p_{11}u, \p_{12}u)\|_{L^2}
			\|(\p_1 u, \p_{12}u, \p_{122}u)\|_{L^2}
			\|\phi \wx\|_{H^2}^2.
		\end{aligned}
		\deq
		Substituting the estimates \eqref{21006-a} and \eqref{21007-a}
		into \eqref{21005-a} and using estimate \eqref{21001}, we have
		\beqq
		\|\p_1^3 u\|_{L^2}
		\lesssim 1+\e^6,
		\deqq
		which implies the estimate $\eqref{21001}_3$.
		
		\textbf{Step 3}:
		Taking vorticity operator to the velocity equation
		$\eqref{new-ns}_2$, we have
		\beq\label{21008}
		\p_1^2 w=\nabla \times (\phi^2 \p_t u +\phi^2 u \cdot \nabla u),
		\deq
		which yields directly
		\beq\label{21009}
		\|\p_1^2 w \wx^3\|_{L^2}
		\le \|\nabla \times (\phi^2 \p_t u )\wx^3\|_{L^2}
		+\|\nabla \times (\phi^2 u \cdot \nabla u)\wx^3\|_{L^2}.
		\deq
		Obviously, it holds
		\beq\label{210010}
		\nabla \times (\phi^2 \p_t u )
		=2\phi \p_1 \phi \p_t u_2-2\phi \p_2 \phi \p_t u_1
		+\phi^2 \p_t \p_1 u_2-\phi^2 \p_t \p_2 u_1,
		\deq
		which, together with the Sobolev and Hardy inequalities, yields directly
		\beq\label{210011}
		\begin{aligned}
			&\|\nabla \times (\phi^2 \p_t u )\wx^3\|_{L^2}\\
			\lesssim &~\|\p_t u \wx^{-1}\|_{L^2}^{\frac12}
			\|\p_1(\p_t u \wx^{-1})\|_{L^2}^{\frac12}
			\|\nabla \phi \wx^2\|_{L^2}^{\frac12}
			\|\p_2\nabla \phi \wx^2\|_{L^2}^{\frac12}
			\|\phi\wx^2\|_{L^\infty}\\
			&~+\|\p_t \p_1 u_2 \|_{L^2}\|\phi \wx^2\|_{L^\infty}^2
			+\|\phi \p_t \p_2 u_1\|_{L^2}\|\phi \wx^3\|_{L^\infty}\\
			\lesssim&~\|\p_1 \p_t u \|_{L^2}\| \phi \wx^2\|_{H^2}^2
			+\|\phi \p_t \p_2 u_1\|_{L^2}\|\phi \wx^3\|_{H^2}.
		\end{aligned}
		\deq
		Obviously, it holds
		\beq\label{210012}
		\begin{aligned}
			\nabla \times(\phi^2 u \cdot \nabla u)
			=&~\p_1(\phi^2)u \cdot \nabla u_2-\p_2(\phi^2)u \cdot \nabla u_1
			+\phi^2 \p_1 u \cdot \nabla u_2\\
			&~-\phi^2 \p_2 u \cdot \nabla u_1
			+\phi^2 u\cdot \nabla \p_1 u_2-\phi^2 u\cdot \nabla \p_2 u_1,
		\end{aligned}
		\deq
		which, together with the Sobolev and Hardy inequalities, yields directly
		\beq\label{210013}
		\begin{aligned}
			&\|\nabla \times(\phi^2 u \cdot \nabla u)\wx^3\|_{L^2}\\
			\lesssim &~\|\nabla u \wx^{-1}\|_{L^2}^{\frac12}
			\|\p_1(\nabla u \wx^{-1})\|_{L^2}^{\frac12}
			\|\nabla \phi \wx^2\|_{L^2}^{\frac12}
			\|\p_2 \nabla \phi \wx^2\|_{L^2}^{\frac12}
			\|u \wx^{-1}\|_{L^\infty}
			\|\phi \wx^3\|_{L^\infty}\\
			&~+\|\nabla u \wx^{-1}\|_{L^2}
			\|\p_1(\nabla u \wx^{-1})\|_{L^2}^{\frac12}
			\|\p_2(\nabla u \wx^{-1})\|_{L^2}^{\frac12}
			\|\phi\wx^3\|_{L^\infty}^2\\
			&~+\|\nabla^2 u \wx^{-1}\|_{L^2}
			\|u \wx^{-1}\|_{L^\infty}
			\|\phi \wx^3\|_{L^\infty}^2\\
			\lesssim &~\|(\p_1 u, \p_{12} u, \p_{122}u)\|_{L^2}^{\frac32}
			(\|\p_{12} u\|_{L^2}^{\frac12}+\|\p_{11} u\|_{L^2}^{\frac12})
			\|\phi \wx^3\|_{H^2}^2\\
			&~+\|(\p_1^2 u, \p_{12} u, \p_{122}u)\|_{L^2}
			\|(\p_1 u, \p_{12} u)\|_{L^2}\|\phi \wx^3\|_{H^2}^2.
		\end{aligned}
		\deq
		Substituting the estimates \eqref{210011}
		and \eqref{210013} into \eqref{21009}, we can get
		\beq\label{210013a}
		\begin{aligned}
			\|\p_1^2 w \wx^3\|_{L^2}
			\lesssim  &~\|\p_1 \p_t u \|_{L^2}\| \phi \wx^2\|_{H^2}^2
			+\|\phi \p_t \p_2 u_1\|_{L^2}\|\phi \wx^3\|_{H^2}\\
			&~+\|(\p_1 u, \p_{12} u, \p_{122}u)\|_{L^2}^{\frac32}
			(\|\p_{12} u\|_{L^2}^{\frac12}+\|\p_{11} u\|_{L^2}^{\frac12})
			\|\phi \wx^3\|_{H^2}^2\\
			&~+\|(\p_1^2 u, \p_{12} u, \p_{122}u)\|_{L^2}
			\|(\p_1 u, \p_{12} u)\|_{L^2}\|\phi \wx^3\|_{H^2}^2,
		\end{aligned}
		\deq
		which, together with estimate $\eqref{21001}_1$
		for the quantity $\p_1^2 u$, yields directly
		\beqq
		\|\p_1^2 w \wx^3\|_{L^2}
		\lesssim  1+\e^6,
		\deqq
		which implies the estimate $\eqref{21001}_4$.
		
		\textbf{Step 4}: The vorticity equation yields directly
		\beq\label{210014}
		\|\p_2 \p_1^2 w \wx^2\|_{L^2}
		\le \|\nabla \times \p_2(\phi^2 \p_t u )\wx^2\|_{L^2}
		+\|\nabla \times \p_2(\phi^2 u \cdot \nabla u)\wx^2\|_{L^2}.
		\deq
		Similar to estimates \eqref{210011} and \eqref{210013}, we can obtain
		\beq\label{210018}
		\|\nabla \times \p_2(\phi^2 \p_t u )\wx^2\|_{L^2}\\
		\le \sigma_1 \|\phi \p_t \p_2^2 u_1\|_{L^2}^2 +\sigma_2 \|\p_t \p_{12} u_2\|_{L^2}^2
		+C_{\sigma_1,\sigma_2}(1+\e^2),
		\deq
		and
		\beq\label{210019}
		\|\nabla \times \p_2(\phi^2 u \cdot \nabla u)\wx^2\|_{L^2}
		\lesssim 1+\e^9,
		\deq
		which can be found in detail in appendix \ref{claim-estimate}.
		Substituting the estimates \eqref{210018} and \eqref{210019}
		into \eqref{210014}, we have
		\beqq
		\|\p_2 \p_1^2 w \wx^2\|_{L^2}
		\le  \sigma_1 \|\phi \p_t \p_2^2 u_1\|_{L^2}^2 +\sigma_2 \|\p_t \p_{12} u_2\|_{L^2}^2
		+C_{\sigma_1,\sigma_2}(1+\e^9),
		\deqq
		which implies the estimate $\eqref{21001}_5$.
		Therefore, we complete the proof of this lemma.
	\end{proof}

	\begin{lemm}\label{lemma211}
		For any smooth solution of equation $\eqref{new-ns}$, it holds  with $\gamma \geq 2$
		\beq\label{21101a}
		\frac{d}{dt}\|\p_x^\alpha \phi \wx^{\gm+1}\|_{L^2}^2
		\lesssim 1+\e^7,
		\deq
		where $|\alpha|=1, \alpha=(2, 0)$ or $\alpha=(1, 1)$; and we also have
		\beq\label{21101b}
		\frac{d}{dt} \|\p_2^2 \phi \wx^{\gm+1}\|_{L^2}^2
		\le  \sigma_1 \|\phi \p_t \p_2^2 u_1\|_{L^2}^2 +\sigma_2 \|\p_t \p_{12} u_2\|_{L^2}^2
		+C_{\sigma_1,\sigma_2}(1+\e^9),
		\deq
		where $\sigma_1$ and $\sigma_2$ are the small constants that will be chosen later.
	\end{lemm}
	\begin{proof}
		\textbf{Step 1}: Deal with the case $|\alpha|=1$.
		Then, the equation $\eqref{new-ns}_1$ yields directly
		\beq\label{21102}
		\begin{aligned}
			&\frac{d}{dt}\frac{1}{2}\int |\p_x^\alpha \phi|^2 \wx^{2\gm+2}dx
			+\underset{J_{2,1}}{\underbrace{
					\int \p_x^\alpha u_1 \p_1 \phi \cdot \p_x^\alpha \phi \wx^{2\gm+2}dx}}\\
			&+
			\underset{J_{2,2}}{\underbrace{
					\int \p_x^\alpha u_2 \p_2 \phi \cdot \p_x^\alpha \phi \wx^{2\gm+2}dx}}
			-\underset{J_{2,3}}{\underbrace{
					(\gm+1)\int u_1 | \p_x^\alpha \phi|^2 \wx^{2\gm+1}dx}}=0.
		\end{aligned}
		\deq
		For the case $\alpha=(1,0)$, we apply the divergence-free condition to obtain
		\beqq
		|J_{2,1}|
		\lesssim \|\p_1 u_1\|_{L^\infty}
		\|\p_1 \phi \wx^{\gm+1}\|_{L^2}^2,
		\deqq
		which together with the estimate
		\beq\label{21104}
		\|\p_1 u_1\|_{L^\infty}
		\lesssim  \|\p_1 u_1\|_{L^2}^{\frac12}\|\p_{11} u_1\|_{L^2}^{\frac12}
		+\|\p_{12} u_1\|_{L^2}^{\frac12}\|\p_{112} u_1\|_{L^2}^{\frac12},
		\deq
		yields directly
		\beq\label{21103}
		\begin{aligned}
			|J_{2,1}|
			\lesssim &~(\|\p_1 u_1\|_{L^2}^{\frac12}\|\p_{12} u_2\|_{L^2}^{\frac12}
			+\|\p_{12} u_1\|_{L^2}^{\frac12}\|\p_{122} u_2\|_{L^2}^{\frac12})
			\|\p_1 \phi \wx^{\gm+1}\|_{L^2}^2\\
			\lesssim &~1+\e^2.
		\end{aligned}
		\deq
		For the case $\alpha=(0,1)$, we have
		\beq\label{21105}
		\begin{aligned}
			|J_{2,1}|
			=&~\left|\int (\p_2 u_1-\p_1 u_2+\p_1 u_2) \p_1 \phi \cdot \p_2 \phi \wx^{2\gm+2}dx\right|\\
			\lesssim &~\|w\|_{L^2}^{\frac12}\|\p_1 w\|_{L^2}^{\frac12}
			\|\p_1 \phi \wx^{\gm+1}\|_{L^2}^{\frac12}
			\|\p_{12} \phi \wx^{\gm+1}\|_{L^2}^{\frac12}
			\|\p_2 \phi \wx^{\gm+1}\|_{L^2}\\
			&~+\|\p_1 u_2\|_{L^2}^{\frac12}\|\p_{21} u_2\|_{L^2}^{\frac12}
			\|\p_2 \phi \wx^{\gm+1}\|_{L^2}^{\frac12}
			\|\p_{12} \phi \wx^{\gm+1}\|_{L^2}^{\frac12}
			\|\p_1 \phi \wx^{\gm+1}\|_{L^2}\\
			\lesssim &~\|\p_1^2 w \wx^2\|_{L^2}
			\|(\p_1 \phi, \p_2 \phi, \p_{12}\phi) \wx^{\gm+1}\|_{L^2}^2
			+(1+\e^2)\\
			\lesssim&~1+\e^7,
		\end{aligned}
		\deq
		where we have used the estimate $\eqref{21001}_4$ in the last inequality.
		Thus, the combination of estimates \eqref{21103} and \eqref{21105} yields directly
		\beq\label{21106}
		|J_{2,1}| \lesssim  1+\e^7.
		\deq
		It is easy to check that
		\beq\label{21107}
		\begin{aligned}
			|J_{2,2}|
			&\lesssim \|\p_x^\alpha u_2\|_{L^2}^{\frac12}
			\|\p_2 \p_x^\alpha u_2\|_{L^2}^{\frac12}
			\|\p_2 \phi \wx^{\gm+1}\|_{L^2}^{\frac12}
			\|\p_1(\p_2 \phi \wx^{\gm+1})\|_{L^2}^{\frac12}
			\|\p_x^\alpha \phi \wx^{\gm+1}\|_{L^2}\\
			&\lesssim \|(\p_1 u, \p_{12}u )\|_{L^2}
			\|(\p_1 \phi, \p_2 \phi, \p_{12} \phi) \wx^{\gm+1}\|_{L^2}^2\\
			&\lesssim 1+\e^2.
		\end{aligned}
		\deq
		Similar to the estimate \eqref{2903}, we can obtain
		\beq\label{21108}
		|J_{2,3}|
		\lesssim \|(\p_1 u_1, \p_{12} u_1)\|_{L^2}
		\|\p_x^\alpha \phi \wx^{\gm+1}\|_{L^2}^2
		\lesssim 1+\e^2.
		\deq
		Substituting the estimates \eqref{21106}, \eqref{21107}
		and \eqref{21108} into \eqref{21102}, one arrives at
		\beqq
		\begin{aligned}
			\frac{d}{dt} \|\p_x^\alpha \phi \wx^{\gm+1}\|_{L^2}^2
			\lesssim 1+\e^7,
		\end{aligned}
		\deqq
		where $\alpha$ satisfies the condition $|\alpha|=1$.
		
		\textbf{Step 2}: Deal with the case $|\alpha|=2$.
		Then, the equation $\eqref{new-ns}_1$ yields directly
		\beq\label{21109}
		\begin{aligned}
			&\frac{d}{dt}\frac{1}{2}\int |\p_x^\alpha \phi|^2 \wx^{2\gm+2}dx
			+\underset{J_{2,4}}{\underbrace{
					\int \p_x^\alpha u_1 \p_1 \phi \cdot \p_x^\alpha \phi \wx^{2\gm+2}dx}}
			+\underset{J_{2,5}}{\underbrace{
					\int \p_x^\beta u_1 \p_x^{\alpha-\beta}\p_1 \phi
					\cdot \p_x^{\alpha} \phi \wx^{2\gm+2}dx}}\\
			&+
			\underset{J_{2,6}}{\underbrace{
					\int \p_x^\alpha u_2 \p_2 \phi \cdot \p_x^\alpha \phi \wx^{2\gm+2}dx}}
			+
			\underset{J_{2,7}}{\underbrace{
					\int \p_x^\beta u_2 \p_x^{\alpha-\beta}\p_2 \phi
					\cdot \p_x^{\alpha} \phi \wx^{2\gm+2}dx}}
			-\underset{J_{2,8}}{\underbrace{
					(\gm+1)\int u_1 | \p_x^\alpha \phi|^2 \wx^{2\gm+1}dx}}=0,
		\end{aligned}
		\deq
		where $|\beta|=1$.
		Deal with term $J_{2,4}$.
		If $1\le \alpha_1 \le 2$, we apply the divergence-free condition to obtain
		\beq\label{211010}
		|J_{2,4}|
		\lesssim \|\p_x^\alpha \phi \wx^{\gm+1}\|_{L^2}
		\|\p_x^\alpha u_1\|_{L^2}^{\frac12}
		\|\p_2 \p_x^\alpha u_1\|_{L^2}^{\frac12}
		\|\p_1 \phi\wx^{\gm+1}\|_{L^2}^{\frac12}
		\|\p_{1}(\p_{1}\phi\wx^{\gm+1})\|_{L^2}^{\frac12}
		\lesssim 1+\e^2.
		\deq
		If $\alpha_1=0$, then it is easy to check that
		\beq\label{211011}
		|J_{2,4}|
		\lesssim \|\p_2^2 \phi \wx^{\gm+1}\|_{L^2}
		\|\p_2^2 u_1\|_{L^2}^{\frac12}
		\|\p_1 \p_2^2 u_1\|_{L^2}^{\frac12}
		\|\p_1 \phi \wx^{\gm+1}\|_{L^2}^{\frac12}
		\|\p_{12}\phi\wx^{\gm+1})\|_{L^2}^{\frac12}.
		\deq
		Using the estimate of vorticity in \eqref{21001}, we can get
		\beq\label{211012}
		\begin{aligned}
			\|\p_2^2 u_1\|_{L^2}
			\le &~\|\p_2(\p_2 u_1-\p_1 u_2)\|_{L^2}+\|\p_{21} u_2\|_{L^2}\\
			\lesssim &~\|\p_{112} w \wx^2\|_{L^2}+\|\p_{21} u_2\|_{L^2}\\
			\le &~\sigma_1 \|\phi \p_t \p_2^2 u_1\|_{L^2}^2 +\sigma_2 \|\p_t \p_{12} u_2\|_{L^2}^2
			+C_{\sigma_1,\sigma_2}(1+\e^9).
		\end{aligned}
		\deq
		Thus, for the case $\alpha_1=0$, we apply the estimates
		\eqref{211011} and \eqref{211012} to obtain
		\beq\label{211013}
		|J_{2,4}|
		\le \sigma_1 \|\phi \p_t \p_2^2 u_1\|_{L^2}^2 +\sigma_2 \|\p_t \p_{12} u_2\|_{L^2}^2
		+C_{\sigma_1,\sigma_2}(1+\e^9).
		\deq
		Similar to the terms $J_{2,1}$, the terms $J_{2,5}$ can be estimated as follows.
		For the case $\beta=(1,0)$, then
		\beq\label{211014}
		|J_{2,5}|
		\lesssim \|\p_1 u_1\|_{L^\infty}
		\|\p_x^{\alpha-\beta}\p_1 \phi \wx^{\gm+1}\|_{L^2}
		\|\p_x^{\alpha} \phi \wx^{\gm+1}\|_{L^2}
		\lesssim 1+\e^2.
		\deq
		For the case $\beta=(0,1)$, then we have
		\beq\label{211015}
		\begin{aligned}
			|J_{2,5}|
			=&~\left|\int (\p_2 u_1-\p_1 u_2+\p_1 u_2)  \p_x^{\alpha-\beta}\p_1 \phi
			\cdot \p_x^{\alpha} \phi \wx^{2\gm+2}dx \right|\\
			\lesssim &~\|w \wx\|_{L^2}^{\frac12}
			\|\p_1(w \wx)\|_{L^2}^{\frac12}
			\|\p_x^{\alpha-\beta}\p_1 \phi \wx^{\gm}\|_{L^2}^{\frac12}
			\|\p_2 \p_x^{\alpha-\beta}\p_1 \phi \wx^{\gm}\|_{L^2}^{\frac12}
			\|\p_x^{\alpha} \phi \wx^{\gm+1}\|_{L^2}\\
			&~+
			\|\p_1 u_2 \|_{L^\infty}
			\|\p_x^{\alpha-\beta} \p_1 \phi \wx^{\gm+1}\|_{L^2}
			\|\p_x^{\alpha} \phi \wx^{\gm+1}\|_{L^2}\\
			\lesssim &~\|\p_1^2 w \wx^3\|_{L^2}
			\|(\p_x^{\alpha-\beta}\p_1 \phi,
			\p_2 \p_x^{\alpha-\beta}\p_1 \phi) \wx^{\gm}\|_{L^2}
			\|\p_x^{\alpha} \phi \wx^{\gm+1}\|_{L^2}\\
			&~+
			(\|\p_1 u_2\|_{L^2}^{\frac12}\|\p_{11} u_2\|_{L^2}^{\frac12}
			+\|\p_{21} u_2\|_{L^2}^{\frac12}\|\p_{112} u_2\|_{L^2}^{\frac12})
			\|(\p_x^{\alpha-\beta} \p_1 \phi,
			\p_x^{\alpha} \phi) \wx^{\gm+1}\|_{L^2}^2\\
			\lesssim &~1+\e^7.
		\end{aligned}
		\deq
		Similar to the terms $J_{2,2}$ and $J_{2,3}$,
		the terms $J_{2,7}$ and $J_{2,8}$ can be estimated as follows
		\beq\label{211016}
		\begin{aligned}
			|J_{2,7}|
			\lesssim &~\|(\p_1 u_2, \p_2 u_2)\|_{L^\infty}
			\|\p_x^{\alpha-\beta}\p_2 \phi \wx^{\gm+1}\|_{L^2}
			\|\p_x^{\alpha} \phi \wx^{\gm+1}\|_{L^2}\\
			\lesssim &~\|(\p_1 u_2, \p_1 u_1)\|_{L^\infty}
			\|(\p_x^{\alpha-\beta}\p_2 \phi, \p_x^{\alpha} \phi)\wx^{\gm+1}\|_{L^2}^2\\
			\lesssim &~(\|\p_1 u\|_{L^2}^{\frac12}\|\p_{11}u\|_{L^2}^{\frac12}
			+\|\p_{21} u\|_{L^2}^{\frac12}\|\p_{211}u\|_{L^2}^{\frac12})
			\|(\p_x^{\alpha-\beta}\p_2 \phi, \p_x^{\alpha} \phi)\wx^{\gm+1}\|_{L^2}^2\\
			\lesssim &~1+\e^6,
		\end{aligned}
		\deq
		and
		\beq\label{211017}
		|J_{2,8}|
		\lesssim \|(\p_1 u_1, \p_{12} u_1)\|_{L^2}
		\|\p_x^\alpha \phi \wx^{\gm+1}\|_{L^2}^2
		\lesssim 1+\e^2.
		\deq
		Finally, let us deal with the term $J_{2,6}$.
		Using the divergence-free condition $\eqref{new-ns}_3$, we have
		\beq\label{211018}
		|J_{2,6}|
		\lesssim \|\p_x^\alpha \phi \wx^{\gamma+1}\|_{L^2}
		\|\p_x^\alpha u_2\|_{L^2}^{\frac12}
		\|\p_2\p_x^\alpha u_2\|_{L^2}^{\frac12}
		\|\p_2 \phi \wx^{\gamma+1}\|_{L^2}^{\frac12}
		\|\p_1(\p_2 \phi \wx^{\gamma+1})\|_{L^2}^{\frac12}
		\lesssim \e^2.
		\deq
		Substituting the estimates \eqref{211010},\eqref{211013}-\eqref{211018} into \eqref{21109}, we can obtain
		\beqq
		\frac{d}{dt} \|\p_x^\alpha \phi \wx^{\gm+1}\|_{L^2}^2 \lesssim 1+\e^7,
		\quad \alpha=(1,1)~\text{or}~\alpha=(2,0);
		\deqq
		and
		\beqq
		\frac{d}{dt} \|\p_2^2 \phi \wx^{\gm+1}\|_{L^2}^2
		\le \sigma_1 \|\phi \p_t \p_2^2 u_1\|_{L^2}^2 +\sigma_2 \|\p_t \p_{12} u_2\|_{L^2}^2
		+C_{\sigma_1,\sigma_2}(1+\e^9).
		\deqq
		Therefore, we complete the proof of this lemma.
	\end{proof}
	
	\begin{lemm}\label{lemma212}
		For any smooth solution of equation $\eqref{new-ns}$, it holds with $\gamma \geq 2$
		\beqq
		\frac{d}{dt}\| \p_x^\alpha \phi \wx^\gm\|_{L^2}^2
		\le
		\var_1 \|\p_1 \p_2^3 u\|_{L^2}^2 +
		\sigma_1 \|\phi \p_t \p_2^2 u_1\|_{L^2}^2 +\sigma_2 \|\p_t \p_{12} u_2\|_{L^2}^2
		+C_{\var_1,\sigma_1,\sigma_2}(1+\e^9),
		\deqq
		where $\alpha$ satisfies $|\alpha |=3$, $\var_1$, $\sigma_1$ and $\sigma_2$ are the small constants that will be chosen later.
	\end{lemm}
	\begin{proof}
		For any $|\alpha|=3$, the density equation yields directly
		\beq\label{21202}
		\begin{aligned}
			&\frac{d}{dt}\frac{1}{2}\int |\p_x^\alpha \phi|^2\wx^{2\gamma}dx
			+\underset{J_{3,1}}{\underbrace{
					\int \p_x^\alpha u_1 \p_1 \phi \cdot \p_x^\alpha \phi \wx^{2\gm}dx}}
			+\underset{J_{3,2}}{\underbrace{
					\int \p_x^{\alpha-\beta}u_1 \p_x^\beta \p_1 \phi \cdot \p_x^\alpha \phi \wx^{2\gm}dx}}\\
			&+\underset{J_{3,3}}{\underbrace{
					\int \p_x^{\alpha-\kappa}u_1 \p_x^\kappa \p_1 \phi \cdot \p_x^\alpha \phi \wx^{2\gm}dx}}
			+\underset{J_{3,4}}{\underbrace{
					\int \p_x^\alpha u_2 \p_2 \phi \cdot \p_x^\alpha \phi \wx^{2\gm}dx}}
			+\underset{J_{3,5}}{\underbrace{
					\int \p_x^{\alpha-\beta}u_2 \p_x^\beta \p_2 \phi \cdot \p_x^\alpha \phi \wx^{2\gm}dx}}\\
			&+\underset{J_{3,6}}{\underbrace{
					\int \p_x^{\alpha-\kappa}u_2 \p_x^\kappa \p_2 \phi \cdot \p_x^\alpha \phi \wx^{2\gm}dx}}
			-\underset{J_{3,7}}{\underbrace{
					\gm\int u_1 | \p_x^\alpha \phi|^2 \wx^{2\gm-1}dx}}=0.
		\end{aligned}
		\deq
		where $|\beta|=1$ and $|\kappa|=2$.
		Deal with $J_{3,1}$ term.
		For the case $1\le \alpha_1 \le 3$, we apply estimate \eqref{21001} to obtain
		\beq\label{21203}
		\begin{aligned}
			J_{3,1}
			&\lesssim\|\p_1 \phi \wx^{\gm}\|_{L^\infty}
			\|\p_x^\alpha u_1\|_{L^2}
			\|\p_x^\alpha \phi \wx^\gm\|_{L^2}\\
			&\lesssim (\|\p_1 \phi \wx^{\gm}\|_{L^2}^{\frac12}
			\|\p_{21} \phi \wx^{\gm}\|_{L^2}^{\frac12}
			+\|\p_1^2 \phi \wx^{\gm}\|_{L^2}^{\frac12}
			\|\p_{12}(\p_1 \phi \wx^{\gm})\|_{L^2}^{\frac12})
			\|\p_x^\alpha u_1\|_{L^2}
			\|\p_x^\alpha \phi \wx^\gm\|_{L^2}\\
			&\lesssim 1+\e^7.
		\end{aligned}
		\deq
		For the case $\alpha_1=0$, the condition of $|\alpha|=3$
		implies $\alpha_2=3$, then we have
		\beq\label{21204}
		\begin{aligned}
			|J_{3,1}|
			&=\left|\int \p_2^3 u_1 \p_1 \phi \cdot \p_2^3 \phi \wx^{2\gm}dx\right|\\
			&\lesssim \|\p_2^3 u_1 \wx^{-1}\|_{L^2}
			\|\p_1(\p_2^3 u_1 \wx^{-1})\|_{L^2}^{\frac12}
			\|\p_1 \phi \wx^{\gm+1}\|_{L^2}^{\frac12}
			\|\p_{12} \phi \wx^{\gm+1}\|_{L^2}^{\frac12}
			\|\p_2^3 \phi \wx^{\gm}\|_{L^2}\\
			&\le \var_1 \|\p_1 \p_2^3 u_1 \|_{L^2}^2
			+C_{\var_1} \|\p_1 \phi \wx^{\gm+1}\|_{L^2}
			\|\p_{12} \phi \wx^{\gm+1}\|_{L^2}
			\|\p_2^3 \phi \wx^{\gm}\|_{L^2}^2.
		\end{aligned}
		\deq
		Deal with $J_{3,2}$ term.
		If $\alpha_1-\beta_1\neq0$, then we apply
		the estimate \eqref{21001} to obtain
		\beq\label{21205}
		\begin{aligned}
			|J_{3,2}|
			\lesssim  \|\p_x^{\alpha-\beta}u_1\|_{L^2}^{\frac12}
			\|\p_2 \p_x^{\alpha-\beta}u_1\|_{L^2}^{\frac12}
			\|\p_x^\beta \p_1 \phi \wx^{\gm}\|_{L^2}^{\frac12}
			\|\p_1(\p_x^\beta \p_1 \phi \wx^{\gm})\|_{L^2}^{\frac12}
			\|\p_x^\alpha \phi \wx^{\gm}\|_{L^2}
			\lesssim 1+\e^2.
		\end{aligned}
		\deq
		If $\alpha_1-\beta_1=0$, the condition of $|\alpha-\beta|=2$
		implies $\alpha_2-\beta_2=2$, then we have
		\beq\label{21206}
		\begin{aligned}
			|J_{3,2}|
			\lesssim &~\|\p_2^2 u_1\|_{L^2}^{\frac12}
			\|\p_1 \p_2^2 u_1\|_{L^2}^{\frac12}
			\|\p_x^\beta \p_1 \phi \wx^\gm\|_{L^2}^{\frac12}
			\|\p_2 \p_x^\beta \p_1 \phi \wx^\gm\|_{L^2}^{\frac12}
			\|\p_x^\alpha \phi \wx^{\gm}\|_{L^2}\\
			\le & ~\sigma_1 \|\phi \p_t \p_2^2 u_1\|_{L^2}^2 +\sigma_2 \|\p_t \p_{12} u_2\|_{L^2}^2
			+C_{\sigma_1,\sigma_2}(1+\e^9),
		\end{aligned}
		\deq
		where we have used the estimate \eqref{211012} in the last inequality.
		Deal with term $J_{3,3}$. If $\alpha_1-\kappa_1=1$, then we
		apply the divergence-free condition $\eqref{new-ns}_3$ to obtain
		\beq\label{21207}
		\begin{aligned}
			|J_{3,3}|
			\lesssim&~\|\p_1 u_1\|_{L^\infty}
			\|\p_x^\kappa \p_1 \phi \wx^\gm\|_{L^2}
			\|\p_x^\alpha \phi \wx^\gm\|_{L^2}\\
			\lesssim&~(\|\p_1 u_1\|_{L^2}^{\frac12}\|\p_1^2 u_1\|_{L^2}^{\frac12}
			+\|\p_1^2 u_1\|_{L^2}^{\frac12}\|\p_{112} u_1\|_{L^2}^{\frac12})
			\|\p_x^\kappa \p_1 \phi \wx^\gm\|_{L^2}
			\|\p_x^\alpha \phi \wx^\gm\|_{L^2}\\
			\lesssim&~(\|\p_1 u_1\|_{L^2}^{\frac12}\|\p_{12} u_2\|_{L^2}^{\frac12}
			+\|\p_{12} u_2\|_{L^2}^{\frac12}\|\p_{122} u_2\|_{L^2}^{\frac12})
			\|\p_x^\kappa \p_1 \phi \wx^\gm\|_{L^2}
			\|\p_x^\alpha \phi \wx^\gm\|_{L^2}\\
			\lesssim&~1+\e^2.
		\end{aligned}
		\deq
		If $\alpha_1-\kappa_1=0$, the condition of $|\alpha-\kappa|=1$
		implies $\alpha_2-\kappa_2=1$, then we have
		\beq\label{21208}
		|J_{3,3}|
		\lesssim \|\p_2 u_1\|_{L^\infty}
		\|\p_x^\kappa \p_1 \phi \wx^\gm\|_{L^2}
		\|\p_x^\alpha \phi \wx^\gm\|_{L^2}.
		\deq
		We apply the Sobolev inequality and
		Hardy inequality  to obtain
		\beq\label{21209}
		\begin{aligned}
			\|\p_2 u_1\|_{L^\infty}
			\lesssim &~\|w \|_{L^\infty}
			+\|\p_1 u_2\|_{L^\infty}\\
			\lesssim &~\|w\|_{L^2}^{\frac12}\|\p_2 w\|_{L^2}^{\frac12}
			+\|\p_1 w\|_{L^2}^{\frac12}\|\p_{12}w\|_{L^2}^{\frac12}
			+\|\p_1 u_2\|_{L^2}^{\frac12}\|\p_{12} u_2\|_{L^2}^{\frac12}
			+\|\p_{11} u_2\|_{L^2}^{\frac12}\|\p_{112} u_2\|_{L^2}^{\frac12}\\
			\lesssim& ~\|\p_1^2 w \wx^2 \|_{L^2}^{\frac12}\|\p_{112}w \wx^2\|_{L^2}^{\frac12}
			+\|\p_1 u_2\|_{L^2}^{\frac12}\|\p_{12} u_2\|_{L^2}^{\frac12}
			+\|\p_{11} u_2\|_{L^2}^{\frac12}\|\p_{112} u_2\|_{L^2}^{\frac12}.
		\end{aligned}
		\deq
		Substituting the estimate \eqref{21209} into \eqref{21208}
		and using estimate \eqref{21001}, we have
		\beq\label{212010}
		|J_{3,3}|
		\le
		\sigma_1 \|\phi \p_t \p_2^2 u_1\|_{L^2}^2 +\sigma_2 \|\p_t \p_{12} u_2\|_{L^2}^2
		+C_{\sigma_1,\sigma_2}(1+\e^9).
		\deq
		Due to the divergence-free condition and estimate
		\eqref{21001}, it is easy to check that
		\beq\label{212011}
		|J_{3,4}|
		\lesssim  \|\p_2 \phi \wx^{\gm}\|_{L^\infty}
		\|\p_x^\alpha u_2 \|_{L^2}
		\| \p_x^\alpha \phi \wx^{\gm}\|_{L^2}
		\lesssim  1+\e^7.
		\deq
		Similarly, it is easy to check that
		\beq\label{212012}
		|J_{3,5}|
		\lesssim \|\p_x^{\alpha-\beta}u_2\|_{L^2}^{\frac12}
		\|\p_2 \p_x^{\alpha-\beta}u_2\|_{L^2}^{\frac12}
		\|\p_x^\beta \p_2 \phi \wx^\gm\|_{L^2}^{\frac12}
		\|\p_1(\p_x^\beta \p_2 \phi \wx^\gm)\|_{L^2}^{\frac12}
		\|\p_x^\alpha \phi \wx^{\gm}\|_{L^2}
		\lesssim 1+\e^4,
		\deq
		and
		\beq\label{212013}
		|J_{3, 7}|
		\lesssim  \|(\p_1 u_1, \p_{12} u_1)\|_{L^2}
		\|\p_x^\alpha \phi \wx^{\gm}\|_{L^2}^2
		\lesssim 1+\e^2.
		\deq
		Finally, we apply the divergence-free condition to obtain
		\beq\label{212014}
		\begin{aligned}
			|J_{3, 6}|
			\lesssim &~\|\p_x^{\alpha-\kappa}u_2\|_{L^\infty}
			\|\p_x^\kappa \p_2 \phi \wx^{\gm}\|_{L^2}
			\|\p_x^\alpha \phi  \wx^{\gm}\|_{L^2}\\
			\lesssim &~(\|\p_1 u_2\|_{L^\infty}
			+\|\p_1 u_1\|_{L^\infty})
			\|\p_x^\kappa \p_2 \phi \wx^{\gm}\|_{L^2}
			\|\p_x^\alpha \phi  \wx^{\gm}\|_{L^2}\\
			\lesssim &~1+\e^5,
		\end{aligned}
		\deq
		where we have used the estimate \eqref{21001} in the last inequality.
		Substituting the estimates \eqref{21203}-\eqref{21207} and \eqref{212010}-\eqref{212014}
		into \eqref{21202}, we have
		\beqq
		\frac{d}{dt}\|\p_1^3 \phi \wx^\gm\|_{L^2}^2
		\lesssim 1+\e^7,
		\deqq
		and
		\beqq
		\frac{d}{dt}\| \p_x^\alpha \phi \wx^\gm\|_{L^2}^2
		\le
		\var_1 \|\p_1 \p_2^3 u\|_{L^2}^2 +
		\sigma_1 \|\phi \p_t \p_2^2 u_1\|_{L^2}^2 +\sigma_2 \|\p_t \p_{12} u_2\|_{L^2}^2
		+C_{\var_1,\sigma_1,\sigma_2}(1+\e^9), \quad 1 \le \alpha_2 \le 3 .
		\deqq
		Therefore, we complete the proof of this lemma.
	\end{proof}
	
	\begin{lemm}\label{lemma213}
		For any smooth solution of equation $\eqref{new-ns}$, it holds 
		\beqq 
		\frac{d}{dt}\|\p_x^\alpha \ln \phi\|_{L^2}^2
		\le \sigma_1 \|\phi \p_t \p_2^2 u_1\|_{L^2}^2 +\sigma_2 \|\p_t \p_{12} u_2\|_{L^2}^2
		+C_{\sigma_1,\sigma_2}(1+\e^9),
		\deqq
		where $|\alpha|=1$, $\sigma_1$ and $\sigma_2$ are the small constants that will be chosen later.
	\end{lemm}
	\begin{proof}
		First of all, the density equation $\eqref{new-ns}_1$ yields directly
		\beq\label{21302}
		\p_t \ln \phi +u_1 \p_1 \ln \phi+u_2 \p_2 \ln \phi=0,
		\deq
		which yields directly for all $|\alpha|=1$,
		\beq\label{21303}
		\p_t (\p_x^\alpha \ln \phi)+u_1 \p_1 (\p_x^\alpha \ln \phi)
		+u_2 \p_2 (\p_x^\alpha  \ln \phi)
		+\p_x^\alpha u_1 \p_1 \ln \phi+\p_x^\alpha u_2 \p_2 \ln \phi=0.
		\deq
		The equation \eqref{21303} and divergence-free condition
		$\eqref{new-ns}_3$ yield directly
		\beqq
		\begin{aligned}
			\frac{d}{dt}\frac{1}{2}\int |\p_x^\alpha \ln \phi|^2 dx
			=&~-\int \p_x^\alpha u_1 \p_1 \ln \phi \cdot \p_x^\alpha \ln  \phi dx
			-\int \p_x^\alpha u_2 \p_2 \ln \phi \cdot \p_x^\alpha \ln  \phi dx\\
			\lesssim&~(\|\p_x^\alpha u_1\|_{L^\infty}+\|\p_x^\alpha u_2\|_{L^\infty})
			\|(\p_1 \ln \phi, \p_2 \ln \phi)\|_{L^2}
			\|\p_x^\alpha \ln  \phi\|_{L^2}\\
			\lesssim&~\|(\p_1 u_1, \p_1 u_2, \p_2 u_1)\|_{L^\infty}
			\|(\p_1 \ln \phi, \p_2 \ln \phi)\|_{L^2}
			\|\p_x^\alpha \ln  \phi\|_{L^2},
		\end{aligned}
		\deqq
		which, together with the estimates \eqref{21104},
		\eqref{21209} and \eqref{21001}, gives directly
		\beqq
		\frac{d}{dt}\|\p_x^\alpha \ln \phi\|_{L^2}^2
		\le \sigma_1 \|\phi \p_t \p_2^2 u_1\|_{L^2}^2 +\sigma_2 \|\p_t \p_{12} u_2\|_{L^2}^2
		+C_{\sigma_1,\sigma_2}(1+\e^9).
		\deqq
		Therefore, we complete the proof of this lemma.
	\end{proof}
	
	Finally, let us establish the estimate for the
	second order derivative of quantity $\ln \phi$.
	\begin{lemm}\label{lemma214}
		For any smooth solution of equation \eqref{new-ns}, it holds 
		\beqq
		\frac{d}{dt}\|\p_x^\alpha \ln \phi\|_{L^2}^2
		\le \sigma_1 \|\phi \p_t \p_2^2 u_1\|_{L^2}^2 +\sigma_2 \|\p_t \p_{12} u_2\|_{L^2}^2
		+C_{\sigma_1,\sigma_2}(1+\e^9),
		\deqq
		where $|\alpha|=2$, $\sigma_1$ and $\sigma_2$ are the small constants that will be chosen later.
	\end{lemm}
	\begin{proof}
		For any $|\alpha|=|(\alpha_1, \alpha_2)|=2$, the equation \eqref{21302} yields directly
		\beqq
		\frac{d}{dt}\frac{1}{2}\int |\p_x^\alpha \ln \phi|^2 dx
		+\int \p_x^\alpha(u_1 \p_1 \ln \phi+u_2 \p_2 \ln \phi)
		\p_x^\alpha \ln \phi dx=0,
		\deqq
		which, together with  the divergence-free
		condition $\eqref{new-ns}_3$, implies directly
		\beq\label{21402}
		\begin{aligned}
			&\frac{d}{dt}\frac{1}{2}\int |\p_x^\alpha \ln \phi|^2 dx
			+\underset{J_{4,1}}{\underbrace{
					\int \p_x^\alpha u_1 \p_1 \ln \phi \cdot \p_x^\alpha \ln \phi dx}}
			+\underset{J_{4,2}}{\underbrace{
					\int \p_x^\beta u_1 \p_x^{\alpha-\beta}\p_1 \ln \phi \cdot \p_x^\alpha \ln \phi dx}}\\
			&\quad \quad
			+\underset{J_{4,3}}{\underbrace{
					\int \p_x^\alpha u_2 \p_2 \ln \phi \cdot \p_x^\alpha \ln \phi dx}}
			+\underset{J_{4,4}}{\underbrace{
					\int \p_x^\beta u_2 \p_x^{\alpha-\beta}\p_2 \ln \phi \cdot \p_x^\alpha \ln \phi dx}}=0,
		\end{aligned}
		\deq
		where $|\beta|=|(\beta_1, \beta_2)|=1$.
		Deal with term $J_{4,1}$.
		For the case $1\le \alpha_1 \le 2$, we can apply the
		divergence-free condition $\eqref{new-ns}_3$ to obtain
		\beq\label{21403}
		|J_{4,1}|
		\lesssim \| \p_x^\alpha \ln \phi\|_{L^2}
		\|\p_x^\alpha u_1\|_{L^2}^{\frac12}
		\|\p_2 \p_x^\alpha u_1\|_{L^2}^{\frac12}
		\|\p_1 \ln \phi\|_{L^2}^{\frac12}
		\|\p_1^2 \ln \phi\|_{L^2}^{\frac12}
		\lesssim 1+\e^2.
		\deq
		For the case $\alpha_1=0$, the condition $|\alpha|=2$
		implies $\alpha_2=2$. Thus, we apply the estimate \eqref{211012}
		to obtain
		\beq\label{21404}
		\begin{aligned}
			|J_{4,1}|
			=&~\left|\int \p_2^2 u_1 \p_1 \ln \phi \cdot \p_2^2 \ln \phi dx \right|\\
			\lesssim&~\|\p_2^2 \ln \phi\|_{L^2}
			\|\p_2^2 u_1\|_{L^2}^{\frac12}
			\|\p_1 \p_2^2 u_1\|_{L^2}^{\frac12}
			\|\p_1 \ln \phi \|_{L^2}^{\frac12}
			\|\p_{12} \ln \phi \|_{L^2}^{\frac12}\\
			\le &~ \sigma_1 \|\phi \p_t \p_2^2 u_1\|_{L^2}^2 +\sigma_2 \|\p_t \p_{12} u_2\|_{L^2}^2
			+C_{\sigma_1,\sigma_2}(1+\e^9).
		\end{aligned}
		\deq
		Due to the divergence-free condition, it is easy to check that
		\beqq
		\begin{aligned}
			|J_{4,2}|+|J_{4,4}|
			\lesssim &~(\|\p_x^\beta u_1\|_{L^\infty}+\|\p_x^\beta u_2\|_{L^\infty})
			\|(\p_x^{\alpha-\beta}\p_1 \ln \phi, \p_x^{\alpha-\beta}\p_2 \ln \phi)\|_{L^2}
			\|\p_x^\alpha \ln \phi\|_{L^2}\\
			\lesssim &~\|(\p_1 u_1, \p_2 u_1, \p_1 u_2)\|_{L^\infty}
			\|(\p_x^{\alpha-\beta}\p_1 \ln \phi, \p_x^{\alpha-\beta}\p_2 \ln \phi)\|_{L^2}
			\|\p_x^\alpha \ln \phi\|_{L^2},\\
		\end{aligned}
		\deqq
		which, together with the estimates \eqref{21104},
		\eqref{21209} and \eqref{21001}, yields directly
		\beq\label{21405}
		|J_{4,2}|+|J_{4,4}|\le
		\sigma_1 \|\phi \p_t \p_2^2 u_1\|_{L^2}^2 +\sigma_2 \|\p_t \p_{12} u_2\|_{L^2}^2
		+C_{\sigma_1,\sigma_2}(1+\e^9).
		\deq
		Using divergence-free condition $\eqref{new-ns}_3$
		and estimate \eqref{21001}, we have
		\beq\label{21406}
		|J_{4,3}|
		\lesssim \| \p_x^\alpha \ln \phi\|_{L^2}
		\|\p_x^\alpha u_2\|_{L^2}^{\frac12}
		\|\p_2 \p_x^\alpha u_2\|_{L^2}^{\frac12}
		\|\p_2 \ln \phi\|_{L^2}^{\frac12}
		\|\p_{12} \ln \phi\|_{L^2}^{\frac12}
		\lesssim 1+\e^4.
		\deq
		Substituting the estimates \eqref{21403}, \eqref{21404},
		\eqref{21405} and \eqref{21406} into \eqref{21402}, we have
		\beqq
		\frac{d}{dt}\|\p_x^\alpha \ln \phi\|_{L^2}^2
		\le \sigma_1 \|\phi \p_t \p_2^2 u_1\|_{L^2}^2 +\sigma_2 \|\p_t \p_{12} u_2\|_{L^2}^2
		+C_{\sigma_1,\sigma_2}(1+\e^9),
		\deqq
		where $|\alpha|=2$. Therefore, we complete the proof of this lemma.
	\end{proof}
	
	\subsection{Closed energy estimate}
	First of all, let us define the energy
	\begin{equation}\label{2.301}
		\begin{aligned}
			\widehat{\mathcal{E}}(t)
			\overset{\text{def}}{=} &\sum_{k=0}^2\|\phi \p_2^k u\|_{L^2}^2
			+C_1\|\phi \p_2^3 u\|_{L^2}^2
			+\sum_{k=0}^2 \|\p_2^k \p_1 u\|_{L^2}^2
			+\|\phi \p_t u\|_{L^2}^2
			+C_2\|\phi \p_t \p_2 u\|_{L^2}^2
			+\|\p_t \p_1 u\|_{L^2}^2\\
			&+\sum_{0\le |\alpha|\le 2}\|\p_x^\alpha \phi \wx^{\gm+1}\|_{L^2}^2
			+\sum_{|\alpha|=3}\|\p_x^\alpha \phi \wx^{\gm}\|_{L^2}^2
			+\sum_{1\le |\alpha| \le 2}\|\p_x^\alpha \ln \phi\|_{L^2}^2,
		\end{aligned}
	\end{equation}
	where the constants $C_1\ge 1$ and $C_2 \ge 1$ will be chosen later.
	Let us define $\widehat{E}(t)\overset{\text{def}}{=} \underset{0\le \tau \le t}{\sup}\widehat{\mathcal{E}}(\tau)$.
	The combination of estimates obtained in Lemmas \ref{lemma21}-\ref{lemma29}
	and \ref{lemma211}-\ref{lemma214} yields directly
	\beq\label{2.303}
	\begin{aligned}
		\widehat{\mathcal{E}}(t)+\int_0^t \dtau d\tau
		\le &~\widehat{\mathcal{E}}(0)
		+ \sigma\int_0^t((C_1+C_2+1)\|\phi \p_2^2 \p_t u\|_{L^2}^2
		+(C_1+1)\|\p_{12} \p_t u\|_{L^2}^2
		+(C_2+1)\|\phi \p_t^2 u\|_{L^2}^2\\
		& ~+\|\p_1 \p_2^3 u\|_{L^2}^2)d\tau +C_{\sigma}\int_0^t(1+\e^9)d\tau\\
		&~+C\int_0^t \|\p_1 u\|_{L^2}^2
		\|(\phi, \p_1 \phi, \p_2 \phi, \p_1 \p_2 \phi) \wx^2\|_{L^2}^2
		\|\p_{12} \p_t u \|_{L^2}^2 d\tau\\
		&~+C\int_0^t \|(\p_1 u, \p_{12} u)\|_{L^2}^2
		\|(\phi, \p_1 \phi, \p_2 \phi, \p_{12}\phi)\wx\|_{L^2}^2
		\|\p_1 \p_2^3 u \|_{L^2}^2 d\tau,
	\end{aligned}
	\deq
	where $\sigma \overset{\text{def}}{=} \max\{\var_1,\sigma_1,\sigma_2,\sigma_3\}.$
	Secondly, the estimates \eqref{22001}, \eqref{2501},
	\eqref{2901} and \eqref{21101a} yields directly
	\beq\label{2.304}
	\|\p_1 u\|_{L^2}^2 +\int_0^t \|\phi \p_t u\|_{L^2}^2  d\tau
	\le \|\p_1 u_0\|_{L^2}^2+ CE(t)^3 t,
	\deq
	\beq\label{2.305}
	\|\p_1 \p_2 u\|_{L^2}^2
	+\int_0^t \|\phi \p_2 \p_t u\|_{L^2}^2 d\tau
	\le \|\p_1 \p_2 u_0\|_{L^2}^2+C(1+E(t)^3)t,
	\deq
	and
	\beq\label{2.306}
	\|(\phi, \p_1 \phi, \p_2 \phi,  \p_{12} \phi)\wx^{\gm+1}\|_{L^2}^2
	\le \|(\phi_0, \p_1 \phi_0, \p_2 \phi_0,  \p_{12} \phi_0)\wx^{\gm+1}\|_{L^2}^2
	+C(1+E(t)^7)t.
	\deq
	Here the definition of $E(t)$ can be found in \eqref{priori-estimate}.
	Thus, the combination of estimates \eqref{2.304}, \eqref{2.305}
	and \eqref{2.306} yields directly
	\beq\label{2.307}
	\begin{aligned}
		&\|\p_1 u\|_{L^2}^2
		\|(\phi, \p_1 \phi, \p_2 \phi, \p_1 \p_2 \phi) \wx^2\|_{L^2}^2
		\|\p_{12} \p_t u \|_{L^2}^2\\
		\le &~(\|\p_1 u_0\|_{L^2}^2+C \widehat{E}(t)^3 t)
		(\|(\phi_0, \p_1 \phi_0, \p_2 \phi_0,  \p_{12} \phi_0)\wx^{\gm+1}\|_{L^2}^2
		+C(1+\widehat{E}(t)^7)t)\|\p_{12} \p_t u \|_{L^2}^2\\
		\le &~(\|\p_1 u_0\|_{L^2}^4
		+\|(\phi_0, \p_1 \phi_0, \p_2 \phi_0,  \p_{12} \phi_0)\wx^{\gm+1}\|_{L^2}^4)
		\|\p_{12} \p_t u \|_{L^2}^2
		+C(1+\widehat{E}(t)^{14}) t^2 \|\p_{12} \p_t u \|_{L^2}^2,
	\end{aligned}
	\deq
	and
	\beq\label{2.308}
	\begin{aligned}
		&
		\|(\phi, \p_1 \phi, \p_2 \phi, \p_{12}\phi)\wx\|_{L^2}^2
		\|(\p_1 u, \p_{12} u)\|_{L^2}^2
		\|\p_1 \p_2^3 u \|_{L^2}^2\\
		\le&~(\|(\phi_0, \p_1 \phi_0, \p_2 \phi_0,  \p_{12} \phi_0)\wx^{\gm+1}\|_{L^2}^2
		+C(1+\widehat{E}(t)^3)t)\\
		&~\times(\|(\p_1 u_0, \p_1 \p_2 u_0)\|_{L^2}^2+C(1+\widehat{E}(t)^7)t)
		\|\p_1 \p_2^3 u \|_{L^2}^2\\
		\le&~ (\|(\phi_0, \p_1 \phi_0, \p_2 \phi_0, \p_{12} \phi_0)\wx^{\gm+1}\|_{L^2}^4
		+\|(\p_1 u_0, \p_1 \p_2 u_0)\|_{L^2}^4)\|\p_1 \p_2^3 u \|_{L^2}^2\\
		&~+C(1+\widehat{E}(t)^{14}) t^2 \|\p_1 \p_2^3 u \|_{L^2}^2.
	\end{aligned}
	\deq
	Let us define $C_0\overset{\text{def}}{=}
	\|(\phi_0, \p_1 \phi_0, \p_2 \phi_0, \p_{12} \phi_0)\wx^{\gm+1}\|_{L^2}^4
	+\|(\p_1 u_0, \p_1 \p_2 u_0)\|_{L^2}^4$.
	Then, substituting the estimates \eqref{2.307} and \eqref{2.308} into \eqref{2.303},
	one arrives at
	\beq\label{2.309}
	\begin{aligned}
		\widehat{\mathcal{E}}(t)+\int_0^t \dtau d\tau
		\le &~\widehat{\mathcal{E}}(0)
		+\sigma\int_0^t((C_1+C_2+1)\|\phi \p_2^2 \p_t u\|_{L^2}^2
		+(C_1+1)\|\p_{12} \p_t u\|_{L^2}^2
		+(C_2+1)\|\phi \p_t^2 u\|_{L^2}^2\\
		&~ +\|\p_1 \p_2^3 u\|_{L^2}^2)d\tau
		+C_{\sigma}\int_0^t(1+\widehat{\mathcal{E}}(t)^9)d\tau
		+C(C_0+(1+\widehat{E}(t)^{14}) t^2)\int_0^t \|\p_{12} \p_t u \|_{L^2}^2 d\tau\\
		&~+C(C_0+(1+\widehat{E}(t)^{14}) t^2)\int_0^t \|\p_1 \p_2^3 u \|_{L^2}^2 d\tau.
	\end{aligned}
	\deq
	Thus, choosing $C_1=C_2=C(C_0+1)$ in \eqref{2.301} and $\sigma$ small enough,
	the estimate \eqref{2.309} yields directly
	\beqq
	\begin{aligned}
		\widehat{\mathcal{E}}(t)+\int_0^t \dtau d\tau
		\le \widehat{\mathcal{E}}(0)+C(1+\widehat{E}(t)^9)t
		+C(1+\widehat{E}(t)^{14}) t^2 \int_0^t(\|\p_{12} \p_t u \|_{L^2}^2
		+\|\p_1 \p_2^3 u \|_{L^2}^2) d\tau.
	\end{aligned}
	\deqq
	Due to the initial data assumption in Proposition \ref{pro-priori},
	it is easy to check that
	\beq\label{initial-control-1}
	\widehat{\mathcal{E}}(0)\le C_{\phi_0, u_0, p_0, \gamma},
	\deq
	which can be found in detail in appendix \ref{initial-control-A}.
	Thus, due to the equivalent relation of $\widehat{E}(t)$
	and ${E}(t)$, we complete the proof of the
	Proposition \ref{pro-priori}.
	
	\subsection{Local-in-time existence of original equation}
	
	In this section, we will establish the local-in-time well-posedness
	for the anisotropic inhomogeneous incompressible Navier-Stokes equations
	\eqref{o-ns}.
	From the analysis of a priori estimate in the previous subsections \ref{estimate-v} and \ref{estimate-rho},
	it is convenient to introduce the good unknown $\phi $.
	Then, the system \eqref{o-ns} will translate into the system \eqref{new-ns}.
	Then, our target is to establish the local-in-time well-posedness for the
	system \eqref{new-ns} and the estimate \eqref{main-estimate} in Theorem \ref{local-well}.

	In this subsection, we will investigate the local-in-time existence
	of original equation \eqref{new-ns}. Thus, let us investigate the following system
	\beq\label{new-ns-4}
	\left\{\begin{aligned}
		&\p_t \phi^\kappa+u^\kappa \cdot \nabla \phi^\kappa=0,\\
		&(\phi^\kappa+\kappa)^2 \p_t u^\kappa
		+(\phi^\kappa)^2 u^\kappa \cdot \nabla u^\kappa
		-\p_1^2 u^\kappa+\nabla p^\kappa=0,\\
		&{\rm div}u^\kappa=0,
	\end{aligned}\right.
	\deq
	with the initial data
	\beq\label{app-ini-data}
	(\phi^\kappa, u^\kappa)|_{t=0}=(\phi_0, u_0),
	\deq
	and
	\beq\label{app-far-data}
	\lim_{|x|\to +\infty}(\phi^\kappa, u^\kappa)=(0,0),
	\deq
	where the constant $\kappa \in (0, 1)$.
	Then, there exists a positive time $T_{\kappa}>0$ ($T_{\kappa}$ may depending on
	the constant $\kappa$), such that the systems \eqref{new-ns-4}-\eqref{app-far-data}
	admit a unique solution $(\phi^\kappa, u^\kappa, p^\kappa)$ that will satisfy the following
	estimate
	\beqq
	\|\phi^\kappa \|_{H^3}^2
	+\sum_{|\alpha|\le 3}\|(\phi^\kappa+\kappa) \p_x^\alpha u^\kappa\|_{L^2}^2
	+\int_0^t \|\p_1 u^\kappa\|_{H^3}^2 d\tau
	\le C_{\phi_0,u_0},
	\deqq
	for all $0<t \le T_{\kappa}$ and the constant $C_{\phi_0,u_0}$
	is independent of the parameter $\kappa$.
	In the following, we will establish the uniform existence life time
	and the solution satisfying the equation \eqref{new-ns-4}.
	Let us define the energy as follows:
	\beqq 
	\begin{aligned}
		\mathcal{E}(\phi^\k, u^\k, \k)(t)
		\overset{\text{def}}{=}
		&~\sum_{j=0}^3\|(\phi^\k+\k) \p_2^j u^\k\|_{L^2}^2
		+\sum_{j=0}^2 \|\p_2^j \p_1 u^\k\|_{L^2}^2
		+\sum_{j=0}^1  \|(\phi^\k+\k) \p_2^j \p_t u^\k\|_{L^2}^2
		+\|\p_t \p_1 u^\k\|_{L^2}^2\\
		&+\sum_{0\le |\alpha|\le 2}\|\p_x^\alpha {\phi^\k} \wx^{\gm+1}\|_{L^2}^2
		+\sum_{|\alpha|=3}\|\p_x^\alpha {\phi^\k} \wx^{\gm}\|_{L^2}^2
		+\sum_{1\le |\alpha| \le 2}\|\p_x^\alpha \ln \phi^\k\|_{L^2}^2,
	\end{aligned}
	\deqq
	and the dissipation part
	\beqq
	\begin{aligned}
		\mathcal{D}(\phi^\k, u^\k, \k)(t)
		\overset{\text{def}}{=}
		&~\sum_{j=0}^2 \|\p_1 \p_2^j u^\k\|_{L^2}^2
		+C_1 \|\p_1 \p_2^3 u^\k\|_{L^2}^2
		+\sum_{j=0}^2 \|(\phi^\k+\k) \p_2^j \p_t u^\k\|_{L^2}^2 \\
		&+\|\p_1 \p_t u^\k\|_{L^2}^2
		+C_2\|\p_{12} \p_t u^\k\|_{L^2}^2 + \|(\phi^\k+\k)  \p_t^2 u^\k\|_{L^2}^2.
	\end{aligned}
	\deqq
	Here we also define $E(\phi^\k, u^\k, \k)(T) \overset{\text{def}}{=}
	\underset{t\in [0, T]}{\sup}\mathcal{E}(\phi^\k, u^\k, \k)(t)$.
	
	\textbf{Step 1: Uniform estimate and life-span time.}
	For the parameter $R$, which will be defined later, we define
	\beq\label{T-star}
	\begin{aligned}
		T_*^\kappa \overset{\text{def}}{=}
		&\sup\left\{T\in (0, 1]|\  E(\phi^\k, u^\k, \k)(t) \le R,\ t \in [0, T] \right\}.
	\end{aligned}
	\deq
	First of all, we can apply the characteristic line method
	and the initial data assumption $\phi_0>0$ to obtain
	\beqq
	\phi^\kappa(x, t)>0, \quad (x, t)\in \mathbb{R}^2\times [0, T_*^\kappa].
	\deqq
	Obviously, from the estimate \eqref{priori-estimate}
	in Proposition \ref{pro-priori} we may conclude
	for all $T \le T_*^\kappa$  that
	\beq\label{4104}
	\begin{aligned}
		& E(\phi^\k, u^\k, \k)(T)
		+\int_0^T \mathcal{D}(\phi^\k, u^\k, \k)(t) d t\\
		\le &~C_3 {E}(\phi^\k, u^\k, \k)(0)
		+C_4 (1+ {E}(\phi^\k, u^\k, \k)(T)^9)T\\
		&+C_5 (1+ {E}(\phi^\k, u^\k, \k)(T)^{14})
		T^2 \int_0^T(\|\p_{12} \p_t u^\kappa \|_{L^2}^2
		+\|\p_1 \p_2^3 u^\kappa \|_{L^2}^2) dt,
	\end{aligned}
	\deq
	here $C_3, C_4, C_5$ are the constants independent of $\k$.
	Due to the assumption of the initial data in Theorem
	\ref{local-well}, we can obtain that
	\beq\label{4105}
	\mathcal{E}(\phi^\k, u^\k, \k)(0)\le C_{\phi_0, u_0, p_0, \gamma},
	\deq
	which can be found in detail in appendix \ref{initial-control-A}.
	Then, the combination of \eqref{4104} and \eqref{4105} yields directly
	\beq\label{4106}
	\begin{aligned}
		& E(\phi^\k, u^\k, \k)(T)
		+\int_0^T \mathcal{D}(\phi^\k, u^\k, \k)(t) d t\\
		\le &~C_3 C_{\phi_0, u_0, p_0, \gamma}
		+C_4 (1+ {E}(\phi^\k, u^\k, \k)(T)^9)T\\
		&+C_5 (1+ {E}(\phi^\k, u^\k, \k)(T)^{14})
		T^2 \int_0^T(\|\p_{12} \p_t u^\kappa \|_{L^2}^2
		+\|\p_1 \p_2^3 u^\kappa \|_{L^2}^2) dt.
	\end{aligned}
	\deq
	Let us choose $R  \overset{\text{def}}{=} 4C_3 C_{\phi_0, u_0, p_0, \gamma}$ and
	$T_*  \overset{\text{def}}{=}\min \left\{1,
	\frac{C_3 C_{\phi_0, u_0, p_0, \gamma}}{C_4(1+(4C_3 C_{\phi_0, u_0, p_0, \gamma})^9)},
	\sqrt{\frac{1}{2C_5(1+(4C_3 C_{\phi_0, u_0, p_0, \gamma})^{14})}}\right\}$,
	then the inequality \eqref{4106} yields directly
	\beq\label{4107}
	E(\phi^\k, u^\k, \k)(T)
	+\frac12\int_0^T \mathcal{D}(\phi^\k, u^\k, \k)(t) d t
	\le 2C_3 C_{\phi_0, u_0, p_0, \gamma}
	=\frac12 R,
	\deq
	for all $T\le \min \left\{T_*, T_*^\kappa\right\}$.
	This yields $T_* \le T_*^\kappa$.
	Indeed otherwise, our criterion about the continuation of
	the solution would contradict the definition of $T_*^\kappa$
	in \eqref{T-star}.
	Let us define $T_{0} \overset{\text{def}}{=} T_*$, then we find the
	uniform existence time $T_{0}$ (independent of $\kappa$)
	such that we have the estimate \eqref{4107}.
	
	\textbf{Step 2: Local-in-time existence.}
	Due to the uniform estimate \eqref{4107}, we apply the Hardy
	inequality to obtain
	\beq\label{4108}
	\|(\p_t u^\kappa, \p_2 \p_t u^\kappa)\wx^{-1}\|_{L^2}^2
	+\|\p_1 \p_t u^\kappa \|_{L^2}^2
	\lesssim \|(\p_1 \p_t u^\kappa, \p_{12}\p_t u^\kappa)\|_{L^2}^2
	\lesssim C_{\phi_0, u_0, p_0, \gamma},
	\deq
	and
	\beq\label{4109}
	\|(u^\kappa, \p_2 u^\kappa, \p_2^2 u^\kappa)\wx^{-1}\|_{L^2}^2
	+\|(\p_1 u^\kappa, \p_{12} u^\kappa, \p_1^2 u^\kappa)\|_{L^2}^2
	\lesssim \|(\p_1 u^\kappa, \p_{12}u^\kappa, \p_{122}u^\kappa, \p_1^2 u^\kappa)\|_{L^2}^2
	\lesssim C_{\phi_0, u_0, p_0, \gamma}.
	\deq
	Then, the estimates \eqref{4108} and \eqref{4109} yield that
	$u^\kappa$ is bounded uniformly in $L^\infty(0, T_0; H^2_{loc})$, and
	$\p_t u^\kappa$ is bounded uniformly in $L^2(0, T_0; H^1_{loc})$ respectively.
	Then, it follows from a strong compactness argument(see Lemma 4 in \cite{Simon1990}) that
	$u^\kappa$ is compact in $C(0, T_0; H^1_{loc})$.
	In particular, there exists a sequence $\kappa_n \rightarrow 0^+$
	and $u \in C(0, T_0; H^1_{loc})$ such that
	\beq\label{4110}
	u^{\kappa_n }  \rightarrow  u
	\quad \text{in} \ C(0, T_0; H^1_{loc}) \
	\text{as} \ {\kappa_n} \rightarrow 0^+.
	\deq
	Similarly, due to the density equation $\eqref{new-ns-4}_1$
	and estimate \eqref{4107}, we can obtain that
	$\phi^\kappa$ is bounded uniformly in $L^\infty(0, T_0; H^3_{loc})$, and
	$\p_t \phi^\kappa$ is bounded uniformly in $L^2(0, T_0; H^2_{loc})$ respectively.
	Then, it follows from a strong compactness argument (see Lemma 4 in \cite{Simon1990}) that
	$\phi^\kappa$ is compact in $C(0, T_0; H^2_{loc})$.
	In particular, there exists a sequence $\kappa_n \rightarrow 0^+$
	and $\phi  \in C(0, T_0; H^2_{loc})$ such that
	\beq\label{4111}
	\phi^{\kappa_n }  \rightarrow  \phi
	\quad \text{in} \ C(0, T_0; H^2_{loc}) \
	\text{as} \ {\kappa_n} \rightarrow 0^+.
	\deq
	Due to the estimates \eqref{21001}, \eqref{4107} and equation $\eqref{new-ns-4}_2$,
	one can obtain the uniform estimate for the pressure
	\beq\label{estimate-p}
	\underset{0\le t \le T_0}{\sup}\|\nabla p^{\kappa_n }(t)\|_{H^1}^2
	+\int_0^{T_0} (\|\nabla^3 p^{\kappa_n } \|_{L^2}^2 + \|\p_t \nabla  p^{\kappa_n } \|_{L^2}^2) d t
	\le C_{\phi_0, u_0, p_0, \gamma}.
	\deq
	Therefore, the combination of uniform estimate of
	$(\phi^{\kappa_n }, u^{\kappa_n }, p^{\kappa_n })$
	and strong convergence of $(\phi^{\kappa_n }, u^{\kappa_n })$
	(see \eqref{4110} and \eqref{4111}), one can obtain the
	solution $(\phi, u, p)$ satisfying the system \eqref{new-ns}.
	Using the characteristic line method
	and the initial data assumption $\phi_0>0$ again gives
	\beqq
	\phi(x, t)>0, \quad (x, t)\in \mathbb{R}^2\times [0, T_0].
	\deqq
	Finally, due to the uniform estimates \eqref{4107}, \eqref{estimate-p} and semi-continuity of norm, we have
	\beqq
	E(\phi, u, 0)(T_0)
	+\underset{0\le t \le T_0}{\sup}\|\nabla p (t)\|_{H^1}^2
	+\int_0^{T_0} (\mathcal{D}(\phi, u, 0)(t)
	+\|\nabla^3 p(t)\|_{L^2}^2 + \|\p_t \nabla  p(t)\|_{L^2}^2)d t
	\le C_{\phi_0, u_0, p_0, \gamma},
	\deqq
	which implies the estimate \eqref{main-estimate} in Theorem \ref{local-well}.
	
	\subsection{Uniqueness of solution of original equation}
	
	In this subsection, we will show the uniqueness of the solution to the original system \eqref{o-ns}. Let $(u^1,\phi^1,p^1)$ and $(u^2,\phi^2,p^2)$ be two positive solutions in the existence time $[0, T_0]$, constructed in the previous subsection, with respect to the initial data $(u^1_0,\phi^1_0,p^1_0)$  and $(u^2_0,\phi^2_0,p^2_0)$ respectively.
	
	Let us set
	$$(\bar{u},\bar{\phi},\bar{p}) \overset{\text{def}}{=} (u^1,\phi^1,p^1)-(u^2,\phi^2,p^2),$$
	then they satisfy the following evolution
	
	\beq\label{bar-equation}
	\left\{\begin{aligned}
		& \p_t \bar{\phi}+ \bar{u} \cdot \nabla \phi^1-  u^2 \cdot \nabla \bar{\phi}=0,\\
		&(\phi^1)^2 \p_t \bar{u} +(\phi^1+\phi^2)\bar{\phi} \p_t u^2  +(\phi^1)^2 u^1 \cdot \nabla \bar{u} +(\phi^1)^2 \bar{u} \cdot \nabla u^2+  (\phi^1+\phi^2) \bar{\phi} u^2 \cdot \nabla u^2 - \p_1^2 \bar{u} +\nabla \bar{p}=0,\\
		&{\rm div}\bar{u}=0,
	\end{aligned}\right.
	\deq
	with the initial data
	$$
	(\bar{u},\bar{\phi},\bar{p})|_{t=0}=(0,0,0),
	$$
	and
	$$
	\lim_{|x|\to +\infty}(\bar{u},\bar{\phi}) =(0,0).
	$$
	Here we assume the two solutions have the same initial data.
	
	Using the standard energy methods to \eqref{bar-equation}, we can establish the following estimate, which we omit the proof for brevity of presentation.
	
	\begin{prop}
		Let $(u^1,\phi^1,p^1)$ and $(u^2,\phi^2,p^2)$ be two solutions of the system \eqref{new-ns} in the existence time $[0, T_0]$ with respect to the same initial data. Then there exists a positive constant
		$C$ independent of $T_0$
		such that $(\bar{u},\bar{\phi},\bar{p})$ satisfies
		\beq\label{bar-estimate}
		\begin{aligned}
			&\|(\phi^1 \bar{u},\bar{\phi}\wx)\|_{L^2}^2
			+ \int_{0}^{t} \| \p_1 \bar{u} \|_{L^2}^2 d\tau\\
			\leq &~ C \|(\p_1 u^1, \p_1 u^2, \p_{12}u^1, \p_{12}u^2, \p_{1}^2 u^2, \p_{1}\p_{2}^2 u^2, \p_1 \p_t u^2, \phi^2 \p_t u^2, \phi^2 \p_2 \p_t u^2)\|_{L^2}^2 \\
			&~\times ( \| (\phi^1,\phi^2) \wx^2 \|_{H^3}^2 + \| (\phi^1,\phi^2) \wx^3 \|_{H^2}^2) \int_{0}^{t} \|(\phi^1 \bar{u} , \bar{\phi}\wx )\|_{L^2}^2  d\tau,
		\end{aligned}
		\deq
		for all $t \leq T_0$.
	\end{prop}
	Then we prove the uniqueness of the solution to the system \eqref{o-ns} as follows.
	
	\begin{proof}[\textbf{Proof of uniqueness.}]
		~By virtue of the estimate \eqref{main-estimate} in Theorem \ref{local-well} for the two solutions $(u^1,\phi^1,p^1)$ and $(u^2,\phi^2,p^2)$, it is easy to check that
		\beqq
		\|(\p_1 u^1, \p_1 u^2, \p_{12}u^1, \p_{12}u^2, \p_{1}^2 u^2,
		\p_{1}\p_2^2 u^2,  \p_1 \p_t u^2, \phi^2 \p_t u^2, \phi^2 \p_2 \p_t u^2)\|_{L^2}^2 \leq C_{\phi_0,u_0,p_0,\gamma},
		\deqq
		and
		\beqq
		\| (\phi^1,\phi^2) \wx^2 \|_{H^3}^2 + \| (\phi^1,\phi^2) \wx^3 \|_{H^2}^2  \leq C_{\phi_0,u_0,p_0,\gamma}.
		\deqq
		Then from \eqref{bar-estimate}, we can get
		\beq\label{bar-estimate-new}
		\begin{aligned}
			\|(\phi^1 \bar{u},\bar{\phi}\wx)\|_{L^2}^2
			+ \int_{0}^{t} \|\p_1 \bar{u} \|_{L^2}^2 d\tau
			\leq &~ C_{\phi_0,u_0,p_0,\gamma}
			\int_{0}^{t} \|(\phi^1 \bar{u} , \bar{\phi}\wx)\|_{L^2}^2 d\tau,
		\end{aligned}
		\deq
		with the initial data
		$$(\phi^1\bar{u},\bar{\phi}\wx)|_{t=0}=0,$$
		for $\phi^1>0$.
		By applying Gr\"{o}nwall inequality to \eqref{bar-estimate-new}, we can obtain $(\phi^1\bar{u},\bar{\phi}\wx)=0$, which implies $(\bar{u},\bar{\phi})=0$.
		Using the equation $\eqref{bar-equation}_2$, it holds 
		$$
		-\nabla \bar{p}
		= (\phi^1)^2 \p_t \bar{u} +(\phi^1+\phi^2)\bar{\phi} \p_t u^2
		+(\phi^1)^2 u^1 \cdot \nabla \bar{u} +(\phi^1)^2 \bar{u} \cdot \nabla u^2
		+(\phi^1+\phi^2) \bar{\phi} u^2 \cdot \nabla u^2 - \p_1^2 \bar{u}
		=0.
		$$
		Therefore we complete the proof of the uniqueness of the system \eqref{new-ns}, which implies the uniqueness of the original system \eqref{o-ns}.
	\end{proof}
	
	\section{Global well-posedness and exponential stability}
	\label{section-global-well-poseness}
	
	In this section, we will establish the global-in-time well-posedness
	for the system \eqref{new-ns} under the condition of small initial data.
	In subsection \ref{gloabl-priori}, we will establish the closed a priori
	estimate by the energy method.
	Then, we apply the continuity argument to establish the global
	well-posedness in subsection \ref{continuity-arguement}.

	\subsection{A priori estimate}\label{gloabl-priori}
	
	First of all, let us define the energy norms
	\beqq
	\begin{aligned}
		\es(t)
		\overset{\text{def}}{=}
		&\sum_{k=0}^3 \|\phi \p_2^k u\|_{L^2}^2
		+\sum_{k=0}^2 \|\p_2^k \p_1 u\|_{L^2}^2
		+\sum_{k=0}^1\|\phi \p_t \p_2^k u\|_{L^2}^2
		+\|\p_t \p_1 u\|_{L^2}^2,
	\end{aligned}
	\deqq
	\beqq
	\begin{aligned}
		\eb(t)
		\overset{\text{def}}{=}
		&\sum_{0\le |\alpha|\le 2}\|\p_x^\alpha \phi \wx^{\gm+1}\|_{L^2}^2
		+\sum_{|\alpha|=3}\|\p_x^\alpha \phi \wx^{\gm}\|_{L^2}^2
		+\sum_{1\le |\alpha| \le 2}\|\p_x^\alpha \ln \phi\|_{L^2}^2,
	\end{aligned}
	\deqq
	and the dissipative part
	\beqq
	\begin{aligned}
		\widehat{\mathcal{D}}(t)\overset{\text{def}}{=}
		&\sum_{k=0}^2 \|\p_1 \p_2^k u\|_{L^2}^2
		+\|\p_1 \p_2^3 u\|_{L^2}^2
		+\sum_{k=0}^2 \|\phi \p_2^k  \p_t u\|_{L^2}^2
		+\|\p_1 \p_t u\|_{L^2}^2
		+\|\p_{12} \p_t u\|_{L^2}^2 +\|\phi \p_t^2 u\|_{L^2}^2.
	\end{aligned}
	\deqq
	For any $T>0$, assume the following estimates
	\beq\label{as1}
	\eb (t)\le 4 \eb (0)\overset{\text{def}}{=}M,
	\deq
	and
	\beq\label{as2}
	\es(t) \le 4 \sqrt{\es(0)},
	\deq
	hold for all $t \le T$.
	In order to establish the global well-posedness and exponential
	stability, we will assume that the initial data $\es(0)$ being small enough.
	In the sequence, under the smallness of initial data $\es(0)$, our target is
	to establish the following estimates
	\beqq
	\eb (t)\le 2 \eb (0), \quad \es(t) \le 2 \sqrt{\es(0)},
	\deqq
	for all $t \le T$.
	First of all, let us state the first estimate.
	
	\begin{lemm}\label{lemma41}
		Under the assumption \eqref{as1},
		for any smooth solution of equation \eqref{new-ns}, it holds 
		\beq\label{44101}
		\|\phi u\|_{L^2}^2 e^{\gamma t}
		+\int_0^t \|\p_1 u\|_{L^2}^2 e^{\gamma \tau} d\tau
		\le \|\phi_0 u_0\|_{L^2}^2,
		\deq
		where $\gamma \overset{\text{def}}{=} 2^{-5}
		M^{-2}$.
	\end{lemm}
	\begin{proof}
		From the equation \eqref{2201}, we have
		\beq\label{44102}
		\frac{d}{dt} \|\phi u\|_{L^2}^2+2\|\p_1 u\|_{L^2}^2=0.
		\deq
		Using the Hardy inequality and the assumption \eqref{as1}, we have
		\beqq
		\|\phi u\|_{L^2}
		\le \|\phi \wx \|_{L^\infty}\|u \wx^{-1}\|_{L^2}
		\le  {4 }\|\phi \wx \|_{L^\infty}\|\p_1 u \|_{L^2}
		\le  {4M}\|\p_1 u \|_{L^2},
		\deqq
		which, yields directly
		\beq\label{44103}
		2^{-4} M^{-2} \|\phi u\|_{L^2}^2 \le  \|\p_1 u \|_{L^2}^2.
		\deq
		Substituting the estimate \eqref{44103} into \eqref{44102}, we have
		\beqq
		\frac{d}{dt}\|\phi u\|_{L^2}^2+2^{-5} M^{-2}\|\phi u\|_{L^2}^2
		+\|\p_1 u\|_{L^2}^2\le 0.
		\deqq
		Multiplying the above estimate by $e^{2^{-5} M^{-2} t}$, it holds 
		\beqq
		\frac{d}{dt}\left\{\|\phi u\|_{L^2}^2 e^{2^{-5} M^{-2} t}\right\}
		+\|\p_1 u\|_{L^2}^2 e^{2^{-5} M^{-2} t} \le 0,
		\deqq
		which, integrating over $[0, t]$, we have
		\beqq
		\|\phi u\|_{L^2}^2 e^{2^{-5} M^{-2} t}
		+\int_0^t \|\p_1 u\|_{L^2}^2 e^{2^{-5} M^{-2} \tau} d\tau
		\le \|\phi_0 u_0\|_{L^2}^2.
		\deqq
		Therefore, we complete the proof this lemma.
	\end{proof}

	\begin{lemm}\label{lemma42}
		Under the assumptions \eqref{as1}-\eqref{as2},
		for any smooth solution of equation \eqref{new-ns}, it holds 
		\beq\label{4201}
		\|\p_1 u\|_{L^2}^2 e^{\gamma t}
		+\int_0^t \|\phi \p_t u\|_{L^2}^2 e^{\gamma \tau}d\tau
		\le  \|\p_1 u_0\|_{L^2}^2
		+C_M \|\phi_0 u_0\|_{L^2}^2.
		\deq
	\end{lemm}
	\begin{proof}
		From the equation \eqref{22002} and estimate \eqref{22003}, we have
		\beqq
		\frac{d}{dt}\|\p_1 u\|_{L^2}^2 +\|\phi \p_t u\|_{L^2}^2
		\lesssim \|\phi \wx^2\|_{L^\infty}^2 \|\p_1 u \|_{L^2}^2
		(\|\p_1 u \|_{L^2} \|\p_{12} u \|_{L^2}+\|\p_{12} u \|_{L^2}^2),
		\deqq
		which, together with the assumptions \eqref{as1} and \eqref{as2}, yields directly
		\beqq
		\frac{d}{dt}\|\p_1 u\|_{L^2}^2 +\|\phi \p_t u\|_{L^2}^2
		\le  C \eb \es \|\p_1 u \|_{L^2}^2
		\le  C_M \|\p_1 u \|_{L^2}^2.
		\deqq
		Multiplying the above estimate by $e^{\gamma t}$, it holds 
		\beqq
		\frac{d}{dt}\left(\|\p_1 u\|_{L^2}^2 e^{\gamma t}\right)
		+\|\phi \p_t u\|_{L^2}^2 e^{\gamma t}
		\le  \gamma \|\p_1 u\|_{L^2}^2 e^{\gamma t}+C_M\|\p_1 u \|_{L^2}^2 e^{\gamma t}.
		\deqq
		Then, integrating over $[0, t]$ and using the estimates
		\eqref{as1}-\eqref{as2}, we have
		\beqq
		\|\p_1 u\|_{L^2}^2 e^{\gamma t}
		+\int_0^t \|\phi \p_t u\|_{L^2}^2 e^{\gamma \tau}d\tau
		\le  \|\p_1 u_0\|_{L^2}^2
		+C_M \int_0^t \|\p_1 u\|_{L^2}^2 e^{\gamma \tau}d\tau,
		\deqq
		which, together with \eqref{44101}, it holds
		\beqq
		\|\p_1 u\|_{L^2}^2 e^{\gamma t}
		+\int_0^t \|\phi \p_t u\|_{L^2}^2 e^{\gamma \tau}d\tau
		\le  \|\p_1 u_0\|_{L^2}^2
		+C_M \|\phi_0 u_0\|_{L^2}^2.
		\deqq
		Therefore, we complete the proof this lemma.
	\end{proof}
	
	\begin{lemm}\label{lemma43}
		Under the assumptions \eqref{as1}-\eqref{as2},
		for any smooth solution of equation \eqref{new-ns}, it holds 
		\beq\label{4301}
		\begin{aligned}
			&(\|\phi \p_2 u\|_{L^2}^2+\|\phi \p_t u\|_{L^2}^2
			+\|\p_2 \p_1 u\|_{L^2}^2+\| \phi \p_2^2 u\|_{L^2}^2)e^{\gamma t}\\
			&+\int_0^t(\|\p_{12} u\|_{L^2}^2+\|\p_1 \p_t u\|_{L^2}^2
			+\|\phi \p_2 \p_t u\|_{L^2}^2
			+\|\p_1 \p_2^2 u\|_{L^2}^2)e^{\gamma \tau}d\tau\\
			\le &~\|\phi_0 \p_2^2 u_0\|_{L^2}^2
			+C_M(\|\p_2 \p_1 u_0\|_{L^2}^2
			+\|(\phi \p_t u)|_{t=0}\|_{L^2}^2
			+\| \phi_0 \p_2 u_0\|_{L^2}^2+\|\p_1 u_0\|_{L^2}^2
			+\|\phi_0 u_0\|_{L^2}^2).
		\end{aligned}
		\deq
	\end{lemm}
	\begin{proof}
		From the analysis of estimate $\eqref{2301}_1$, it is easy to check that
		\beq\label{4302}
		\begin{aligned}
			&\frac{d}{dt} \| \phi \p_2 u\|_{L^2}^2+\frac{3}{2}\|\p_1 \p_2 u\|_{L^2}^2\\
			\lesssim &~\|\phi \p_t u\|_{L^2}^2 \|(\p_2 \phi,\p_2^2 \phi)\wx\|_{L^2}^2
			+\|\phi \p_2 u\|_{L^2}^2 \|\p_1 u\|_{L^2}^2
			\|(\p_2 \phi, \p_{12} \phi, \p_{2}^2 \phi )\wx^2 \|_{L^2} \\
			&+\|\phi \p_2 u\|_{L^2}^4 \|\p_1 u\|_{L^2}^2
			(\|\nabla \phi \wx^2\|_{L^2}^2\|\p_2 \nabla \phi \wx^2\|_{L^2}^2
			+\|\phi \wx\|_{L^\infty}^4).
		\end{aligned}
		\deq
		Due to the Hardy inequality, we have
		\beqq
		\|\phi \p_2 u\|_{L^2}
		\le \|\phi \wx \|_{L^\infty}\|\p_2 u \wx^{-1}\|_{L^2}
		\le  4\|\phi \wx \|_{L^\infty}\|\p_{12} u \|_{L^2}
		\le  {4M}\|\p_{12} u \|_{L^2},
		\deqq
		which, yields directly
		\beq\label{4303}
		2^{-4} M^{-2} \|\phi \p_2 u\|_{L^2}^2 \le  \|\p_{12} u \|_{L^2}^2.
		\deq
		The combination of \eqref{4302} and \eqref{4303} gives directly
		\beqq
		\begin{aligned}
			\frac{d}{dt} \| \phi \p_2 u\|_{L^2}^2
			+\gamma \|\phi \p_2 u\|_{L^2}^2+\|\p_{12} u\|_{L^2}^2
			\le C\eb\|\phi \p_t u\|_{L^2}^2
			+C(\eb^{\f12} \es+\eb \es+\eb^2 \es^2)\|\p_1 u\|_{L^2}^2,
		\end{aligned}
		\deqq
		which, yields directly
		\beqq
		\begin{aligned}
			\frac{d}{dt}(\| \phi \p_2 u\|_{L^2}^2 e^{\gamma t})
			+\|\p_{12} u\|_{L^2}^2 e^{\gamma t}
			\le C\eb\|\phi \p_t u\|_{L^2}^2 e^{\gamma t}
			+C(\eb^{\f12} \es+\eb \es+\eb^2 \es^2)\|\p_1 u\|_{L^2}^2 e^{\gamma t}.
		\end{aligned}
		\deqq
		Integrating over $[0, t]$ and using estimates \eqref{44101}
		and \eqref{4201}, we have
		\beq\label{4304}
		\begin{aligned}
			&\| \phi \p_2 u\|_{L^2}^2 e^{\gamma t}
			+\int_0^t\|\p_{12} u\|_{L^2}^2 e^{\gamma \tau}d\tau\\
			\le &~\| \phi_0 \p_2 u_0\|_{L^2}^2
			+C_M\int_0^t \|\phi \p_t u\|_{L^2}^2 e^{\gamma \tau}d\tau
			+C_M\int_0^t \|\p_1 u\|_{L^2}^2 e^{\gamma \tau}d\tau\\
			\le &~\| \phi_0 \p_2 u_0\|_{L^2}^2
			+C_M (\|\p_1 u_0\|_{L^2}^2+\|\phi_0 u_0\|_{L^2}^2).
		\end{aligned}
		\deq
		Similarly, from the analysis of \eqref{2401}, it is easy to check that
		\beq\label{4305}
		\begin{aligned}
			&\frac{d}{dt}\|\phi \p_t u\|_{L^2}^2+\frac{3}{2}\|\p_1 \p_t u\|_{L^2}^2 \\
			\lesssim &~\|\phi \p_t u\|_{L^2}^2 \|(\p_1 u, \p_{12}u)\|_{L^2}^2
			\|\nabla \phi \wx^2\|_{L^2} \|\p_2 \nabla \phi \wx^2\|_{L^2}\\
			&+\|\p_1 u\|_{L^2}^2 \|\p_{12} u\|_{L^2}^2 \|\p_1 \nabla u\|_{L^2}^2
			\|\nabla \phi \wx^2\|_{L^2}\|\p_2 \nabla \phi \wx^2\|_{L^2}
			\|\phi \wx^2\|_{L^\infty}^2\\
			&+\|\phi \p_t u\|_{L^2}^2 \|(\p_1 u, \p_{12}u)\|_{L^2}
			\|(\p_{12} u, \p_{122}u)\|_{L^2}
			\|\phi \wx^2\|_{L^\infty}^2,
		\end{aligned}
		\deq
		which yields directly
		\beqq
		\frac{d}{dt}\|\phi \p_t u\|_{L^2}^2+\frac{3}{2}\|\p_1 \p_t u\|_{L^2}^2
		\lesssim \eb \es\|\phi \p_t u\|_{L^2}^2
		+\es \eb^2(\es^3 \eb^4+\es^2\eb^2+\es \eb+\es)\|\p_1 u\|_{L^2}^2,
		\deqq
		where we have used the estimate
		\beqq
		\|\p_1^2 u\|_{L^2}\lesssim \es^{\frac{3}{2}}\eb^2+\es\eb
		+\es^{\frac{1}{2}}\eb^{\frac{1}{2}}.
		\deqq
		Thus, due to the assumptions \eqref{as1} and \eqref{as2},
		the estimate \eqref{4305} yields
		\beq\label{4306}
		\begin{aligned}
			&\|\phi \p_t u\|_{L^2}^2 e^{\gamma t}
			+\int_0^t \|\p_1 \p_t u\|_{L^2}^2 e^{\gamma \tau}d\tau\\
			\le &~\|(\phi \p_t u)|_{t=0}\|_{L^2}^2
			+\gamma \int \|\phi \p_t u\|_{L^2}^2 e^{\gamma \tau}d\tau
			+C_M \int_0^t(\|\phi \p_t u\|_{L^2}^2+\|\p_1 u\|_{L^2}^2)
			e^{\gamma \tau}d\tau\\
			\le &~\|(\phi \p_t u)|_{t=0}\|_{L^2}^2
			+C_M (\|\p_1 u_0\|_{L^2}^2+\|\phi_0 u_0\|_{L^2}^2),
		\end{aligned}
		\deq
		where we have used the estimates \eqref{44101} and \eqref{4201}.
		Similarly, from the analysis of estimate \eqref{2501}, we have
		\beqq
		\begin{aligned}
			&\frac{d}{dt} \|\p_2 \p_1 u\|_{L^2}^2 +\frac32\|\phi \p_2 \p_t u\|_{L^2}^2\\
			\lesssim &~\|(\p_2 \phi, \p_2^2 \phi)\wx\|_{L^2}^2\|\p_1 \p_t u\|_{L^2}^2
			+\|(\p_1 u, \p_{12}u)\|_{L^2}^2\|(\p_{12}u, \p_{122}u)\|_{L^2}^2
			\|\phi \wx^2\|_{L^\infty}^2\\
			&+\|(\p_1 u, \p_{12}u)\|_{L^2}^3
			\|(\p_{12}u, \p_{122}u)\|_{L^2}
			(\|(\p_2 \phi, \p_{12}\phi)\wx^2\|_{L^2}^2+\|\phi \wx^2\|_{L^\infty}^2)\\
			\lesssim &~\eb \|\p_1 \p_t u\|_{L^2}^2
			+\eb \es \|(\p_1 u, \p_{12}u)\|_{L^2}^2,
		\end{aligned}
		\deqq
		which yields directly
		\beq\label{4307}
		\begin{aligned}
			&\|\p_2 \p_1 u\|_{L^2}^2 e^{\gamma t}
			+\int_0^t \|\phi \p_2 \p_t u\|_{L^2}^2 e^{\gamma \tau}d\tau\\
			\le &~ \|\p_2 \p_1 u_0\|_{L^2}^2
			+\gamma \int \|\p_2 \p_1 u\|_{L^2}^2 e^{\gamma \tau}d\tau
			+C_M \int_0^t (\|\p_1 \p_t u\|_{L^2}^2
			+\|(\p_1 u, \p_{12}u)\|_{L^2}^2) e^{\gamma \tau}d\tau\\
			\le &~\|\p_2 \p_1 u_0\|_{L^2}^2
			+C_M(\|(\phi \p_t u)|_{t=0}\|_{L^2}^2
			+\| \phi_0 \p_2 u_0\|_{L^2}^2+\|\p_1 u_0\|_{L^2}^2
			+\|\phi_0 u_0\|_{L^2}^2).
		\end{aligned}
		\deq
		Finally, from the analysis of estimate $\eqref{2301}_2$, we have
		\beqq
		\begin{aligned}
			&\frac{d}{dt} \| \phi \p_2^2 u\|_{L^2}^2
			+\frac74\|\p_1 \p_2^2 u\|_{L^2}^2\\
			\lesssim &~(\eb+\eb \es)\|(\p_1 u, \p_{12}u)\|_{L^2}^2
			+\eb \es^{\f12} \|\p_{122}u\|_{L^2}^2
			+(\eb+\eb^2)(\|\phi \p_2 \p_t u\|_{L^2}^2+\|\p_1 \p_t u\|_{L^2}^2).
		\end{aligned}
		\deqq
		Due to the smallness of initial data $\es(0)$, we have
		\beqq
		\begin{aligned}
			\frac{d}{dt} \| \phi \p_2^2 u\|_{L^2}^2
			+\frac32\|\p_1 \p_2^2 u\|_{L^2}^2
			\lesssim \|(\p_1 u, \p_{12}u)\|_{L^2}^2
			+\|\phi \p_2 \p_t u\|_{L^2}^2+\|\p_1 \p_t u\|_{L^2}^2,
		\end{aligned}
		\deqq
		which yields directly
		\beq\label{4308}
		\begin{aligned}
			&\| \phi \p_2^2 u\|_{L^2}^2 e^{\gamma t}
			+\int_0^t \|\p_1 \p_2^2 u\|_{L^2}^2 e^{\gamma \tau} d\tau\\
			\le
			&~\|\phi_0 \p_2^2 u_0\|_{L^2}^2
			+C_M(\|\p_2 \p_1 u_0\|_{L^2}^2
			+\|(\phi \p_t u)|_{t=0}\|_{L^2}^2
			+\| \phi_0 \p_2 u_0\|_{L^2}^2+\|\p_1 u_0\|_{L^2}^2
			+\|\phi_0 u_0\|_{L^2}^2).
		\end{aligned}
		\deq
		Then, the combination of estimates \eqref{4304},
		\eqref{4306}, \eqref{4307} and \eqref{4308} yield
		estimate \eqref{4301}.
		Thus, we complete the proof of this lemma.
	\end{proof}
	
	\begin{lemm}\label{lemma44}
		Under the assumptions \eqref{as1}-\eqref{as2},
		for any smooth solution of equation \eqref{new-ns}, it holds 
		\beq\label{4401}
		\begin{aligned}
			&(\|\phi \p_2 \p_t u\|_{L^2}^2 +\|\p_1 \p_t u\|_{L^2}^2) e^{\kappa t}
			+\int_0^t(\|\phi \p_t^2 u\|_{L^2}^2
			+\|\p_{12}\p_t u\|_{L^2}^2)e^{\kappa \tau} d\tau\\
			\le
			&~\|(\phi \p_2 \p_t u, \p_1 \p_t u)|_{t=0}\|_{L^2}^2
			+C_M(\|(\phi \p_t u)|_{t=0}\|_{L^2}^2
			+\|(\phi_0 u_0, \p_1 u_0, \p_2 \p_1 u_0,
			\phi_0 \p_2 u_0, \phi_0 \p_2^2 u_0)\|_{L^2}^2.
		\end{aligned}
		\deq
	\end{lemm}
	\begin{proof}
		Similar to the estimate \eqref{2701},
		\beq\label{4402}
		\begin{aligned}
			\frac{d}{dt}\|\phi \p_2 \p_t u\|_{L^2}^2
			+\frac{7}{4}\|\p_{12}\p_t u\|_{L^2}^2
			\le \sigma_3 \|\phi \p_t^2 u\|_{L^2}^2
			+C_M  \|(\phi \p_2\p_t u,\p_{1} \p_t u, \p_1 u, \p_{12}u, \p_{1}\p_2^2 u )\|_{L^2}^2,
		\end{aligned}
		\deq
		which can be found in detail in appendix \ref{claim-estimate}.
		
		From the analysis of estimate  \eqref{2601},
		it is easy to check that
		\beq\label{4403}
		\begin{aligned}
			\frac{d}{dt}\|\p_1 \p_t u\|_{L^2}^2
			+\frac{3}{2}\|\phi \p_t^2 u\|_{L^2}^2
			\le
			C_M \es \|\p_{12}\p_t u\|_{L^2}^2 +C_M \|(\p_1 \p_t u, \p_1 u, \p_{12} u)\|_{L^2}^2.
		\end{aligned}
		\deq
		Due to the smallness of $\sigma_3$
		and initial data $\es(0)$, the combination of estimates
		\eqref{4402} and \eqref{4403} yields
		\beqq
		\begin{aligned}
			&\frac{d}{dt}\|\phi \p_2 \p_t u\|_{L^2}^2
			+\frac{3}{2}\|\p_{12}\p_t u\|_{L^2}^2
			+\frac{d}{dt} \|\p_1 \p_t u\|_{L^2}^2
			+\|\phi \p_t^2 u\|_{L^2}^2
			\le C_M \|(\p_1 \p_t u,  \phi \p_2\p_t u,
			\p_1 u, \p_{12}u, \p_{1}\p_2^2 u)\|_{L^2}^2,
		\end{aligned}
		\deqq
		which gives directly
		\beqq
		\begin{aligned}
			&\|\phi \p_2 \p_t u\|_{L^2}^2 e^{\gamma t}
			+\|\p_1 \p_t u\|_{L^2}^2 e^{\gamma t}
			+\int_0^t(\|\phi \p_t^2 u\|_{L^2}^2
			+\|\p_{12}\p_t u\|_{L^2}^2)e^{\gamma \tau} d\tau\\
			\le &~\|(\phi \p_2 \p_t u)|_{t=0}\|_{L^2}^2+\|(\p_1 \p_t u)|_{t=0}\|_{L^2}^2
			+\gamma \int_0^t \|\p_1 \p_t u\|_{L^2}^2 e^{\gamma \tau} d\tau
			+C_M \int_0^t \|\phi \p_2 \p_t u\|_{L^2}^2 e^{\gamma \tau} d\tau\\
			&+C_M \int_0^t \|(\p_1 \p_t u, \phi \p_2\p_t u,
			\p_1 u, \p_{12}u)\|_{L^2}^2 e^{\gamma \tau} d\tau.
		\end{aligned}
		\deqq
		This, together with estimates \eqref{44101} and \eqref{4301}, gives
		\beqq
		\begin{aligned}
			&\|\phi \p_2 \p_t u\|_{L^2}^2 e^{\gamma t}
			+\|\p_1 \p_t u\|_{L^2}^2 e^{\gamma t}
			+\int_0^t(\|\phi \p_t^2 u\|_{L^2}^2
			+\|\p_{12}\p_t u\|_{L^2}^2)e^{\gamma \tau} d\tau\\
			\le
			&~\|(\phi \p_2 \p_t u, \p_1 \p_t u)|_{t=0}\|_{L^2}^2
			+C_M(\|(\phi \p_t u)|_{t=0}\|_{L^2}^2
			+\|(\phi_0 u_0, \p_1 u_0, \p_2 \p_1 u_0,
			\phi_0 \p_2 u_0, \phi_0 \p_2^2 u_0)\|_{L^2}^2).
		\end{aligned}
		\deqq
		Therefore, we complete the proof of this lemma.
	\end{proof}
	
	\begin{lemm}\label{lemma45}
		Under the assumptions \eqref{as1}-\eqref{as2},
		for any smooth solution of equation \eqref{new-ns}, it holds 
		\beq\label{4501}
		\begin{aligned}
			(\|\phi \p_2^3 u\|_{L^2}^2+\|\p_1 \p_2^2 u\|_{L^2}^2) e^{\gamma t}
			+\int_0^t(\|\p_{1}\p_2^3 u\|_{L^2}^2 +\|\phi \p_t \p_2^2 u\|_{L^2}^2 )e^{\gamma \tau}d\tau
			\le C_M \es(0).
		\end{aligned}
		\deq
	\end{lemm}
	\begin{proof}
		From the analysis of estimates \eqref{23001} and \eqref{2801}, we have
		\beq\label{4502}
		\begin{aligned}
			\frac{d}{dt}\|\phi \p_2^3 u\|_{L^2}^2+\frac{7}{4}\|\p_{1}\p_2^3 u\|_{L^2}^2
			\le  C_M\|\phi \p_t \p_2^2 u\|_{L^2}^2
			+C_M(\|(\p_1 \p_t u,\p_{12}\p_t u, \p_1 u, \p_{12}u, \p_{122} u)\|_{L^2}^2,
		\end{aligned}
		\deq
		and
		\beq\label{4503}
		\begin{aligned}
			\frac{d}{dt}\|\p_1 \p_2^2 u\|_{L^2}^2+\|\phi \p_t \p_2^2 u\|_{L^2}^2
			\le   C_M \sqrt{\es(0)}\|\p_1 \p_2^3 u\|_{L^2}^2
			+C_M\|(\p_1 \p_t u, \p_{12}\p_t u,\phi \p_t \p_2 u,
			\p_1 u, \p_{12}u)\|_{L^2}^2.
		\end{aligned}
		\deq
		Multiplying the estimate \eqref{4503} by $C_M+1$ and adding with
		the estimate \eqref{4502}, we apply the smallness of $\es(0)$ to obtain
		\beqq
		\begin{aligned}
			&\frac{d}{dt}\|\phi \p_2^3 u\|_{L^2}^2+\frac{3}{2}\|\p_{1}\p_2^3 u\|_{L^2}^2
			+(C_M+1)\frac{d}{dt}\|\p_1 \p_2^2 u\|_{L^2}^2+\|\phi \p_t \p_2^2 u\|_{L^2}^2\\
			\le &~
			C_M\|(\p_1 \p_t u, \p_{12}\p_t u, \phi \p_t \p_2 u,
			\p_1 u, \p_{12}u, \p_{122} u)\|_{L^2}^2,
		\end{aligned}
		\deqq
		which yields directly
		\beqq
		\begin{aligned}
			&\|\phi \p_2^3 u\|_{L^2}^2 e^{\gamma t}
			+C_M\|\p_1 \p_2^2 u\|_{L^2}^2 e^{\gamma t}
			+\int_0^t(\|\p_{1}\p_2^3 u\|_{L^2}^2
			+\|\phi \p_t \p_2^2 u\|_{L^2}^2 )e^{\gamma \tau}d\tau\\
			\le &~\|\phi \p_2^3 u_0\|_{L^2}^2
			+C_M\|\p_1 \p_2^2 u_0\|_{L^2}^2
			+C_M\int_0^t \|(\p_1 \p_t u, \p_{12}\p_t u, \phi \p_t \p_2 u,
			\p_1 u, \p_{12}u, \p_{122} u)\|_{L^2}^2 e^{\gamma \tau}d\tau\\
			\le &~\|\phi \p_2^3 u_0\|_{L^2}^2
			+C_M(\|(\phi \p_2 \p_t u, \p_1 \p_t u, \phi \p_t u)|_{t=0}\|_{L^2}^2
			+\|(\phi_0 u_0, \p_1 u_0, \p_2 \p_1 u_0,\p_1 \p_2^2 u_0,
			\phi_0 \p_2 u_0, \phi_0 \p_2^2 u_0)\|_{L^2}^2),
		\end{aligned}
		\deqq
		where we have used the estimates \eqref{44101},
		\eqref{44103} and \eqref{4401}.
		Therefore, we complete the proof of this lemma.
	\end{proof}
	
	Therefore, the combination of estimates \eqref{44101},
	\eqref{4201}, \eqref{4301}, \eqref{4401} and \eqref{4501}
	yields the following estimate.
	\begin{coro}
		Under the assumptions \eqref{as1}-\eqref{as2},
		for any smooth solution of equation \eqref{new-ns}, it holds 
		\beq\label{4701}
		\es(t) e^{\kappa t}+\int_0^t \widehat{\mathcal{D}}(\tau )e^{\kappa \tau} d\tau \le C_M \es(0).
		\deq
	\end{coro}
	
	Similar to the Lemma \ref{lemma210}, we have the following estimate.
	\begin{lemm}\label{lemma46}
		Under the assumptions \eqref{as1}-\eqref{as2},
		for any smooth solution of equation \eqref{new-ns}, it holds 
		\beq\label{4601}
		\|\p_1^2 u\|_{L^2}
		+\|\p_2^2 u_1\|_{L^2}+\|\p_1^2 \p_2 u\|_{L^2}
		+\|\p_1^3 u\|_{L^2}
		+\|\p_1^2 w \wx^3\|_{L^2}+\|\p_1^2 \p_2 w \wx^2\|_{L^2}
		\le C_{M} \sqrt{\widehat{\mathcal{D}}(t)}.
		\deq
	\end{lemm}
	\begin{proof}
		First of all, the estimate \eqref{210021} yields directly
		\beqq
		\|\p_1^2 u\|_{L^2}
		\lesssim \|(\p_1 u, \p_{12} u)\|_{L^2}^3\|\phi \wx\|_{H^2}^4
		+\|(\p_1 u, \p_{12} u)\|_{L^2}^2\|\phi \wx\|_{H^2}^2
		+\|\phi\|_{H^2}\|\phi \p_t u\|_{L^2},
		\deqq
		which, together with the a priori assumptions
		\eqref{as1} and \eqref{as2}, yields directly
		\beq\label{4602}
		\begin{aligned}
			\|\p_1^2 u\|_{L^2}
			&\lesssim \|(\p_1 u, \p_{12} u)\|_{L^2}^3\|\phi \wx\|_{H^2}^4
			+\|(\p_1 u, \p_{12} u)\|_{L^2}^2\|\phi \wx\|_{H^2}^2
			+\|\phi\|_{H^2}\|\phi \p_t u\|_{L^2}\\
			&\le C_{M}\|(\p_1 u, \p_{12} u, \phi \p_t u)\|_{L^2}.
		\end{aligned}
		\deq
		Similarly, the estimates \eqref{210022},
		\eqref{21005-a}, \eqref{21006-a} and \eqref{21007-a} yield directly
		\beq\label{4603}
		\begin{aligned}
			\|\p_1^2 \p_2 u\|_{L^2}
			&\lesssim \|\phi \p_t u\|_{L^2}\|\p_2 \phi\|_{H^2}
			+\|\phi \p_t \p_2 u\|_{L^2}\|\phi\|_{H^2}
			+\|(\p_1 u, \p_{12}u, \p_{122}u)\|_{L^2}^2
			\|\phi \wx\|_{H^2}^2\\
			&\le C_M  \|(\p_1 u, \p_{12}u, \p_{122}u,
			\phi \p_t u, \phi \p_t \p_2 u)\|_{L^2},
		\end{aligned}
		\deq
		and
		\beqq
		\begin{aligned}
			\|\p_1^3 u\|_{L^2}
			&\lesssim  \|(\phi \p_t u, \p_t \p_1 u)\|_{L^2}(1+\|\phi\|_{H^2}^2)
			+\|(\p_1 u, \p_{12}u, \p_1^2 u, \p_{122}u)\|_{L^2}^2
			\|\phi \wx\|_{H^2}^2\\
			&\le C_M \|(\phi \p_t u, \p_t \p_1 u,
			\p_1 u, \p_{12}u, \p_1^2 u, \p_{122}u)\|_{L^2}.
		\end{aligned}
		\deqq
		Recall the estimates \eqref{210013a} and \eqref{210014}, we have
		\beqq
		\begin{aligned}
			\|\p_1^2 w \wx^3\|_{L^2}
			\lesssim  &~\|\p_1 \p_t u \|_{L^2}\| \phi \wx^2\|_{H^2}^2
			+\|\phi \p_t \p_2 u_1\|_{L^2}\|\phi \wx^3\|_{H^2}\\
			&+\|(\p_1 u, \p_{12} u, \p_{122}u)\|_{L^2}^{\frac32}
			\|\p_1^2 u\|_{L^2}^{\frac12}\|\phi \wx^3\|_{H^2}^2\\
			&+\|(\p_1 u, \p_{12} u, \p_{122}u)\|_{L^2}^2 \|\phi \wx^3\|_{H^2}^2\\
			&+\|\p_1^2 u\|_{L^2}\|(\p_1 u, \p_{12}u)\|_{L^2}
			\|\phi \wx^3\|_{H^2}^2\\
			\le       &~C_M \|(\p_1 u, \p_{12} u, \p_{122}u,
			\phi \p_t u, \p_1 \p_t u,\phi \p_t \p_2 u)\|_{L^2},
		\end{aligned}
		\deqq
		and
		\beqq
		\begin{aligned}
			\|\p_1^2 \p_2 w \wx^2\|_{L^2}
			\lesssim
			&~\|(\p_{1}\p_t u, \p_{12}\p_t u, \phi \p_t \p_2 u, \phi \p_t \p_2^2 u)\|_{L^2}
			(1+\|\phi \wx^2\|_{H^3}^2)\\
			&+\|(\p_1 u, \p_1^2 u, \p_{12}u, \p_1\p_2^2 u, \p_1^2 \p_2 u)\|_{L^2}^2
			\|\phi \wx^2\|_{H^3}^2\\
			&+\|\phi \p_2^3 u\|_{L^2}\|(\p_1 u, \p_{12}u)\|_{L^2}\|\phi \wx^2\|_{H^2}^2\\
			\le &~C_M  \|(\p_1 u, \p_{12} u, \p_{122}u,
			\p_{1}\p_t u, \p_{12}\p_t u, \phi \p_t \p_2 u,
			\phi \p_t \p_2^2 u)\|_{L^2},
		\end{aligned}
		\deqq
		where we have used the estimates  \eqref{4602}
		and \eqref{4603} in the last inequality.
		Finally, it is easy to check that
		\beqq
		\|\p_2^2 u_1\|_{L^2}
		\le \|\p_2(\p_2 u_1-\p_1 u_2)\|_{L^2}+\|\p_{21} u_2\|_{L^2}
		\le \|\p_1^2 \p_2 w \wx^2\|_{L^2}+\|\p_{21} u_2\|_{L^2}
		\le C_{M} \sqrt{\widehat{\mathcal{D}}(t)}.
		\deqq
		Therefore, we complete the proof of this lemma.
	\end{proof}
	
	Finally, let us establish the estimate for the density.
	\begin{lemm}\label{lemma48}
		Under the assumption \eqref{as1}-\eqref{as2},
		for any smooth solution of equation \eqref{new-ns}, it holds
		\beqq 
		\eb(t) \le \eb(0) e^{C_M \sqrt{\es(0)}}.
		\deqq
	\end{lemm}
	\begin{proof}
		From the analysis of Lemmas \ref{lemma29}, \ref{lemma211},
		\ref{lemma212} and \ref{lemma213}, we can apply estimate
		\eqref{4601} in Lemma \ref{lemma46} to obtain
		\beqq
		\frac{d}{dt}\eb(t) \le C_M \sqrt{\widehat{\mathcal{D}}(t)}\eb(t).
		\deqq
		Then, we apply the Gr\"{o}nwall inequality to obtain
		\beq\label{4802}
		\eb(t) \le \eb(0) e^{C_M \int_0^t \sqrt{\widehat{\mathcal{D}}(\tau)} d\tau}
		\le \eb(0) e^{C_M \int_0^t \sqrt{\widehat{\mathcal{D}}(\tau)} d\tau}.
		\deq
		The estimate \eqref{4701} and H\"{o}lder inequality yield directly
		\beq\label{4803}
		\begin{aligned}
			\int_0^t \sqrt{\widehat{\mathcal{D}}(\tau)} d\tau
			\le (\int_0^t \widehat{\mathcal{D}}(\tau)e^{\gamma \tau} d\tau)^{\frac12}
			(\int_0^t e^{-\gamma \tau} d\tau)^{\frac12}
			\le C_M \sqrt{\es(0)}.
		\end{aligned}
		\deq
		Then, the combination of estimates \eqref{4802} and \eqref{4803} gives directly
		\beqq
		\eb(t) \le \eb(0) e^{C_M \sqrt{\es(0)}}.
		\deqq
		Therefore, we complete the proof of this lemma.
	\end{proof}

	Under the assumption of a priori estimate
	\beqq
	\eb (t)\le 4 \eb (0), \quad \es(t) \le 4 \sqrt{\es(0)},
	\deqq
	we can establish the estimates
	\beq\label{4901}
	\es(t) e^{\kappa t}+\int_0^t \widehat{\mathcal{D}}(\tau )e^{\kappa \tau} d\tau
	\le C_{\eb (0)} \es(0),
	\deq
	and
	\beq\label{4902}
	\eb(t) \le \eb(0) e^{C^*_{\eb (0)} \sqrt{\es(0)}}.
	\deq
	Here $C^*_{\eb (0)}$ depends on $\eb (0)$ and is different
	from $C_{\eb (0)}$ in \eqref{4901}.
	If we require the initial data $\es(0)$ satisfies
	\beqq
	\es(0)\le \min\left\{1, \frac{4}{(1+C_{\eb (0)})^2},
	\frac{(\ln 2)^2}{(1+C^*_{\eb (0)})^2}\right\},
	\deqq
	the estimates \eqref{4901} and \eqref{4902} give directly
	\beq\label{c2}
	\eb (t)\le 2 \eb (0), \quad
	\es(t) e^{\kappa t}
	+\int_0^t \widehat{\mathcal{D}}(\tau )e^{\kappa \tau} d\tau
	\le 2 \sqrt{\es(0)}.
	\deq
	Therefore, we have established the closed estimate in this subsection.
	
	\subsection{Global well-posedness and exponential stability}
	\label{continuity-arguement}
	
	In this subsection, we will establish the global-in-time well-posedness
	and exponential stability for the inhomogeneous anisotropic incompressible
	Navier-Stokes equation \eqref{new-ns} under the condition of small
	initial data $\es(0)$. The closed energy estimate, established in previous
	subsection, will play an important role in this process.
	
	\begin{proof}[\textbf{Proof of Theorem \ref{global-well}}]
		Suppose the assumptions in Theorem \ref{local-well} hold,
		it is easy to obtain the local-in-time well-posedness of Eq.\eqref{new-ns}.
		Next, we use the standard continuity argument to show the global well-posedness.
		From the local existence result in Theorem \ref{local-well}
		and smallness assumption of initial condition in Theorem \ref{global-well},
		it holds 
		\beqq 
		\eb (t)\le 4 \eb (0)~~\text{and} ~~\es(t) \le 4 \sqrt{\es(0)},
		\quad \forall t\in [0, T_0).
		\deqq
		Set
		\beq\label{criterion}
		T^*\overset{\text{def}}{=}\underset{T_0}{\sup}
		\left\{T_0~|\eb (t)\le 4 \eb (0)~~\text{and} ~~\es(t) \le 4 \sqrt{\es(0)},
		\quad \forall t\in [0, T_0)\right\},
		\deq
		we claim that $T^*=+\infty$. Otherwise, applying the estimate \eqref{c2}
		and the local-in-time existence result in Theorem \ref{local-well},
		there exists a positive constant $T^{**}$ such that $T^{**}>T^{*}$,
		it holds that for any $T\in (T^{*}, T^{**})$,
		\beqq
		\eb (t)\le 4 \eb (0)~~\text{and} ~~\es(t) \le 4 \sqrt{\es(0)},
		\quad \forall t\in [0, T).
		\deqq
		This contradicts the definition of $T^*$ in \eqref{criterion}.
		Thus, we complete the proof of the global-in-time well-posedness
		in Theorem \ref{global-well}.
		Finally, due to the uniform estimate \eqref{c2} and
		Eq.\eqref{new-ns}, we can obtain
		\beq\label{c3}
		\|\nabla p \|_{H^1}^2 e^{\gamma t}\le C_{\eb(0)}.
		\deq
		Then, the combination of estimates \eqref{c2} and \eqref{c3}
		implies the uniform estimate \eqref{uniform-estimate}
		and exponential decay estimate \eqref{decay} in Theorem \ref{global-well}.
		Therefore, we complete the proof of Theorem \ref{global-well}.
	\end{proof}
	
	\begin{appendices}

		\section{Proof of claimed estimates}\label{claim-estimate}
		
		In this appendix, we will give the proof of claimed estimates
		\eqref{2317}, \eqref{2321}, \eqref{2330}, \eqref{2710}, \eqref{2810},
		\eqref{210018}, \eqref{210019}
		in Section \ref{section-well-poseness} and the estimate \eqref{4402} in Section \ref{section-global-well-poseness}.
		
		\begin{proof}[\textbf{Proof of estimate \eqref{2317}}]
			Using the density equation $\eqref{new-ns}_1$, we have
			\beq\label{2312}
			\begin{aligned}
				I_{3}
				=&\int \phi^2 \p_t \p_2^2 u \cdot \p_2^2 u dx
				+2 \int \p_2(\phi^2)\p_2 \p_t u \cdot \p_2^2 u dx
				+\int \p_2^2(\phi^2)\p_t u \cdot \p_2^2 u dx\\
				=&\frac{d}{dt}\frac12 \int \phi^2 |\p_2^2 u|^2 dx
				+\int |\p_2^2 u|^2 \phi u \cdot \nabla \phi dx\\
				&+2 \int \p_2(\phi^2)\p_2 \p_t u \cdot \p_2^2 u dx
				+\int \p_2^2(\phi^2)\p_t u \cdot \p_2^2 u dx\\
				\overset{\text{def}}{=} &\frac{d}{dt}\frac12 \int \phi^2 |\p_2^2 u|^2 dx
				+I_{3,1}+I_{3,2}+I_{3,3}.
			\end{aligned}
			\deq
			Using Sobolev and Hardy inequalities, we have
			\beq\label{2313}
			\begin{aligned}
				|I_{3,1}|
				&\lesssim \|\phi \p_2^2 u\|_{L^2}
				\|\p_2^2 u \wx^{-1}\|_{L^2}^{\frac12}
				\|\p_1(\p_2^2 u \wx^{-1})\|_{L^2}^{\frac12}
				\|\nabla \phi \wx^2\|_{L^2}^{\frac12}
				\|\p_2\nabla \phi \wx^2\|_{L^2}^{\frac12}
				\|u \wx^{-1}\|_{L^\infty}\\
				&\lesssim \|\phi \p_2^2 u\|_{L^2}
				\|\p_1 \p_2^2 u \|_{L^2}
				\|(\nabla \phi, \p_2\nabla \phi)\wx^2\|_{L^2}
				\|(\p_1 u, \p_{12}u)\|_{L^2}\\
				&\le \var \|\p_1 \p_2^2 u \|_{L^2}^2
				+C_\var \|\phi \p_2^2 u\|_{L^2}^2
				\|(\nabla \phi, \p_2\nabla \phi)\wx^2\|_{L^2}^2
				\|(\p_1 u, \p_{12}u)\|_{L^2}^2,
			\end{aligned}
			\deq
			and
			\beq\label{2314}
			\begin{aligned}
				|I_{3,2}|
				&\lesssim \|\phi \p_2 \p_t u\|_{L^2}
				\|\p_2^2 u \wx^{-1}\|_{L^2}^{\frac12}
				\|\p_1(\p_2^2 u \wx^{-1})\|_{L^2}^{\frac12}
				\|\p_2 \phi \wx\|_{L^2}^{\frac12}
				\|\p_2^2 \phi \wx\|_{L^2}^{\frac12}\\
				&\lesssim \|\phi \p_2 \p_t u\|_{L^2}
				\|\p_1 \p_2^2 u \|_{L^2}
				\|(\p_2 \phi, \p_2^2 \phi)\wx\|_{L^2}\\
				&\le \var \|\p_1 \p_2^2 u \|_{L^2}^2
				+C_\var\|\phi \p_2 \p_t u\|_{L^2}^2
				\|(\p_2 \phi, \p_2^2 \phi)\wx\|_{L^2}^2.
			\end{aligned}
			\deq
			Obviously, it is easy to check that
			\beqq
			\begin{aligned}
				I_{3,3}
				=2\int (\p_2 \phi)^2 \p_t u \cdot \p_2^2 u dx
				+2\int \phi \p_2^2 \phi \p_t u \cdot \p_2^2 u dx
				\overset{\text{def}}{=} I_{3,3,1}+I_{3,3,2}.
			\end{aligned}
			\deqq
			Using Sobolev and Hardy inequalities, we have
			\beq\label{2315}
			\begin{aligned}
				|I_{3,3,1}|
				\lesssim &~\|\p_2^2 u \p_2 \phi\|_{L^2}
				\|\p_t u \p_2 \phi\|_{L^2} \\
				\lesssim &~\|\p_2^2 u \wx^{-1}\|_{L^2}^{\frac12}
				\|\p_1(\p_2^2 u \wx^{-1})\|_{L^2}^{\frac12}
				\|\p_2 \phi \wx\|_{L^2}^{\frac12}
				\|\p_2^2 \phi \wx\|_{L^2}^{\frac12}\\
				&~\times\|\p_t u \wx^{-1}\|_{L^2}^{\frac12}
				\|\p_1(\p_t u \wx^{-1})\|_{L^2}^{\frac12}
				\|\p_2 \phi \wx\|_{L^2}^{\frac12}
				\|\p_2^2 \phi \wx\|_{L^2}^{\frac12}\\
				\le &~\var\|\p_1 \p_2^2 u \|_{L^2}^2
				+C_\var\|\p_1 \p_t u\|_{L^2}^2
				\|(\p_2 \phi,\p_2^2 \phi)\wx\|_{L^2}^4,
			\end{aligned}
			\deq
			and
			\beq\label{2316}
			\begin{aligned}
				|I_{3,3,2}|
				&\lesssim \|\phi\p_2^2 u\|_{L^2}
				\|\p_t u \wx^{-1}\|_{L^2}^{\frac12}
				\|\p_1(\p_t u \wx^{-1})\|_{L^2}^{\frac12}
				\|\p_2^2 \phi \wx\|_{L^2}^{\frac12}
				\|\p_2^3 \phi \wx\|_{L^2}^{\frac12}\\
				&\lesssim \|\phi\p_2^2 u\|_{L^2}
				\|\p_1 \p_t u\|_{L^2}
				\|(\p_2^2 \phi, \p_2^3 \phi)\wx\|_{L^2}.
			\end{aligned}
			\deq
			Substituting the estimates of \eqref{2313}, \eqref{2314},
			\eqref{2315} and \eqref{2316} into \eqref{2312}, we can get that
			\beqq
			I_{3}
			\ge \frac{d}{dt}\frac12 \int \phi^2 |\p_2^2 u|^2 dx
			-\var \|\p_1 \p_2^2 u \|_{L^2}^2-C_\var (1+\e^3),
			\deqq
			Therefore, we complete the proof of estimate \eqref{2317}.
		\end{proof}
		
		\begin{proof}[\textbf{Proof of estimate \eqref{2321}}]
			Indeed, it is easy to check that
			\beqq
			\begin{aligned}
				I_4
				=&\int \phi^2 \p_2^2 u \cdot \nabla u\cdot \p_2^2 udx
				+\int \phi^2 (u \cdot \nabla \p_2^2 u) \cdot \p_2^2 udx\\
				&+2\int \phi^2 (\p_2 u \cdot \nabla \p_2 u)\cdot \p_2^2 udx
				+2\int \p_2 (\phi^2)(u\cdot \nabla \p_2 u)\cdot \p_2^2 udx\\
				&+2\int \p_2 (\phi^2)(\p_2 u \cdot \nabla u)\cdot \p_2^2 udx
				+\int \p_2^2(\phi^2)(u\cdot \nabla u)\cdot \p_2^2 u dx \\
				\overset{\text{def}}{=} &~I_{4,1}+I_{4,2}+I_{4,3}+I_{4,4}+I_{4,5}+I_{4,6}.
			\end{aligned}
			\deqq
			We will give the estimates for the terms from $I_{4,1}$
			to $I_{4,6}$.
			Using the Sobolev inequality, we get that
			\beqq
			\begin{aligned}
				|I_{4,1}|
				\lesssim&~\|\phi \wx^2\|_{L^\infty}
				\|\phi \p_2^2 u\|_{L^2}
				\|\p_2^2 u \wx^{-1}\|_{L^2}^{\frac12}
				\|\p_1(\p_2^2 u \wx^{-1})\|_{L^2}^{\frac12}\\
				&~\times\|\nabla u \wx^{-1}\|_{L^2}^{\frac12}
				\|\p_2(\nabla u \wx^{-1})\|_{L^2}^{\frac12}\\
				\le &~\var \|\p_1 \p_2^2 u\|_{L^2}^2
				+C_\var (1+\|\phi \p_2^2 u\|_{L^2}^4\|\phi \wx^2\|_{H^2}^4)
				\|(\p_1 u, \p_{12}u)\|_{L^2}^2
				.
			\end{aligned}
			\deqq
			Integrating by part and using divergence-free condition, we have
			\beqq
			\begin{aligned}
				|I_{4,2}|
				=&\left|\int |\p_2^2 u|^2 \phi u \cdot \nabla \phi dx \right|\\
				\lesssim&~\|u \wx^{-1}\|_{L^\infty}
				\|\phi \p_2^2 u\|_{L^2}
				\|\p_2^2 u \wx^{-1}\|_{L^2}^{\frac12}
				\|\p_1(\p_2^2 u \wx^{-1})\|_{L^2}^{\frac12}\\
				&~\times\|\nabla \phi \wx^2\|_{L^2}^{\frac12}
				\|\p_2 \nabla \phi \wx^2\|_{L^2}^{\frac12}\\
				\lesssim&~\|\phi \p_2^2 u\|_{L^2}
				\|\p_1 \p_2^2 u\|_{L^2}
				\|(\p_1 u, \p_{12}u)\|_{L^2}
				\|(\nabla \phi, \p_2 \nabla \phi)\wx^2\|_{L^2}\\
				\le &~\var\|\p_1 \p_2^2 u\|_{L^2}^2
				+C_\var \|\phi \p_2^2 u\|_{L^2}^2
				\|(\p_1 u, \p_{12}u)\|_{L^2}^2
				\| \phi \wx^2\|_{H^2}^2.
			\end{aligned}
			\deqq
			Using the Sobolev inequality and divergence-free condition, we can get
			\beqq
			\begin{aligned}
				|I_{4,3}|
				\lesssim&~\|\phi \wx\|_{L^\infty}\|\phi \p_2^2 u\|_{L^2}
				\|\p_2 u_1 \wx^{-1}\|_{L^2}^{\frac12}
				\|\p_1(\p_2 u_1 \wx^{-1})\|_{L^2}^{\frac12}
				\|\p_{12} u\|_{L^2}^{\frac12}\|\p_{212} u\|_{L^2}^{\frac12}\\
				&~+\|\phi \wx\|_{L^\infty}\|\phi \p_2^2 u\|_{L^2}
				\|\p_1 u_1 \|_{L^2}^{\frac12}
				\|\p_{21} u_1\|_{L^2}^{\frac12}
				\|\p_{22} u \ \wx^{-1}|_{L^2}^{\frac12}
				\|\p_{1} (\p_{22} u \ \wx^{-1})\|_{L^2}^{\frac12}\\
				\lesssim&~ \|\phi \wx\|_{H^2}\|\phi \p_2^2 u\|_{L^2}
				(\|\p_{12}u\|_{L^2}^{\frac32}\|\p_{122}u\|_{L^2}^{\frac12}
				+\|(\p_1 u_1, \p_{12}u_1)\|_{L^2}\|\p_{122}u\|_{L^2})\\
				\le &~\var\|\p_1 \p_2^2 u\|_{L^2}^2
				+C_\var \|(\p_1 u, \p_{12}u)\|_{L^2}^2
				(1+\|\phi \wx\|_{H^2}^2\|\phi \p_2^2 u\|_{L^2}^2).
			\end{aligned}
			\deqq
			Similarly, it is easy to check that
			\beqq
			\begin{aligned}
				|I_{4,4}|
				\lesssim&~\|u_1\wx^{-1}\|_{L^\infty}
				\|\phi \p_2^2 u\|_{L^2}
				\|\p_{12}u\|_{L^2}^{\frac12}
				\|\p_{122}u\|_{L^2}^{\frac12}
				\|\p_{2}\phi \wx\|_{L^2}^{\frac12}
				\|\p_{1}(\p_{2}\phi \wx)\|_{L^2}^{\frac12}\\
				&~+\|u_2\wx^{-1}\|_{L^\infty}
				\|\phi \p_2^2 u\|_{L^2}
				\|\p_{22}u \wx^{-1}\|_{L^2}^{\frac12}
				\|\p_{1}(\p_{22}u \wx^{-1})\|_{L^2}^{\frac12}
				\|\p_{2}\phi \wx^2\|_{L^2}^{\frac12}
				\|\p_{2}^2\phi \wx^2\|_{L^2}^{\frac12}\\
				\le &~\var\|\p_1 \p_2^2 u\|_{L^2}^2
				+C_\var (1+\|\phi \p_2^2 u\|_{L^2}^2
				\|\phi \wx^2\|_{H^2}^2 )
				\|(\p_1 u, \p_{12} u)\|_{L^2}^2;\\
				|I_{4,5}|
				\lesssim&~\|\p_2 \phi \wx^2\|_{L^\infty}
				\|\phi \p_2^2 u\|_{L^2}
				\|\p_2 u \wx^{-1}\|_{L^2}^{\frac12}
				\|\p_1(\p_2 u \wx^{-1})\|_{L^2}^{\frac12}
				\|\nabla u \wx^{-1}\|_{L^2}^{\frac12}
				\|\p_2(\nabla u \wx^{-1})\|_{L^2}^{\frac12}\\
				\lesssim&~ \|\p_2 \phi \wx^2\|_{L^\infty}
				\|\phi \p_2^2 u\|_{L^2}
				\|(\p_1 u, \p_{12}u)\|_{L^2}^{\frac32}
				(\|\p_{12}u \|_{L^2}^{\frac12}+\|\p_{122}u\|_{L^2}^{\frac12})\\
				\le &~\var\|\p_1 \p_2^2 u \|_{L^2}^2
				+C_\var (1+\|\p_2 \phi \wx^2\|_{H^2}^{\frac43}
				\|\phi \p_2^2 u\|_{L^2}^{\frac43}
				\|(\p_1 u, \p_{12}u)\|_{L^2}^2).
			\end{aligned}
			\deqq
			Finally, let us deal with the term $I_{4,6}$.
			\beqq
			\begin{aligned}
				I_{4,6}
				&=2\int \phi \p_2^2 \phi (u \cdot \nabla u)\cdot \p_2^2 u dx
				+2\int (\p_2 \phi)^2 u_1 \p_1 u \cdot \p_2^2 u dx
				+2\int (\p_2 \phi)^2 u_2 \p_2 u \cdot \p_2^2 u dx\\
				& \overset{\text{def}}{=} I_{4,6,1}+I_{4,6,2}+I_{4,6,3}.
			\end{aligned}
			\deqq
			It is easy to check that
			\beq\label{2318}
			\begin{aligned}
				|I_{4,6,1}|
				&\lesssim \|u \wx^{-1}\|_{L^\infty}
				\|\phi \p_2^2 u\|_{L^2}
				\|\nabla u \wx^{-1}\|_{L^2}^{\frac12}
				\|\p_2 \nabla u \wx^{-1}\|_{L^2}^{\frac12}
				\|\p_2^2 \phi \wx^2\|_{L^2}^{\frac12}
				\|\p_1(\p_2^2 \phi \wx^2)\|_{L^2}^{\frac12}\\
				&\lesssim \|\phi \p_2^2 u\|_{L^2}
				\|(\p_1 u, \p_{12}u)\|_{L^2}^{\frac{3}{2}}
				\|\phi \wx^2\|_{H^3}
				(\|\p_{12}u\|_{L^2}^{\frac12}+\|\p_1 \p_2^2 u\|_{L^2}^{\frac12})\\
				&\le \var\|\p_1 \p_2^2 u\|_{L^2}^2
				+C_\var(1+\|\phi \p_2^2 u\|_{L^2}^{\frac43}\|\phi \wx^2\|_{H^3}^{\frac43})
				\|(\p_1 u, \p_{12}u)\|_{L^2}^2,
			\end{aligned}
			\deq
			and
			\beq\label{2319}
			\begin{aligned}
				|I_{4,6,2}|
				&\lesssim\|u \wx^{-1}\|_{L^\infty}
				\|\p_2^2 u \wx^{-1}\|_{L^2}
				\|\p_1 u \|_{L^2}^{\frac12}
				\|\p_{21}  u \|_{L^2}^{\frac12}
				\|\p_2^2 \phi \wx^2\|_{L^2}^{\frac12}
				\|\p_1(\p_2^2 \phi \wx^2)\|_{L^2}^{\frac12}\\
				&\lesssim\|(\p_1 u, \p_{12}u)\|_{L^2}^{2}
				\|\p_1\p_2^2 u \|_{L^2}
				\|\phi \wx^2\|_{H^3}\\
				&\le \var \|\p_1\p_2^2 u \|_{L^2}^2
				+C_\var \|(\p_1 u, \p_{12}u)\|_{L^2}^4\|\phi \wx^2\|_{H^3}^2.
			\end{aligned}
			\deq
			Integrating by part and using divergence-free condition, we have
			\beq\label{2320}
			\begin{aligned}
				|I_{4,6,3}|
				=&~\left|\int |\p_2 u|^2 (\p_2 \phi)^2 \p_2 u_2 dx
				+2\int |\p_2 u|^2 u_2 \p_2 \phi\p_2^2 \phi   dx \right|\\
				\lesssim&~\|\p_2 u \wx^{-1}\|_{L^2}
				\|\p_2 u \wx^{-1}\|_{L^2}^{\frac12}
				\|\p_1(\p_2 u \wx^{-1})\|_{L^2}^{\frac12}
				\|\p_1 u_1\|_{L^2}^{\frac12}
				\|\p_{21} u_1\|_{L^2}^{\frac12}
				\|\p_2 \phi \wx\|_{L^\infty}^2\\
				&~+\|\p_2 u \wx^{-1}\|_{L^2}
				\|\p_2 u \wx^{-1}\|_{L^2}^{\frac12}
				\|\p_1(\p_2 u \wx^{-1})\|_{L^2}^{\frac12}
				\|\p_{2}^2 \phi \wx^2\|_{L^2}^{\frac12}\\
				&~\times\|\p_{2}^3 \phi \wx^2\|_{L^2}^{\frac12}
				\|u_2\wx^{-1}\|_{L^\infty}
				\|\p_2 \phi \wx\|_{L^\infty}\\
				\lesssim&~\|(\p_1 u, \p_{12}u)\|_{L^2}^2\|\phi \wx^2\|_{H^3}^2.
			\end{aligned}
			\deq
			Thus, the combination of estimates \eqref{2318}, \eqref{2319}
			and \eqref{2320} yields directly
			\beqq
			|I_{4,6}| \le \var \|\p_1\p_2^2 u \|_{L^2}^2+C_\var (1+\e^5).
			\deqq
			Thus, the combination of estimates from term $I_{4,1}$ to $I_{4,6}$
			yields directly
			\beqq
			|I_{4}| \le \var \|\p_1\p_2^2 u \|_{L^2}^2+C_\var (1+\e^5).
			\deqq
			Therefore, we complete the proof of estimate \eqref{2321}.
		\end{proof}
		
		\begin{proof}[\textbf{Proof of estimate \eqref{2330}}]
			It is easy to check that
			\beqq
			\begin{aligned}
				I_6=&\int \phi^2 (u \cdot \nabla \p_2^3 u) \cdot \p_2^3 u dx
				+\int \phi^2 (\p_2^3 u \cdot \nabla u)\cdot \p_2^3 u dx
				+3\int \phi^2 (\p_2^2 u\cdot \nabla \p_2 u)\cdot \p_2^3 u dx\\
				&+3\int \phi^2(\p_2 u\cdot \nabla \p_2^2 u)\cdot \p_2^3 u dx
				+3\int \p_2(\phi^2)(\p_2^2 u \cdot \nabla u)\cdot \p_2^3 u dx
				+6\int \p_2(\phi^2)(\p_2 u \cdot \nabla \p_2 u )\cdot \p_2^3 u dx\\
				&+3\int  \p_2 (\phi^2)(u\cdot \nabla \p_2^2 u)\cdot \p_2^3 u dx
				+3\int \p_2^2(\phi^2)(\p_2 u \cdot \nabla u)\cdot \p_2^3 u dx
				+3\int \p_2^2(\phi^2)(u \cdot \nabla \p_2 u)\cdot \p_2^3 u dx\\
				&+\int \p_2^3 (\phi^2)(u\cdot \nabla u)\cdot \p_2^3 u dx
				\overset{\text{def}}{=} \sum_{i=1}^{10}I_{6,i}.
			\end{aligned}
			\deqq
			Then, we can apply the Sobolev and  Hardy inequalities to give the estimate
			from term $I_{6,1}$ to $I_{6,10}$ as follows:
			\beqq
			\begin{aligned}
				|I_{6,1}|
				\lesssim&~\|\p_2^3 u \wx^{-1}\|_{L^2}
				\|\phi \p_2^3 u\|_{L^2}^{\frac12}
				\|\p_1(\phi \p_2^3 u)\|_{L^2}^{\frac12}
				\|\nabla \phi \wx^2\|_{L^2}^{\frac12}
				\|\p_2 \nabla \phi \wx^2\|_{L^2}^{\frac12}
				\|u \wx^{-1}\|_{L^\infty}\\
				\lesssim &~\|\p_1 \p_2^3 u \|_{L^2}
				\|\phi \p_2^3 u\|_{L^2}^{\frac12}
				\|\p_1 \p_2^3 u\|_{L^2}^{\frac12}
				\|\phi \wx\|_{H^2}^{\frac32}
				\|(\p_1 u, \p_{12}u)\|_{L^2}\\
				\le &~\var\|\p_1 \p_2^3 u \|_{L^2}^2
				+C_\var \|\phi \p_2^3 u\|_{L^2}^2
				\|\phi \wx\|_{H^2}^6
				\|(\p_1 u, \p_{12}u)\|_{L^2}^4;\\
			\end{aligned}
			\deqq
			and
			\begin{align*}
				|I_{6,2}|
				\lesssim&~\|\phi \wx^2\|_{L^\infty}\|\phi \p_2^3 u\|_{L^2}
				\|\p_2^3 u \wx^{-1}\|_{L^2}^{\frac12}
				\|\p_1(\p_2^3 u \wx^{-1})\|_{L^2}^{\frac12}
				\|\nabla u\wx^{-1}\|_{L^2}^{\frac12}
				\|\p_2(\nabla u\wx^{-1})\|_{L^2}^{\frac12}\\
				\lesssim&~\|\phi \p_2^3 u\|_{L^2}
				\|\p_1 \p_2^3 u \|_{L^2}
				\|(\p_1 u, \p_{12}u)\|_{L^2}^{\frac12}
				(\|\p_{12} u \|_{L^2}^{\frac12}+\|\p_{122} u \|_{L^2}^{\frac12})
				\|\phi \wx^2\|_{H^2}\\
				\le &~\var\|\p_1 \p_2^3 u \|_{L^2}^2
				+C_\var (1+ \|\phi \p_2^3 u\|_{L^2}^4
				\|(\p_1 u, \p_{12}u)\|_{L^2}^2
				\|\phi \wx^2\|_{H^2}^4);\\
				|I_{6,3}|
				\lesssim&~ \|\phi \wx^2\|_{L^\infty}
				\|\phi \p_2^3 u\|_{L^2}
				\|\p_2^2 u \wx^{-1}\|_{L^2}^{\frac12}
				\|\p_1(\p_2^2 u \wx^{-1})\|_{L^2}^{\frac12}
				\|\nabla \p_2 u \wx^{-1}\|_{L^2}^{\frac12}
				\|\p_2(\nabla \p_2 u \wx^{-1})\|_{L^2}^{\frac12}\\
				\lesssim&~\|\phi \p_2^3 u\|_{L^2}
				\|\p_1\p_2^2 u\|_{L^2}
				\|(\p_{12}u, \p_{122}u)\|_{L^2}^{\frac12}
				\|(\p_{122}u, \p_{1}\p_2^3 u)\|_{L^2}^{\frac12}
				\|\phi \wx^2\|_{H^2} \\
				\le &~\var\|\p_1\p_2^3 u \|_{L^2}^2
				+C_\var(1+\|\phi \p_2^3 u\|_{L^2}^2\|\phi \wx^2\|_{H^2}^2)
				\|(\p_{12}u, \p_{122}u)\|_{L^2}^2;\\
				|I_{6,4}|
				\lesssim&~  \|\phi \wx^2\|_{L^\infty}
				\|\phi \p_2^3 u\|_{L^2}
				\|\p_2 u_1 \wx^{-1}\|_{L^2}^\frac{1}{2}
				\|\p_1(\p_2 u \wx^{-1})\|_{L^2}^\frac{1}{2} \| \p_1 \p_{2}^2 u\wx^{-1}\|_{L^2}^\frac{1}{2}
				\|\p_2( \p_1 \p_{2}^2 u \wx^{-1})\|_{L^2}^\frac{1}{2}
				\\
				&~+ \|\phi \wx^2\|_{L^\infty}
				\|\phi \p_2^3 u\|_{L^2}
				\|\p_1 u_1 \wx^{-1}\|_{L^2}^\frac{1}{2}
				\|\p_2(\p_1 u \wx^{-1})\|_{L^2}^\frac{1}{2}  \| \p_{2}^3 u\wx^{-1}\|_{L^2}^\frac{1}{2}
				\|\p_1( \p_{2}^3 u \wx^{-1})\|_{L^2}^\frac{1}{2}
				\\
				\le &~\var \|\p_1 \p_2^3 u\|_{L^2}^2
				+C_\var (1+\|\phi \p_2^3 u\|_{L^2}^2\|\phi \wx^2\|_{H^2}^2)
				\|(\p_1 u, \p_1 \p_2  u, \p_1 \p_2^2 u)\|_{L^2}^2;\\
				|I_{6,5}|
				\lesssim &~\|\p_2 \phi \wx^2\|_{L^\infty}
				\|\phi \p_2^3 u\|_{L^2}
				\|\p_2^2 u \wx^{-1}\|_{L^2}^{\frac12}
				\|\p_1(\p_2^2 u \wx^{-1})\|_{L^2}^{\frac12}
				\|\nabla u \wx^{-1}\|_{L^2}^{\frac12}
				\|\p_2 \nabla u \wx^{-1}\|_{L^2}^{\frac12}\\
				\lesssim &~\|\phi \p_2^3 u\|_{L^2}
				\|\p_1 \p_2^2 u\|_{L^2}
				\|(\p_1 u, \p_{12}u, \p_1\p_2^2 u)\|_{L^2}
				\| \phi \wx^2\|_{H^3};\\
				|I_{6,6}|
				\lesssim&~\|\p_2 \phi \wx^2\|_{L^\infty}
				\|\phi \p_2^3 u\|_{L^2}
				\|\nabla \p_2 u \wx^{-1}\|_{L^2}^{\frac12}
				\|\p_2 \nabla \p_2 u \wx^{-1}\|_{L^2}^{\frac12}
				\|\p_2 u \wx^{-1}\|_{L^2}^{\frac12}
				\|\p_{1}(\p_{2} u \wx^{-1})\|_{L^2}^{\frac12}\\
				\lesssim&~\| \phi \wx^2\|_{H^3}\|\phi \p_2^3 u\|_{L^2}
				\|(\p_{12}u, \p_{122}u)\|_{L^2}^{\frac32}
				(\|\p_{122}u\|_{L^2}^{\frac12}+\|\p_{1222}u\|_{L^2}^{\frac12});\\
				\le &~\var\|\p_1 \p_2^3 u\|_{L^2}^2
				+C_\var(1+\| \phi \wx^2\|_{H^3}^{\frac43}\|\phi \p_2^3 u\|_{L^2}^{\frac43})
				\|(\p_{12}u, \p_{122}u)\|_{L^2}^2;\\
				|I_{6,7}|
				\lesssim &~\|u_1\wx^{-1}\|_{L^\infty}
				\|\phi \p_2^3 u\|_{L^2}
				\|\p_1 \p_2^2 u\|_{L^2}^{\frac12}
				\|\p_2 \p_1 \p_2^2 u\|_{L^2}^{\frac12}
				\|\p_2 \phi \wx\|_{L^2}^{\frac12}
				\|\p_1(\p_2 \phi \wx)\|_{L^2}^{\frac12}\\
				&~+\|u_2\wx^{-1}\|_{L^\infty}\|\phi \p_2^3 u\|_{L^2}
				\|\p_2^3 u \wx^{-1}\|_{L^2}^{\frac12}
				\|\p_1 (\p_2^3 u \wx^{-1})\|_{L^2}^{\frac12}\\
				&~\times \|\p_2 \phi \wx\|_{L^2}^{\frac12}
				\|\p_2(\p_2 \phi \wx)\|_{L^2}^{\frac12}\\
				\le &~\var\|\p_1 \p_2^3 u\|_{L^2}^2
				+C_\var(1+\|\phi \p_2^3 u\|_{L^2}^2
				\|(\p_2 \phi, \p_{12}\phi, \p_{22}\phi)\|_{L^2}^2)
				\|(\p_1 u, \p_{12}u,\p_{1} \p_{2}^2 u)\|_{L^2}^2; \\
				|I_{6,8}|
				=&~\left|6\int ((\p_2 \phi)^2+\phi \p_2^2 \phi)
				(\p_2 u \cdot \nabla u)\cdot\p_2^3 u dx \right|\\
				\lesssim &~\|\p_2 \phi \wx^{\frac32}\|_{L^\infty}^2
				\|\p_2^3 u \wx^{-1}\|_{L^2}
				\|\p_2 u \wx^{-1}\|_{L^2}^{\frac12}
				\|\p_1(\p_2 u \wx^{-1})\|_{L^2}^{\frac12}
				\|\nabla u\wx^{-1}\|_{L^2}^{\frac12}
				\|\p_2(\nabla u\wx^{-1})\|_{L^2}^{\frac12}\\
				&~+\|\p_2 u \wx^{-1}\|_{L^\infty}
				\|\phi \p_2^3 u\|_{L^2}
				\|\nabla u \wx^{-1}\|_{L^2}^{\frac12}
				\|\p_2 \nabla u \wx^{-1}\|_{L^2}^{\frac12}
				\|\p_2^2 \phi \wx^2\|_{L^2}^{\frac12}
				\|\p_1(\p_2^2 \phi \wx^2)\|_{L^2}^{\frac12}\\
				\lesssim &~\|\p_2 \phi \wx^2\|_{H^2}^2
				\|\p_1 \p_2^3 u\|_{L^2}
				\|(\p_1 u, \p_1 \p_2 u, \p_1 \p_2^2 u)\|_{L^2}^2\\
				&~+\|\phi \p_2^3 u\|_{L^2}
				\|(\p_1 u, \p_{12}u, \p_1 \p_2^2 u)\|_{L^2}^2
				\|\p_2^2 \phi \wx^2\|_{H^1}\\
				\le &~\var \|\p_1 \p_2^3 u\|_{L^2}^2
				+C_\var (1+\|\phi \p_2^3 u\|_{L^2}^2+\|\p_2 \phi \wx^2\|_{H^2}^4)
				(1+ \|(\p_1 u, \p_1 \p_2 u, \p_1 \p_2^2 u)\|_{L^2}^4);\\
				|I_{6,9}|
				=&~\left|6\int(\phi \p_2^2 \phi+(\p_2 \phi)^2)
				(u \cdot \nabla \p_2 u)\cdot \p_2^3 u dx \right|\\
				\lesssim &~ \|u\wx^{-1}\|_{L^\infty}
				\|\phi \p_2^3 u\|_{L^2}
				\|\nabla \p_2 u \wx^{-1}\|_{L^2}^{\frac12}
				\|\p_2 \nabla \p_2 u \wx^{-1}\|_{L^2}^{\frac12}
				\|\p_2^2 \phi \wx^2\|_{L^2}^{\frac12}
				\|\p_1(\p_2^2 \phi \wx^2)\|_{L^2}^{\frac12}\\
				&~+\|\p_2 \phi \wx\|_{L^\infty}
				\|u\wx^{-1}\|_{L^\infty}
				\|\p_2^3 u \wx^{-1}\|_{L^2}
				\|\nabla \p_2 u \wx^{-1}\|_{L^2}^{\frac12}
				\|\p_2 \nabla \p_2 u \wx^{-1}\|_{L^2}^{\frac12}\\
				&~\times\|\p_2 \phi \wx^2\|_{L^2}^{\frac12}
				\|\p_1(\p_2 \phi \wx^2)\|_{L^2}^{\frac12}\\
				\lesssim&~\|\phi \p_2^3 u\|_{L^2}
				\|(\p_1 u, \p_{12}u, \p_{122}u)\|_{L^2}^{\frac32}
				(\|\p_1 \p_2^2 u\|_{L^2}^{\frac{1}{2}}
				+\|\p_1 \p_2^3 u\|_{L^2}^{\frac{1}{2}})
				\|\phi \wx^2\|_{H^3}\\
				&~+\|\p_1 \p_2^3 u\|_{L^2}
				\|(\p_1 u, \p_{12}u, \p_{122}u)\|_{L^2}^{\frac32}
				(\|\p_1 \p_2^2 u\|_{L^2}^{\frac{1}{2}}
				+\|\p_1 \p_2^3 u\|_{L^2}^{\frac{1}{2}})
				\|\phi \wx^2\|_{H^3}^2\\
				\le &~\var \|\p_1 \p_2^3 u\|_{L^2}^2
				+C_\var (1+\|\phi \p_2^3 u\|_{L^2}^8+ \|\phi \wx^2\|_{H^3}^4)
				(1+\|(\p_1 u, \p_{12}u, \p_{122}u)\|_{L^2}^6);\\
				|I_{6,10}|
				\lesssim
				&~\|\p_2 \phi \wx\|_{L^\infty}\|\p_2^3 u \wx^{-1}\|_{L^2}
				\|\p_2^2 \phi \wx^2\|_{L^2}^{\frac12}
				\|\p_1(\p_2^2 \phi \wx^2)\|_{L^2}^{\frac12}\\
				&~\times\|\nabla u \wx^{-1}\|_{L^2}^{\frac12}
				\|\p_2 \nabla u \wx^{-1}\|_{L^2}^{\frac12}\|u\wx^{-1}\|_{L^\infty}\\
				&~+\|\p_2^3 \phi \wx^2\|_{L^2}
				\|\p_2^3 u \wx^{-1}\|_{L^2}^{\frac12}
				\|\p_1(\p_2^3 u \wx^{-1})\|_{L^2}^{\frac12}
				\|\nabla u \wx^{-1}\|_{L^2}^{\frac12}\\
				&~\times \|\p_2 \nabla u \wx^{-1}\|_{L^2}^{\frac12}
				\|u\wx^{-1}\|_{L^\infty}\|\phi \wx\|_{L^\infty}\\
				\le
				&~\var\|\p_1 \p_2^3 u\|_{L^2}^2
				+C_\var \|(\p_1 u, \p_{12}u, \p_{122} u)\|_{L^2}^4
				\|\phi \wx^2\|_{H^3}^4.
			\end{align*}
			Thus, the combination of estimates from terms
			$I_{6,1}$ to $I_{6,10}$ yields directly
			\beqq
			|I_{6}|
			\le \var \|\p_1 \p_2^3 u \|_{L^2}^2
			+C_\var (1+\e^7).
			\deqq
			Therefore, we complete the proof of estimate \eqref{2330}.
		\end{proof}
		
		\begin{proof}[\textbf{Proof of estimate \eqref{2710}}]
			Integrating by part, we have
			\beqq
			\begin{aligned}
				I_{11}
				=-\int \p_t (\phi^2) (u\cdot \nabla u) \cdot \p_2^2 \p_t u dx
				-\int \phi^2 (\p_t u\cdot \nabla u) \cdot \p_2^2 \p_t u dx
				-\int \phi^2 (u\cdot \nabla \p_t u) \cdot \p_2^2 \p_t u dx
				\overset{\text{def}}{=} \sum_{i=1}^3 I_{11,i}.
			\end{aligned}
			\deqq
			Using the density equation, we have
			\beq\label{2708}
			\begin{aligned}
				|I_{11,1}|
				&=\left|2 \int \phi (u\cdot \nabla \phi) (u\cdot \nabla u) \cdot \p_2^2 \p_t u dx\right|\\
				&\lesssim \|\phi \p_2^2 \p_t u\|_{L^2}
				\|\nabla u \wx^{-1}\|_{L^2}^{\frac12}
				\|\p_2 \nabla u \wx^{-1}\|_{L^2}^{\frac12}
				\|\nabla \phi \wx^3\|_{L^2}^{\frac12}
				\|\p_1 (\nabla \phi \wx^3)\|_{L^2}^{\frac12}
				\|u \wx^{-1}\|_{L^\infty}^2\\
				&\lesssim \|\phi \p_2^2 \p_t u\|_{L^2}
				\|(\p_1 u, \p_1\p_2 u, \p_1 \p_2^2 u)\|_{L^2}^3
				\|(\nabla \phi, \p_1 \nabla \phi)\wx^3\|_{L^2}\\
				&\le \sigma_1\|\phi \p_2^2 \p_t u\|_{L^2}^2
				+C_{\sigma_1}\|(\p_1 u, \p_1\p_2 u, \p_1 \p_2^2 u)\|_{L^2}^6
				\|(\nabla \phi, \p_1 \nabla \phi)\wx^3\|_{L^2}^2.
			\end{aligned}
			\deq
			Similarly, we have
			\beq\label{2709}
			\begin{aligned}
				|I_{11,2}|
				\lesssim &~\|\phi \p_2^2 \p_t u\|_{L^2}
				\|\p_t u \wx^{-1}\|_{L^2}^{\frac12}
				\|\p_1(\p_t u \wx^{-1})\|_{L^2}^{\frac12}\\
				&~\times
				\|\nabla u \wx^{-1}\|_{L^2}^{\frac12}
				\|\p_2(\nabla u \wx^{-1})\|_{L^2}^{\frac12}
				\|\phi \wx^2\|_{L^\infty}\\
				\lesssim &~\|\phi \p_2^2 \p_t u\|_{L^2}
				\|\p_1 \p_t u \|_{L^2}
				\|(\p_1 u, \p_{12}u, \p_{122} u)\|_{L^2}
				\|\phi \wx^2\|_{H^2}\\
				\le &~\sigma_1 \|\phi \p_2^2 \p_t u\|_{L^2}^2
				+C_{\sigma_1}\|\p_1 \p_t u \|_{L^2}^2
				\|(\p_1 u, \p_{12}u, \p_{122} u)\|_{L^2}^2
				\|\phi \wx^2\|_{H^2}^2.
			\end{aligned}
			\deq
			Integrating by part, we have
			\beqq
			\begin{aligned}
				I_{11,3}
				=&\int \phi^2 u_1 \p_1 \p_t u \cdot \p_2^2 \p_t u dx
				+\int \phi^2 u_2 \p_2\left(\frac12 |\p_2 \p_t u|^2\right) dx\\
				=&\int \phi^2 u_1 \p_1 \p_t u \cdot \p_2^2 \p_t u dx
				-\frac12 \int \phi^2 \p_2 u_2 |\p_2 \p_t u|^2 dx
				-\int \phi u_2 \p_2 \phi |\p_2 \p_t u|^2 dx\\
				\overset{\text{def}}{=} &~I_{11,3,1}+I_{11,3,2}+I_{11,3,3}.
			\end{aligned}
			\deqq
			It is easy to check that
			\beqq
			\begin{aligned}
				|I_{11,3,1}|
				\lesssim
				&~\|\phi \p_2^2 \p_t u\|_{L^2}\|\p_1 \p_t u\|_{L^2}
				\|u_1 \wx^{-1}\|_{L^\infty}\|\phi \wx\|_{L^\infty}\\
				\le
				&~\sigma_1 \|\phi \p_2^2 \p_t u\|_{L^2}^2
				+C_{\sigma_1} \|\p_1 \p_t u\|_{L^2}^2
				\|(\p_1 u, \p_{12}u)\|_{L^2}^2
				\|\phi \wx\|_{H^2}^2.
			\end{aligned}
			\deqq
			Using the divergence-free condition $\eqref{new-ns}_3$
			and anisotropic Sobolev inequality, we have
			\beqq
			\begin{aligned}
				|I_{11,3,2}|
				\lesssim
				&~\|\phi\wx\|_{L^\infty}\|\phi \p_2 \p_t u\|_{L^2}
				\|\p_1 u_1 \|_{L^2}^{\frac12}
				\|\p_2 \p_1 u_1 \|_{L^2}^{\frac12}\\
				&~\times
				\|\p_2 \p_t u \wx^{-1}\|_{L^2}^{\frac12}
				\|\p_1(\p_2 \p_t u \wx^{-1})\|_{L^2}^{\frac12}\\
				\lesssim
				&~\|\phi\wx\|_{H^2}\|\phi \p_2 \p_t u\|_{L^2}
				\|(\p_1 u_1, \p_2 \p_1 u_1)\|_{L^2}
				\|\p_1 \p_2 \p_t u \|_{L^2}\\
				\le
				&~\var\|\p_1 \p_2 \p_t u \|_{L^2}^2
				+C_{\var} \|\phi\wx\|_{H^2}^2\|\phi \p_2 \p_t u\|_{L^2}^2
				\|(\p_1 u_1, \p_2 \p_1 u_1)\|_{L^2}^2,
			\end{aligned}
			\deqq
			and
			\beqq
			\begin{aligned}
				|I_{11,3,3}|
				\lesssim
				&~\|\phi \p_2 \p_t u\|_{L^2}
				\|\p_2 \p_t u \wx^{-1}\|_{L^2}^{\frac12}
				\|\p_1(\p_2 \p_t u \wx^{-1})\|_{L^2}^{\frac12}
				\|\p_2  \phi \wx\|_{L^2}^{\frac12}\\
				&~\times
				\|\p_2^2 \phi \wx\|_{L^2}^{\frac12}
				\|u_2 \wx^{-1}\|_{L^\infty}
				\|\p_2 \phi\wx^2\|_{L^\infty}\\
				\le
				&~\var \|\p_1 \p_2 \p_t u \|_{L^2}^2
				+C_{\var} \|\phi \p_2 \p_t u\|_{L^2}^2
				\|(\p_1 u, \p_{12}u)\|_{L^2}^2\|\phi\wx^2\|_{H^3}^4.
			\end{aligned}
			\deqq
			Thus, the combination of estimates from terms
			$I_{11,3,1}$ to $I_{11,3,3}$ yields directly
			\beqq
			|I_{11,3}|\le \var\|\p_{12}\p_t u \|_{L^2}^2
			+\sigma_1\|\phi \p_2^2 \p_t u\|_{L^2}^2
			+C_{\var,  \sigma_1}(1+\e^4),
			\deqq
			which, together with estimates \eqref{2708}
			and \eqref{2709}, gives directly
			\beqq
			|I_{11}|\le \var\|\p_{12}\p_t u \|_{L^2}^2
			+\sigma_1\|\phi \p_2^2 \p_t u\|_{L^2}^2
			+C_{\var,  \sigma_1}(1+\e^4),
			\deqq
			Therefore, we complete the proof of estimate \eqref{2710}.
		\end{proof}
		
		\begin{proof}[\textbf{Proof of estimate \eqref{2810}}]
			It is easy to check that
			\beq\label{2809}
			\begin{aligned}
				I_{13}
				=&\int \p_2^2(\phi^2) (u\cdot \nabla u)\cdot \p_t  \p_2^2 u dx
				+2\int \p_2(\phi^2) \p_2(u\cdot \nabla u)\cdot \p_t  \p_2^2 u dx\\
				&+\int \phi^2 \p_2^2(u\cdot \nabla u)\cdot \p_t  \p_2^2 u dx\\
				=&~2 \int (\p_2 \phi)^2 (u\cdot \nabla u)\cdot \p_t  \p_2^2 u dx
				+2 \int \phi \p_2^2 \phi (u\cdot \nabla u)\cdot \p_t  \p_2^2 u dx\\
				&+4\int \phi \p_2 \phi \p_2(u\cdot \nabla u)\cdot \p_t  \p_2^2 u dx
				+\int \phi^2 \p_2^2(u\cdot \nabla u)\cdot \p_t  \p_2^2 u dx\\
				\overset{\text{def}}{=} &I_{13,1}+I_{13,2}+I_{13,3}+I_{13,4}.
			\end{aligned}
			\deq
			It is easy to check that
			\beq\label{2805}
			\begin{aligned}
				|I_{13,1}|
				\lesssim
				&~\|\phi \p_t \p_2^2 u\|_{L^2}
				\|\p_2 \ln \phi\|_{L^2}^{\frac12}
				\|\p_1\p_2 \ln \phi\|_{L^2}^{\frac12}
				\|\nabla u \wx^{-1}\|_{L^2}^{\frac12}\\
				&~\times \|\p_2(\nabla u \wx^{-1})\|_{L^2}^{\frac12}
				\|u \wx^{-1}\|_{L^\infty}
				\|\p_2 \phi \wx^2\|_{L^\infty}\\
				\le
				&~\var \|\phi \p_t \p_2^2 u\|_{L^2}^2
				+C_\var\|(\p_2 \ln \phi, \p_{12} \ln \phi)\|_{L^2}^2
				\|(\p_1 u, \p_{12}u, \p_{122} u)\|_{L^2}^4
				\|\p_2 \phi \wx^2\|_{H^2}^2.
			\end{aligned}
			\deq
			Using the H\"{o}lder and Hardy inequalities, we have
			\beq\label{2806}
			\begin{aligned}
				|I_{13,2}|
				\lesssim &~\|u\wx^{-1}\|_{L^\infty}\|\phi \p_t  \p_2^2 u\|_{L^2}
				\|\p_2^2 \phi \wx^2\|_{L^2}^{\frac12}
				\|\p_1(\p_2^2 \phi \wx^2)\|_{L^2}^{\frac12}\\
				&~\times\|\nabla u \wx^{-1}\|_{L^2}^{\frac12}
				\|\p_2 \nabla u \wx^{-1}\|_{L^2}^{\frac12}\\
				\le &~\var \|\phi \p_t  \p_2^2 u\|_{L^2}^2
				+C_\var \|\phi \wx^2\|_{H^3}^2
				\|(\p_1 u, \p_{12}u, \p_{122} u)\|_{L^2}^4.
			\end{aligned}
			\deq
			Similarly, it is easy to check that
			\beq\label{2807}
			\begin{aligned}
				|I_{13,3}|
				\lesssim &~\|\phi \p_t  \p_2^2 u\|_{L^2}
				\|\p_2 u \wx^{-1}\|_{L^2}^{\frac{1}{2}}
				\|\p_1(\p_2 u \wx^{-1})\|_{L^2}^{\frac{1}{2}}\\
				&~\times \|\nabla u \wx^{-1}\|_{L^2}^{\frac12}
				\|\p_2(\nabla u \wx^{-1})\|_{L^2}^{\frac12}
				\|\p_2 \phi \wx^2\|_{L^\infty}\\
				&~+\|\phi \p_t  \p_2^2 u\|_{L^2}
				\|\nabla \p_2 u \wx^{-1}\|_{L^2}
				\|u \wx^{-1}\|_{L^\infty}\|\p_2 \phi \wx^2\|_{L^\infty}\\
				\le  &~\var \|\phi \p_t  \p_2^2 u\|_{L^2}^2
				+C_\var \|\p_2 \phi \wx^2\|_{H^2}^2
				\|(\p_1 u, \p_{12}u, \p_1 \p_2^2 u)\|_{L^2}^4.
			\end{aligned}
			\deq
			Obviously, it holds
			\beqq
			\begin{aligned}
				I_{13,4}
				=
				&\int \phi^2 (\p_2^2 u\cdot \nabla u)\cdot \p_t  \p_2^2 u dx
				+2\int \phi^2 (\p_2 u\cdot \p_2 \nabla u)\cdot \p_t  \p_2^2 u dx\\
				&+\int \phi^2  (u\cdot \p_2^2 \nabla u)\cdot \p_t  \p_2^2 u dx\\
				\overset{\text{def}}{=}
				&I_{13,4,1}+I_{13,4,2}+I_{13,4,3}.
			\end{aligned}
			\deqq
			It is easy to check that
			\beqq
			\begin{aligned}
				|I_{13,4,1}|
				\lesssim &~\|\phi \wx^2\|_{L^\infty}
				\|\phi \p_t  \p_2^2 u\|_{L^2}
				\|\p_2^2 u\wx^{-1}\|_{L^2}^{\frac12}
				\|\p_1(\p_2^2 u\wx^{-1})\|_{L^2}^{\frac12}\\
				&~\times \|\nabla u\wx^{-1}\|_{L^2}^{\frac12}
				\|\p_2 \nabla u\wx^{-1}\|_{L^2}^{\frac12}\\
				\le &~\var\|\phi \p_t  \p_2^2 u\|_{L^2}^2
				+C_\var\|(\p_1 u, \p_{12}u, \p_{122}u)\|_{L^2}^4
				\|\phi \wx^2\|_{H^2}^2,\\
				|I_{13,4,2}|
				\lesssim &~\|\phi \wx^2\|_{L^\infty}
				\|\phi \p_t  \p_2^2 u\|_{L^2}
				\|\p_2 u_1\wx^{-1}\|_{L^2}^{\frac12}
				\|\p_1(\p_2 u_1\wx^{-1})\|_{L^2}^{\frac12}\\
				&~\times \|\p_2 \p_1 u\wx^{-1}\|_{L^2}^{\frac12}
				\|\p_2^2 \p_1 u\wx^{-1}\|_{L^2}^{\frac12}\\
				&~+\|\phi \wx^2\|_{L^\infty}
				\|\phi \p_t  \p_2^2 u\|_{L^2}
				\|\p_2 u_2\wx^{-1}\|_{L^2}^{\frac12}
				\|\p_2 (\p_2 u_2\wx^{-1})\|_{L^2}^{\frac12}\\
				&~\times \|\p_2^2 u\wx^{-1}\|_{L^2}^{\frac12}
				\|\p_1(\p_2^2 u\wx^{-1})\|_{L^2}^{\frac12}\\
				\le &~\var\|\phi \p_t  \p_2^2 u\|_{L^2}^2
				+C_\var\|(\p_{12}u, \p_{122}u)\|_{L^2}^4
				\|\phi \wx^2\|_{H^2}^2.
			\end{aligned}
			\deqq
			Similarly, it is easy to check that
			\beqq
			\begin{aligned}
				|I_{13,4,3}|\lesssim
				&~\|\phi \p_t \p_2^2 u\|_{L^2}\|\nabla \p_2^2 u \wx^{-1}\|_{L^2}
				\|u \wx^{-1}\|_{L^\infty}\|\phi \wx^2\|_{L^\infty}\\
				\lesssim &~\|\phi \p_t \p_2^2 u\|_{L^2}
				(\|\p_1 \p_2^2 u\|_{L^2}+\|\p_1 \p_2^3 u \|_{L^2})
				\|(\p_1 u, \p_{12} u)\|_{L^2}
				\|(\phi, \p_1 \phi, \p_2 \phi, \p_{12}\phi)\wx\|_{L^2}\\
				\le &~\var\|\phi \p_t \p_2^2 u\|_{L^2}^2
				+C_\var \|\p_1 \p_2^3 u \|_{L^2}^2
				\|(\p_1 u, \p_{12} u)\|_{L^2}^2
				\|(\phi, \p_1 \phi, \p_2 \phi, \p_{12}\phi)\wx\|_{L^2}^2\\
				&~+C_\var \|(\p_1 u, \p_{12} u, \p_{122} u)\|_{L^2}^4 \|\phi \wx\|_{H^2}^2.
			\end{aligned}
			\deqq
			Thus, the combination of estimates from terms $I_{13,4,1}$
			to $I_{13,4,3}$ yields directly
			\beq\label{2808}
			|I_{13, 4}|
			\le
			\var\|\phi \p_t \p_2^2 u\|_{L^2}^2
			+C_\var \e^3
			+C_\var  \|\p_1 \p_2^3 u \|_{L^2}^2
			\|(\p_1 u, \p_{12} u)\|_{L^2}^2
			\|(\phi, \p_1 \phi, \p_2 \phi, \p_{12}\phi)\wx\|_{L^2}^2.
			\deq
			Substituting the estimates \eqref{2805}, \eqref{2806},
			\eqref{2807} and \eqref{2808} into \eqref{2809}, we have
			\beqq
			|I_{13}|
			\le \var\|\phi \p_t \p_2^2 u\|_{L^2}^2
			+C_\var (1+\e^4)
			+C_\var \|\p_1 \p_2^3 u \|_{L^2}^2
			\|(\p_1 u, \p_{12} u)\|_{L^2}^2
			\|(\phi, \p_1 \phi, \p_2 \phi, \p_{12}\phi)\wx\|_{L^2}^2.
			\deqq
			Therefore, we complete the proof of estimate \eqref{2810}.
		\end{proof}
		
		\begin{proof}[\textbf{Proof of estimate \eqref{210018}}]
			Due to the relation \eqref{210010}, we can obtain
			\beq\label{210015}
			\nabla \times \p_2(\phi^2 \p_t u )
			=2\p_2(\phi \p_1 \phi \p_t u_2)
			-2\p_2 (\phi \p_2 \phi \p_t u_1)
			+\p_2(\phi^2 \p_t \p_1 u_2)-\p_2(\phi^2 \p_t \p_2 u_1).
			\deq
			Using the Sobolev and Hardy inequalities, we can obtain
			\beq\label{210016}
			\begin{aligned}
				&\|\p_2(\phi \p_1 \phi \p_t u_2) \wx^2\|_{L^2}\\
				\lesssim&~\|\p_t u \wx^{-1}\|_{L^2}^{\frac12}\|\p_1(\p_t u \wx^{-1})\|_{L^2}^{\frac12}
				\|\p_1 \phi \wx\|_{L^2}^{\frac12}\|\p_2 \p_1 \phi \wx\|_{L^2}^{\frac12}
				\|\p_2 \phi \wx^2\|_{L^\infty}\\
				&~+\|\p_t u \wx^{-1}\|_{L^2}^{\frac12}\|\p_1(\p_t u \wx^{-1})\|_{L^2}^{\frac12}
				\|\p_{12} \phi \wx\|_{L^2}^{\frac12}\|\p_2 \p_{12} \phi \wx\|_{L^2}^{\frac12}
				\| \phi \wx^2\|_{L^\infty}\\
				&~+\|\phi \p_t \p_2 u\|_{L^2} \|\p_1 \phi \wx^2\|_{L^\infty}\\
				\lesssim&~\|\p_1 \p_t u\|_{L^2}\|\phi \wx^2\|_{H^3}^2
				+\|\phi \p_t \p_2 u\|_{L^2} \| \p_1 \phi \wx^2\|_{H^2}.
			\end{aligned}
			\deq
			Similarly, we can get
			\beq\label{210017}
			\begin{aligned}
				\|\p_2 (\phi \p_2 \phi \p_t u_1) \wx^2\|_{L^2}
				\lesssim &~\|\p_1 \p_t u\|_{L^2}\|\phi \wx^2\|_{H^3}^2
				+\|\phi \p_t \p_2 u\|_{L^2} \|\p_2 \phi \wx^2\|_{H^2};\\
				\|\p_2(\phi^2 \p_t \p_1 u_2) \wx^2\|_{L^2}
				\lesssim&~\|\p_t \p_{12} u_2\|_{L^2}\|\phi \wx\|_{L^\infty}^2
				+\|\p_t \p_{1} u_2\|_{L^2}
				\|\p_2 \phi \wx\|_{L^\infty}\|\phi \wx\|_{L^\infty}\\
				\le &~\var \|\p_t \p_{12} u_2\|_{L^2}^2
				+C_\var(\|\phi \wx\|_{H^2}^4
				+\|\p_t \p_{1} u_2\|_{L^2} \|\phi \wx\|_{H^3}^2);\\
				\|\p_2(\phi^2 \p_t \p_2 u_1) \wx^2\|_{L^2}
				\lesssim&~\|\phi \p_t \p_2^2 u_1\|_{L^2}
				\|\phi \wx^2\|_{L^\infty}
				+\|\phi \p_t \p_2 u_1\|_{L^2}
				\|\p_2 \phi \wx^2\|_{L^\infty}\\
				\le &~\var\|\phi \p_t \p_2^2 u_1\|_{L^2}^2
				+C_\var(\|\phi \wx^2\|_{H^2}^2
				+\|\phi \p_t \p_2 u_1\|_{L^2}
				\|\p_2 \phi \wx^2\|_{H^2}).
			\end{aligned}
			\deq
			Due to the relation \eqref{210015}, we can
			apply the estimates \eqref{210016} and \eqref{210017} to obtain
			\beqq
			\|\nabla \times \p_2(\phi^2 \p_t u )\wx^2\|_{L^2}\\
			\le \sigma_1 \|\phi \p_t \p_2^2 u_1\|_{L^2}^2 +\sigma_2 \|\p_t \p_{12} u_2\|_{L^2}^2
			+C_{\sigma_1,\sigma_2}(1+\e^2).
			\deqq
			Therefore, we complete the proof of estimate \eqref{210018}.
		\end{proof}
		
		\begin{proof}[\textbf{Proof of estimate \eqref{210019}}]
			Due to the relation \eqref{210012}, it is easy to check that
			\beqq
			\begin{aligned}
				&\|\nabla \times \p_2(\phi^2 u \cdot \nabla u)\wx^2\|_{L^2}\\
				\lesssim&~\|\nabla^2(\phi^2)u \cdot \nabla u\wx^2\|_{L^2}
				+\|\nabla(\phi^2)|\nabla u|^2 \wx^2 \|_{L^2}
				+\|\phi \nabla \phi u \cdot \nabla(\nabla u) \wx^2\|_{L^2}\\
				&+\|\phi^2 \nabla u \nabla^2 u \wx^2\|_{L^2}
				+\|\phi^2 u \cdot \nabla \p_{12}u \wx^2\|_{L^2}
				+\|\phi^2 u \p_2^3 u_1 \wx^2\|_{L^2}\\
				\overset{\text{def}}{=} &~J_{1,1}+J_{1,2}+J_{1,3}+J_{1,4}+J_{1,5}+J_{1,6}.
			\end{aligned}
			\deqq
			Using the Sobolev and Hardy inequalities, we can get that
			\beqq
			\begin{aligned}
				|J_{1,1}|
				\lesssim&~(\|\nabla \phi \wx^2\|_{L^2}^{\frac12}
				\|\p_1 (\nabla \phi \wx^2)\|_{L^2}^{\frac12}
				\|\nabla \phi \wx^2\|_{L^\infty}
				+ \|\nabla^2 \phi \wx^2\|_{L^2}^{\frac12}
				\|\p_1 (\nabla^2 \phi \wx^2)\|_{L^2}^{\frac12}
				\|\phi \wx^2\|_{L^\infty})\\
				&~\times\|\nabla u \wx^{-1}\|_{L^2}^{\frac12}
				\|\p_2 \nabla u \wx^{-1}\|_{L^2}^{\frac12}
				\|u \wx^{-1}\|_{L^\infty}\\
				\lesssim&~\|(\p_1 u, \p_{12}u, \p_1 \p_2^2 u)\|_{L^2}^2
				\|\phi \wx^2\|_{H^3}^2;\\
				|J_{1,2}|
				\lesssim&~\|\nabla u \wx^{-1}\|_{L^2}
				\|\p_1(\nabla u \wx^{-1})\|_{L^2}^{\frac12}
				\|\p_2 \nabla u \wx^{-1}\|_{L^2}^{\frac12}
				\|\phi \wx^2\|_{L^\infty}
				\|\nabla \phi \wx^2\|_{L^\infty}\\
				\lesssim&~\|(\p_1 u, \p_1^2 u, \p_{12}u, \p_{1} \p_2^2 u)\|_{L^2}^2
				\|\phi \wx^2\|_{H^3}^2;\\
				|J_{1,3}|
				\lesssim&~\|\nabla^2 u \wx^{-1}\|_{L^2}
				\|u \wx^{-1}\|_{L^\infty}
				\|\nabla \phi \wx^2\|_{L^\infty}
				\|\phi \wx^2\|_{L^\infty}\\
				\lesssim&~\|(\p_1 u, \p_1^2 u, \p_{12}u, \p_{1} \p_2^2 u)\|_{L^2}^2
				\|\phi \wx^2\|_{H^3}^2;\\
				|J_{1,4}|
				\lesssim&~\|\p_1^2 u\|_{L^2}^{\frac12}
				\|\p_2 \p_1^2 u\|_{L^2}^{\frac12}
				\|\nabla u \wx^{-1}\|_{L^2}^{\frac12}
				\|\p_1(\nabla u \wx^{-1})\|_{L^2}^{\frac12}
				\|\phi \wx^2\|_{L^\infty}^2\\
				&~+\|\p_2 \nabla u \wx^{-1}\|_{L^2}^{\frac12}
				\|\p_1(\p_2 \nabla u \wx^{-1})\|_{L^2}^{\frac12}
				\|\nabla u \wx^{-1}\|_{L^2}^{\frac12}
				\|\p_2 \nabla u \wx^{-1}\|_{L^2}^{\frac12}
				\|\phi \wx^2\|_{L^\infty}^2\\
				\lesssim&~\|(\p_1 u, \p_1^2 u, \p_{12}u, \p_{1} \p_2^2 u, \p_{1}^2 \p_2 u)\|_{L^2}^2
				\|\phi \wx^2\|_{H^2}^2;\\
				|J_{1,5}|
				\lesssim&~\|\nabla \p_{12}u\|_{L^2}
				\|u \wx^{-1}\|_{L^\infty}
				\|\phi\wx^2\|_{L^\infty}^2\\
				\lesssim&~\|(\p_1 u, \p_{12}u, \p_1^2 \p_2 u, \p_1 \p_2^2 u)\|_{L^2}^2
				\|\phi \wx^2\|_{H^2}^2;\\
				|J_{1,6}|
				\lesssim&~\|\phi \p_2^3 u\|_{L^2}
				\|u \wx^{-1}\|_{L^\infty}
				\|\phi \wx^2\|_{L^\infty}^2
				\lesssim \|(\phi \p_2^3 u, \p_1 u, \p_{12}u)\|_{L^2}^2
				\|\phi \wx^2\|_{H^2}^2.
			\end{aligned}
			\deqq
			Using the estimates of terms from $J_{1,1}$ to $J_{1,6}$
			and estimates $\eqref{21001}_1$ and $\eqref{21001}_2$, we have
			\beqq
			\|\nabla \times \p_2(\phi^2 u \cdot \nabla u)\wx^2\|_{L^2}
			\lesssim 1+\e^9.
			\deqq
			Therefore, we complete the proof of estimate \eqref{210019}.
		\end{proof}
		
		\begin{proof}[\textbf{Proof of estimate \eqref{4402}}]
			Similar to the lemma \ref{lemma27}, we have
			\beqq
			\bal
			\f12 \frac{d}{dt}\|\phi \p_2 \p_t u\|_{L^2}^2
			+\|\p_{12} \p_tu\|_{L^2}^2
			\le &~ \var \|\p_{12} \p_t u\|_{L^2}^2
			+\sigma_3 \|\phi \p_t^2 u\|_{L^2}^2
			+C_{\var}(\eb \es + \es^3) \|(\p_1 u, \p_{12}u, \phi \p_t\p_2 u)\|_{L^2}^2\\
			& + \underbrace{2\int \phi(u\cdot \nabla \phi) \p_t u \cdot \p_2^2 \p_t u dx}_{K_1}
			-\underbrace{\int \p_2 \p_t (\phi^2 u\cdot \nabla u)\cdot \p_2 \p_t u dx}_{K_2}.
			\dal
			\deqq
			Integrating by part and using the equation $\eqref{1-ns}_1$ gives 
			\beqq
			\bal
			K_1 
			&=2 \int 
			\phi (u \cdot \nabla \phi)  \cdot |\p_2 \p_t u|^2 dx
			+  2 \int (\p_2 \phi ( u \cdot \nabla \phi) + \phi (\p_2 u \cdot \nabla \phi) + \phi  (u \cdot \nabla \p_2\phi) )    \p_t u \cdot \p_2 \p_t u dx\\
			&\overset{\text{def}}{=} K_{1,i} (i=1,\cdots,4). 
			\dal
			\deqq
			Using the Hardy inequality, we have
			\beqq
			\bal
			|K_{1,1}|&\lesssim \| \phi \p_2 \p_t u \|_{L^2} \|  \p_2 \p_t u \wx^{-1}\|_{L^2} \|  u \wx^{-1 }\|_{L^{\infty}}  \|  \nabla \phi  \wx^{2}\|_{L^{\infty}}\le \var \|\p_{12} \p_t u\|_{L^2}^2 + C_{\var} \es\eb  \|(\p_1 u, \p_{12} u)\|_{L^2}^2;\\ 
			|K_{1,2}|&\lesssim \|  \p_t u \wx^{-1}\|_{L^2} \|  \p_2 \p_t u \wx^{-1}\|_{L^2} \|  u \wx^{-1 }\|_{L^{\infty}}  \|  \nabla \phi  \wx^{2}\|_{L^{\infty}} \|  \p_2 \phi  \wx\|_{L^{\infty}} \\
			&\le   \var \|\p_{12} \p_t u\|_{L^2}^2 + C_{\var} \es\eb^2  \|(\p_1 u, \p_{12} u)\|_{L^2}^2;\\ 
			|K_{1,3}|&\lesssim \| \phi \p_t \p_2 u\|_{L^2} \|  \p_t u \wx^{-1}\|_{L^2}^{\f12} \|  \p_1 (\p_t u \wx^{-1})\|_{L^2}^{\f12}
			\| \p_2 u \wx^{-1 }\|_{L^{2}}^{\f12}  \| \p_2^2 u \wx^{-1 }\|_{L^{2}}^{\f12}  \|  \nabla \phi  \wx^{2}\|_{L^{\infty}}\\
			&\lesssim   \eb^{\f12} \es^{\f12} \|(\p_{12} u, \p_1 \p_2^2 u ,\phi \p_t \p_2 u) \|_{L^2}^2;\\ 
			|K_{1,4}|&\lesssim \| \phi \p_t \p_2 u\|_{L^2} \|  \p_t u \wx^{-1}\|_{L^2}^{\f12} \|  \p_1 (\p_t u \wx^{-1})\|_{L^2}^{\f12}
			\| \nabla \p_2 \phi  \wx^{2 }\|_{L^{2}}^{\f12}  \| \nabla \p_2^2 \phi  \wx^{2}\|_{L^{2}}^{\f12}  \|  u  \wx^{-1}\|_{L^{\infty}} \\
			&\lesssim   \eb^{\f12} \es^{\f12} \|(\p_{1} \p_t u, \p_1 u , \p_{12} u) \|_{L^2}^2,
			\dal
			\deqq
			which, together with the assumptions \eqref{as1} and \eqref{as2}, yields that 
			\beqq
			|K_1| \le
			\var \|\p_{12} \p_t u\|_{L^2}^2 + C_M  \|(\p_1 u, \p_{12} u, \p_{1} \p_t u, \p_1 \p_2^2 u ,\phi \p_t \p_2 u)\|_{L^2}^2. 
			\deqq
			Let us deal with the term $K_2$.
			\beqq
			\bal
			K_2=&-\int (\p_t \p_2(\phi^2) u \cdot \nabla u +  \p_2(\phi^2) \p_t u \cdot \nabla u +  \p_2(\phi^2) u \cdot \p_t \nabla u 
			+ \p_t (\phi^2) \p_2 u \cdot \nabla u + \phi^2 \p_t \p_2 u \cdot \nabla u+  \phi^2 \p_2 u \cdot \nabla \p_t u\\
			&+\p_t (\phi^2) u \cdot \nabla \p_2 u + \phi^2 \p_t u \cdot \nabla \p_2 u +\phi^2 u \cdot \nabla \p_t \p_2 u) \cdot \p_t\p_2 u dx\\
			\overset{\text{def}}{=}&~ K_{2,i}(i=1,\cdots,9).
			\dal
			\deqq
			Using the Hardy inequality, we have
			\beqq
			\bal
			|K_{2,1}| \lesssim &~ \|  \p_t \p_2 u \wx^{-1} \|_{L^2}  \| \nabla u \wx^{-1}\|_{L^2}  \|  u  \wx^{-1}\|_{L^{\infty}}^2	 \| \nabla \phi  \wx^{2 }\|_{L^{\infty}}   \| \p_2 \phi  \wx^{2}\|_{L^{\infty}} \\
			& +   \|  \p_t \p_2 u \wx^{-1}\|_{L^2}  
			\|  \nabla u \wx^{-1} \|_{L^2}^{\f12} \|  \p_2 \nabla u \wx^{-1} \|_{L^2}^{\f12} \|   
			\|  \p_2 u \wx^{-1}\|_{L^2}^{\f12} \|  \p_1(\p_2 u \wx^{-1})\|_{L^2}^{\f12}\\
			& \times \|  u  \wx^{-1}\|_{L^{\infty}}	 \| \nabla \phi  \wx^{2 }\|_{L^{\infty}}   \|  \phi  \wx^{2}\|_{L^{\infty}} \\
			&+   \|  \p_t \p_2 u \wx^{-1}\|_{L^2} 
			\|  \nabla u \wx^{-1} \|_{L^2}^{\f12} \|  \p_2 \nabla u \wx^{-1} \|_{L^2}^{\f12} \|
			\| \nabla \p_2 \phi  \wx^{2 }\|_{L^{2}}^{\f12}   \| \p_1 (\nabla \p_2 \phi  \wx^{2})\|_{L^{2}}^{\f12}\\
			& \times \|  u  \wx^{-1}\|_{L^{\infty}}^2	    \|  \phi  \wx^{2}\|_{L^{\infty}} \\
			\le &~ \var \|\p_{12} \p_t u\|_{L^2}^2 + C_{\var} \es^2 \eb^2  \|(\p_1 u, \p_{12} u)\|_{L^2}^2.
			\dal
			\deqq 
			Similarly, we can obtain that
			\begin{align*}
			|K_{2,2}| &\lesssim  \| \phi \p_t \p_2 u  \|_{L^2}  \| \p_t u \wx^{-1}\|_{L^2}^{\f12} \| \p_1 (\p_t u \wx^{-1})\|_{L^2}^{\f12}
			\| \nabla u \wx^{-1}\|_{L^2}^{\f12} \| \p_2 \nabla u \wx^{-1}\|_{L^2}^{\f12} \| \p_2 \phi  \wx^{2}\|_{L^{\infty}}\\
			&\lesssim  \es^{\f12} \eb^{\f12}  \|(\phi \p_t \p_2 u, \p_{1} \p_t u)\|_{L^2}^2;\\
			|K_{2,3}| &\lesssim  \| \phi \p_t \p_2 u  \|_{L^2}  \| \nabla \p_t u \wx^{-1}\|_{L^2}
			\| 	
			u \wx^{-1}\|_{L^{\infty}}  \| \p_2 \phi  \wx^{2}\|_{L^{\infty}}\\
			&\le  \var \|\p_{12} \p_t u\|_{L^2}^2 + C_{\var} \eb(\es^{\f12} +\es)  \|(\p_1 u, \p_{12} u, \p_1 \p_t u)\|_{L^2}^2;\\
			|K_{2,4}| &\lesssim  \| \phi \p_t \p_2 u  \|_{L^2}   \| \p_2 u \wx^{-1}\|_{L^2}^{\f12} \| \p_1 (\p_2 u \wx^{-1})\|_{L^2}^{\f12}
			\| \p_1 u \|_{L^2}^{\f12} \| \p_2 \p_1  u\|_{L^2}^{\f12}
			\| u  \wx^{-1}\|_{L^{\infty}}
			\| \nabla \phi  \wx^{2}\|_{L^{\infty}}\\
			&\lesssim  \es \eb^{\f12}  \|( \p_1 u, \p_{12} u)\|_{L^2}^2;\\
			|K_{2,5}| &\lesssim  \| \phi \p_t \p_2 u  \|_{L^2}   \| \p_t \p_2 u \wx^{-1}\|_{L^2}^{\f12} \| \p_1 (\p_t \p_2 u \wx^{-1})\|_{L^2}^{\f12}
			\| \nabla u \wx^{-1} \|_{L^2}^{\f12} \| \p_2 \nabla u \wx^{-1}\|_{L^2}^{\f12}
			\| \phi  \wx^{2}\|_{L^{\infty}}\\
			&\le  \var \|\p_{12} \p_t u\|_{L^2}^2 + C_{\var} \eb \es \| \phi \p_t \p_2 u\|_{L^2}^2;\\
			\end{align*}
			and
			\beqq
			\bal
			|K_{2,6}| \lesssim &~ \| \phi \p_t \p_2 u  \|_{L^2}   \| \p_t \p_2 u \wx^{-1}\|_{L^2}^{\f12} \| \p_1 (\p_t \p_2 u \wx^{-1})\|_{L^2}^{\f12}
			\| \p_1 u_1  \|_{L^2}^{\f12} \| \p_2 \p_1 u_1 \|_{L^2}^{\f12}
			\| \phi  \wx\|_{L^{\infty}}\\
			& + \| \phi \p_t \p_2 u  \|_{L^2}   \| \p_t \p_1 u \wx^{-1}\|_{L^2}^{\f12} \| \p_2 \p_t \p_{1} u \wx^{-1})\|_{L^2}^{\f12}
			\| \p_2 u_1 \wx^{-1} \|_{L^2}^{\f12} \| \p_1 (\p_2 u_1 \wx^{-1})\|_{L^2}^{\f12}
			\| \phi  \wx^2\|_{L^{\infty}}\\
			\le &~ \var \|\p_{12} \p_t u\|_{L^2}^2 + C_{\var} (\eb^2 + \es \eb+\es) \| (\p_1 u, \p_{12} u, \p_1 \p_t  u)\|_{L^2}^2;\\
			|K_{2,7}| \lesssim &~ \| \phi \p_t \p_2 u  \|_{L^2}   \| \p_{12} u \|_{L^2} 
			\| u  \wx^{-1}\|_{L^{\infty}}^2 \| \nabla \phi  \wx^2\|_{L^{\infty}}\\
			& + \| \phi \p_t \p_2 u  \|_{L^2}   \|  \p_2^2 u \wx^{-1}\|_{L^2}^{\f12} \| \p_1 (\p_2^2 u  \wx^{-1})\|_{L^2}^{\f12}
			\| \nabla \phi \wx^{3} \|_{L^2}^{\f12} \| \p_2 \nabla \phi \wx^{3}\|_{L^2}^{\f12}
			\| u  \wx^{-1}\|_{L^{\infty}}^2\\	
			\lesssim &~ \es \eb^{\f12}  \|( \p_1 u, \p_{12} u)\|_{L^2}^2;\\
			|K_{2,8}| \lesssim &~ \| \phi \p_t \p_2 u  \|_{L^2}   \| \p_t u \wx^{-1}\|_{L^2}^{\f12} \| \p_2 (\p_t u \wx^{-1})\|_{L^2}^{\f12}
			\| \p_2^2 u \wx^{-1}  \|_{L^2}^{\f12} \| \p_1 (\p_2^2 u \wx^{-1})\|_{L^2}^{\f12}
			\| \phi  \wx^2 \|_{L^{\infty}}\\
			& + \| \phi \p_t \p_2 u  \|_{L^2}   \| \p_t  u \wx^{-1}\|_{L^2}^{\f12} \| \p_1 (\p_t  u \wx^{-1})\|_{L^2}^{\f12}
			\| \p_{12} u  \|_{L^2}^{\f12} \| \p_2 \p_{12} u \|_{L^2}^{\f12}
			\| \phi  \wx\|_{L^{\infty}}\\	
			\le &~ \var \|\p_{12} \p_t u\|_{L^2}^2 + C_{\var} (\es^2 + \es^{\f12} \eb^{\f12}+\eb) \| (\phi \p_t \p_2 u, \p_{1} \p_t u)\|_{L^2}^2.\\
			\dal
			\deqq
			Integrating by part and using the Hardy inequality, we have
			\beqq
			\bal
			|K_{2,9}| =&~  |\int \phi^2 (u_1 \cdot  \p_t \p_{12} u + u_2 \cdot  \p_t \p_{2}^2 u) \cdot \p_t\p_2 u dx|\\
			=&~|\int \phi^2 u_1 \cdot  \p_t \p_{12} u \cdot \p_t\p_2 u dx- \f12 \int  \p_2(\phi^2) u_2 |\p_t\p_2 u|^2 + \phi^2 \p_2 u_2 |\p_t\p_2 u|^2 dx|\\
			\lesssim &~ \| \phi \p_t \p_2 u  \|_{L^2}  \| \p_{12} \p_t u \|_{L^2}
			\| u \wx^{-1}\|_{L^{\infty}}  \|  \phi  \wx\|_{L^{\infty}}
			+ \| \phi \p_t \p_2 u  \|_{L^2}  \| \p_{2} \p_t u \wx^{-1}\|_{L^2}
			\| u \wx^{-1}\|_{L^{\infty}}  \| \p_2 \phi  \wx^2\|_{L^{\infty}}\\
			&+\| \phi \p_t \p_2 u  \|_{L^2}   \| \p_t \p_2 u \wx^{-1}\|_{L^2}^{\f12} \| \p_1 (\p_t \p_2 u \wx^{-1})\|_{L^2}^{\f12}
			\| \p_1 u_1  \|_{L^2}^{\f12} \| \p_2 \p_1 u_1 \|_{L^2}^{\f12}
			\| \phi  \wx\|_{L^{\infty}}\\
			\le &~ \var \|\p_{12} \p_t u\|_{L^2}^2 + C_{\var} \es^2 \eb  \|(\p_1 u, \p_{12} u)\|_{L^2}^2.
			\dal
			\deqq
			Combining the estimates from $K_{2,1}$ to $K_{2,9}$, we can obtain that
			\beqq
			|K_2| \le
			\var \|\p_{12} \p_t u\|_{L^2}^2 + C_M  \|(\p_1 u, \p_{12} u, \p_{1} \p_t u, \phi \p_t \p_2 u)\|_{L^2}^2, 
			\deqq
			which, together with the estimate of $K_1$, yields directly 
			\beqq
			\bal
			\f12 \frac{d}{dt}\|\phi \p_2 \p_t u\|_{L^2}^2
			+\|\p_{12} \p_tu\|_{L^2}^2
			\le  \var \|\p_{12} \p_t u\|_{L^2}^2
			+\sigma_3 \|\phi \p_t^2 u\|_{L^2}^2
			+C_M \|(\p_1 u, \p_{12}u, \p_{1} \p_2^2 u, \p_{1}\p_t u, \phi \p_t\p_2 u)\|_{L^2}^2.\\
			\dal
			\deqq
			Choosing $\var$ small enough, we complete the proof of estimate \eqref{4402}.
		\end{proof}
		
		\section{Control of initial data}\label{initial-control-A}
		\quad
		In this section, we will give the proof for the estimate for the control of initial data
		\eqref{initial-control-1} and \eqref{4105}.
		\begin{proof}[\textbf{Proof of estimate \eqref{initial-control-1}}]
			Indeed, in order to establish the estimate \eqref{initial-control-1},
			one only need to prove the following estimate:
			\beq\label{a1}
			\|\phi_0 (\p_t u)|_{t=0}\|_{L^2}
			+\|\phi_0 (\p_t \p_2 u)|_{t=0}\|_{L^2}
			+\|(\p_t \p_1 u)|_{t=0}\|_{L^2}
			\le C_{\phi_0, u_0, p_0, \gamma}.
			\deq
			First of all, since the initial data satisfies the
			compatible condition in Proposition \ref{pro-priori},
			then we have
			\beq\label{a2}
			\phi_0^2 (\p_t u)|_{t=0}+\phi_0^2 u_0 \cdot \nabla u_0
			-\p_1^2 u_0+\nabla p_0=0.
			\deq
			Then, the initial condition assumption \eqref{c-d} implies directly
			\beq\label{a3}
			\|\phi_0(\p_t u)|_{t=0} \|_{L^2}
			\le  \|\frac{\phi_0^2 u_0\cdot \nabla u_0-\p_1^2 u_0+\nabla p_0}
			{\phi_0}\|_{L^2}
			\le C_{\phi_0, u_0, p_0, \gamma},
			\deq
			and
			\beq\label{a4}
			\|(\p_1 \p_t u)|_{t=0} \|_{L^2}
			\le  \|\p_1\left\{\frac{\phi_0^2 u_0\cdot \nabla u_0-\p_1^2 u_0+\nabla p_0}
			{\phi_0^2}\right\}\|_{L^2}
			\le C_{\phi_0, u_0, p_0, \gamma}.
			\deq
			Furthermore, the Eq.\eqref{a2} implies directly
			\beq\label{a5}
			\phi_0 (\p_2 \p_t u)|_{t=0}
			=-\p_2 \phi_0 (\p_t u)|_{t=0}
			-\p_2 \left\{\frac{\phi_0^2 u_0 \cdot \nabla u_0-\p_1^2 u_0+\nabla p_0}{\phi_0}\right\}.
			\deq
			Using the anisotropic Sobolev inequality, we have
			\beqq
			\begin{aligned}
				\|\p_2 \phi_0 \p_t u|_{t=0}\|_{L^2}
				&\le \|\p_2 \ln \phi_0\|_{L^2}^{\frac12}
				\|\p_1 \p_2 \ln \phi_0\|_{L^2}^{\frac12}
				\|\phi_0 \p_t u|_{t=0}\|_{L^2}^{\frac12}
				\|\p_2(\phi_0 \p_t u|_{t=0})\|_{L^2}^{\frac12}\\
				&\le  \frac12\|\p_2 \phi_0 \p_t u|_{t=0}\|_{L^2}
				+\frac12\|\phi_0 \p_2\p_t u|_{t=0}\|_{L^2}
				+
				C\|(\p_2 \ln \phi_0, \p_1 \p_2 \ln \phi_0)\|_{L^2}^2
				\|\phi_0 \p_t u|_{t=0}\|_{L^2},
			\end{aligned}
			\deqq
			which, together with Eq.\eqref{a5}, yields directly
			\beq\label{a6}
			\begin{aligned}
				\|\phi_0 \p_2 \p_t u|_{t=0}\|_{L^2}
				&\le C\|(\p_2 \ln \phi_0, \p_1 \p_2 \ln \phi_0)\|_{L^2}^2
				\|\phi_0 \p_t u|_{t=0}\|_{L^2}
				+\|\p_2 \left\{\frac{\phi_0^2 u_0 \cdot \nabla u_0
					-\p_1^2 u_0+\nabla p_0}{\phi_0}\right\}\|_{L^2}\\
				&\le C_{\phi_0, u_0, p_0, \gamma},
			\end{aligned}
			\deq
			where we have used the estimate \eqref{a3} and the initial assumption \eqref{c-d}.
			The combination of estimates \eqref{a3}, \eqref{a4}
			and \eqref{a6} yields \eqref{a1} directly.
			Therefore, we complete the proof of the claimed estimate \eqref{initial-control-1}.
		\end{proof}
		
		\begin{proof}[\textbf{Proof of estimate \eqref{4105}}]
			Indeed, in order to establish the estimate \eqref{4105},
			one only need to prove the following estimate:
			\beq\label{a7}
			\|(\phi_0+\kappa) (\p_t u^\kappa)|_{t=0}\|_{L^2}
			+\|(\phi_0+\kappa)(\p_2 \p_t u^\kappa)|_{t=0}\|_{L^2}
			+\|(\p_1 \p_t u^\kappa)|_{t=0}\|_{L^2}
			\le C_{\phi_0, u_0, p_0, \gamma}.
			\deq
			Due to the assumption of the compatible condition, we have
			\beq\label{a8}
			(\phi_0+\kappa)^2 (\p_t u^\kappa)|_{t=0}
			+\phi_0^2 u_0 \cdot \nabla u_0
			-\p_1^2 u_0+\nabla p_0=0,
			\deq
			which yields directly
			\beqq
			(\phi_0+\kappa) (\p_t u^\kappa)|_{t=0}
			=-\frac{\phi_0^2 u_0 \cdot \nabla u_0-\p_1^2 u_0+\nabla p_0}{\phi_0+\kappa}.
			\deqq
			Thus, we have
			\beq\label{a9}
			\|(\phi_0+\kappa) (\p_t u^\kappa)|_{t=0}\|_{L^2}
			\le \|\frac{\phi_0^2 u_0 \cdot \nabla u_0-\p_1^2 u_0+\nabla p_0}{\phi_0}\|_{L^2}
			\le C_{\phi_0, u_0, p_0, \gamma}.
			\deq
			Similarly, due to the relation
			$|\p_1\left\{\frac{\phi_0^2}{(\phi_0+\kappa)^2}\right\}|
			\le |\p_1 \ln \phi_0|$ and anisotropic Sobolev inequality, one arrives at
			\beq\label{a10}
			\begin{aligned}
				&\|(\p_1\p_t u^\kappa)|_{t=0}\|_{L^2}\\
				=&\|\p_1\left\{\frac{\phi_0^2 u_0 \cdot \nabla u_0-\p_1^2 u_0+\nabla p_0}
				{(\phi_0+\kappa)^2}\right\}\|_{L^2}\\
				\le& \|\p_1\left\{\frac{\phi_0^2}{(\phi_0+\kappa)^2}\right\}
				\frac{\phi_0^2 u_0 \cdot \nabla u_0-\p_1^2 u_0+\nabla p_0}
				{\phi_0^2}\|_{L^2}
				+\|\frac{\phi_0^2}{(\phi_0+\kappa)^2}
				\p_1\left\{\frac{\phi_0^2 u_0 \cdot \nabla u_0-\p_1^2 u_0+\nabla p_0}
				{\phi_0^2}\right\}\|_{L^2}\\
				\le& \|\p_1 \ln \phi_0\|_{L^2}^{\frac12}
				\|\p_2 \p_1 \ln \phi_0\|_{L^2}^{\frac12}
				\|\frac{\phi_0^2 u_0 \cdot \nabla u_0-\p_1^2 u_0+\nabla p_0}
				{\phi_0^2}\|_{L^2}^{\frac12}
				\|\p_1\left\{\frac{\phi_0^2 u_0 \cdot \nabla u_0-\p_1^2 u_0+\nabla p_0}
				{\phi_0^2}\right\}\|_{L^2}^{\frac12}\\
				&\quad +\|\p_1\left\{\frac{\phi_0^2 u_0 \cdot \nabla u_0-\p_1^2 u_0+\nabla p_0}
				{\phi_0^2}\right\}\|_{L^2}\\
				\le& C_{\phi_0, u_0, p_0, \gamma}.
			\end{aligned}
			\deq
			Finally, the Eq.\eqref{a8} yields directly
			\beqq
			(\phi_0+\kappa)(\p_2 \p_t u^\kappa)|_{t=0}
			=-\p_2 \phi_0(\p_t u^\kappa)|_{t=0}
			-\p_2\left\{\frac{\phi_0^2 u_0 \cdot \nabla u_0-\p_1^2 u_0+\nabla p_0}
			{\phi_0+\kappa}\right\},
			\deqq
			which implies
			\beq\label{a11}
			\|(\phi_0+\kappa)(\p_2 \p_t u^\kappa)|_{t=0}\|_{L^2}
			\le \|\p_2 \phi_0(\p_t u^\kappa)|_{t=0}\|_{L^2}
			+\|\p_2\left\{\frac{\phi_0^2 u_0 \cdot \nabla u_0-\p_1^2 u_0+\nabla p_0}
			{\phi_0+\kappa}\right\}\|_{L^2}.
			\deq
			Using the anisotropic Sobolev inequality and Cauchy inequality, we have
			\beq\label{a12}
			\begin{aligned}
				&\|\p_2 \phi_0(\p_t u^\kappa)|_{t=0}\|_{L^2}\\
				\le &\|\p_2 \ln \phi_0\|_{L^2}^{\frac12}
				\|\p_1 \p_2 \ln \phi_0\|_{L^2}^{\frac12}
				\|\phi_0 \p_t u^\kappa|_{t=0}\|_{L^2}^{\frac12}
				\|\p_2(\phi_0 \p_t u^\kappa|_{t=0})\|_{L^2}^{\frac12}\\
				\le &\frac12\|\p_2 \phi_0 \p_t u^\kappa|_{t=0}\|_{L^2}
				+\frac12\|\phi_0 \p_2\p_t u^\kappa|_{t=0}\|_{L^2}
				+C\|(\p_2 \ln \phi_0, \p_1 \p_2 \ln \phi_0)\|_{L^2}^2
				\|(\phi_0+\kappa) \p_t u^\kappa|_{t=0}\|_{L^2}\\
				\le &\frac12\|\p_2 \phi_0 \p_t u^\kappa|_{t=0}\|_{L^2}
				+\frac12\|\phi_0 \p_2\p_t u^\kappa|_{t=0}\|_{L^2}
				+C_{\phi_0, u_0, p_0, \gamma},
			\end{aligned}
			\deq
			where we have used the estimate \eqref{a9} in the last inequality.
			Due to the fact
			$|\p_2\left\{\frac{\phi_0}{\phi_0+\kappa}\right\}|\le |\p_2 \ln \phi_0|$
			and anisotropic Sobolev inequality, we have
			\beq\label{a13}
			\begin{aligned}
				&\|\p_2\left\{\frac{\phi_0^2 u_0 \cdot \nabla u_0-\p_1^2 u_0+\nabla p_0}
				{\phi_0+\kappa}\right\}\|_{L^2}\\
				=&\|\p_2\left\{\frac{\phi_0}{\phi_0+\kappa}
				\frac{\phi_0^2 u_0 \cdot \nabla u_0-\p_1^2 u_0+\nabla p_0}{\phi_0}\right\}\|_{L^2}\\
				\le&\|\p_2\left\{\frac{\phi_0}{\phi_0+\kappa}\right\}
				\frac{\phi_0^2 u_0 \cdot \nabla u_0-\p_1^2 u_0+\nabla p_0}{\phi_0}\|_{L^2}
				+\|\frac{\phi_0}{\phi_0+\kappa}
				\p_2\left\{\frac{\phi_0^2 u_0 \cdot \nabla u_0-\p_1^2 u_0+\nabla p_0}{\phi_0}\right\}\|_{L^2}\\
				\le&  \|\p_2 \ln \phi_0\|_{L^2}^{\frac12}\|\p_{12} \ln \phi_0\|_{L^2}^{\frac12}
				\|\frac{\phi_0^2 u_0 \cdot \nabla u_0-\p_1^2 u_0+\nabla p_0}{\phi_0}\|_{L^2}^{\frac12}
				\|\p_2\left\{\frac{\phi_0^2 u_0 \cdot \nabla u_0-\p_1^2 u_0+\nabla p_0}
				{\phi_0}\right\}\|_{L^2}^{\frac12}\\
				&+\|\p_2\left\{\frac{\phi_0^2 u_0 \cdot \nabla u_0-\p_1^2 u_0+\nabla p_0}
				{\phi_0}\right\}\|_{L^2}\\
				\le&C_{\phi_0, u_0, p_0, \gamma}.
			\end{aligned}
			\deq
			Substituting the estimates \eqref{a12} and \eqref{a13} into \eqref{a11},
			one arrives at
			\beq\label{a14}
			\|(\phi_0+\kappa)(\p_2 \p_t u^\kappa)|_{t=0}\|_{L^2}
			\le C_{\phi_0, u_0, p_0, \gamma}.
			\deq
			The combination of estimates \eqref{a9}, \eqref{a10}
			and \eqref{a14} yields \eqref{a7} directly.
			Therefore, we complete the proof of the claimed estimate \eqref{4105}.
		\end{proof}
		
	\end{appendices}

	\section*{Data availability}
	No data was used for the research described in the article.
	
	\section*{Acknowledgments}
	
	This research was partially supported by
	National Key Research and Development Program of China(2021YFA1002100, 2020YFA0712500),
	Guangdong Basic and Applied Basic Research Foundation(2022A1515011798, 2021B1515310003)
	and National Natural Science Foundation of China(11971496, 12126609).

	\phantomsection
	
\end{document}